\documentclass[11pt,A4]{article}
\usepackage{amsmath}
\usepackage{amssymb}
\usepackage{amsfonts}
\usepackage{amsthm}
\usepackage{xcolor}
\usepackage{bm}
\usepackage{comment}
\usepackage{graphicx}
\usepackage{caption}
\usepackage{enumerate}
\usepackage{wrapfig}
\usepackage{subcaption} %

\graphicspath{{./}{Final paper plots/}}

\usepackage{tikz}
\usetikzlibrary{shapes.geometric, arrows}
\tikzstyle{bas} = [rectangle, minimum width=4cm, minimum height=1cm, text centered, text width=5cm ,draw=black, fill=blue!20]
\tikzstyle{arrow} = [thick,->,>=stealth]

\DeclareMathOperator\supp{supp}
\DeclareMathOperator\diam{diam}
\DeclareMathOperator\dist{dist}
\DeclareMathOperator\Hull{Hull}

\newcommand{\N}{\mathbb{N}}
\newcommand{\Z}{\mathbb{Z}}
\newcommand{\R}{\mathbb{R}}
\newcommand{\C}{\mathbb{C}}
\newcommand{\re}{\mathrm{e}}
\newcommand{\ri}{\mathrm{i}}
\newcommand{\rd}{\mathrm{d}}
\newcommand{\bc}{\mathbf{c}}
\newcommand{\bd}{\mathbf{d}}
\newcommand{\bg}{\mathbf{g}}
\newcommand{\bj}{\mathbf{j}}

\newcommand{\cA}{\mathcal{A}}
\newcommand{\cB}{\mathcal{B}}
\newcommand{\cT}{\mathcal{T}}

\newcommand{\cI}{\mathcal{I}}
\newcommand{\cV}{\mathcal{V}}
\newcommand{\cM}{\mathcal{M}}
\newcommand{\cZ}{\mathcal{Z}}

\newcommand{\eps}{\varepsilon}

\newcommand{\m}{\mathfrak{m}}
\newcommand{\jj}{\mathfrak{j}}
\newcommand{\dimH}{\mathrm{dim}_{\mathrm{H}}}
\DeclareRobustCommand
{\mathringbig}[1]{\accentset{\smash{\raisebox{-0.1ex}{$\scriptstyle\circ$}}}{#1}\rule{0pt}{2.3ex}}
\newcommand{\cirH}{\mathringbig{H}{}}

\usepackage{mathtools}

\usepackage{bigints}

\makeatletter
\def\Arg{\mathop{\operator@font Arg}\nolimits}
\def\cis{\mathop{\operator@font cis}\nolimits}
\makeatother 

\usepackage[mathscr]{euscript}

\newtheorem{theorem}{Theorem}[section]
\newtheorem{lemma}[theorem]{Lemma}
\newtheorem{proposition}[theorem]{Proposition}
\newtheorem{corollary}[theorem]{Corollary}
\newtheorem{definition}[theorem]{Definition}
\newtheorem{remark}[theorem]{Remark}

\newtheorem{assumption}[theorem]{Assumption}
\newtheorem{example}[theorem]{Example}

\usepackage{mathtools}

\makeatletter
\renewcommand*\env@matrix[1][\arraystretch]{%
  \edef\arraystretch{#1}%
  \hskip -\arraycolsep
  \let\@ifnextchar\new@ifnextchar
  \array{*\c@MaxMatrixCols c}}
\makeatother

\usepackage{accents}
\usepackage[a4paper, margin=1in]{geometry}  %

\usepackage{hyperref}

\title{Acoustic scattering by fractal inhomogeneities via geometry-conforming Galerkin methods for the Lippmann-Schwinger equation}
\author{J. Bannister$^{\text{a}}$, D. P. Hewett$^{\text{a}}$, A. Gibbs$^{\text{a}}$
\\
$^{\text{a}}${\footnotesize Department of Mathematics, University College London, London, United Kingdom}
}

\date{\today} %
\begin{document}
\maketitle
\abstract{ 
We propose and analyse a 
numerical method for 
time-harmonic acoustic scattering in $\R^n$, $n=2,3$, by a class of inhomogeneities (penetrable scatterers) 
with 
fractal boundary. Our 
method is 
based on a Galerkin discretisation of the Lippmann-Schwinger volume integral equation, using a discontinuous piecewise-polynomial approximation space 
on a geometry-conforming mesh comprising elements 
which themselves have 
fractal boundary. 
We first 
provide a semi-discrete well-posedness and error analysis 
for both the $h$- and $p$-versions of our method
for 
completely arbitrary inhomogeneities (without any regularity assumption on the boundary of the inhomogeneity or of the mesh elements). We prove convergence estimates for the integral equation solution and superconvergence estimates for linear functionals such as scattered field and far-field pattern evaluations, and elucidate how the regularity of the inhomogeneity boundary and the regularity of the refractive index affect the rates of convergence predicted. We then specialise to the case where the inhomogeneity is an ``$n$-attractor'', i.e.\ the fractal attractor of an iterated function system satisfying the open set condition with non-empty interior, showing how in this case the self-similarity of the inhomogeneity can be used to generate 
geometry-conforming meshes. 
For the 
$h$-version with piecewise constant approximation we also present singular quadrature rules, supported by a fully discrete error analysis, permitting practical implementation of our method. We present numerical results for two-dimensional examples, 
which validate our theoretical results and show 
that our method is significantly more accurate than a 
comparable method involving replacement of the fractal inhomogeneity by a smoother prefractal approximation.

} 



\section{Introduction}\label{s:intro}

The simulation of 
scattering by an \textit{inhomogeneity}, a compact region of space in which the 
refractive index differs from that of a homogeneous background medium, is a standard problem in computational wave propagation. 
Most previous 
studies (e.g.\ \cite{ck2,martin2003acoustic, vainikko2006multidimensional, 
vainikko2000fast, chen2002fast,bruno2005higher,hohage2006fast, duan2009high, 
anand2016efficient,
gopal2022accelerated,
gujjula2022new,
anderson2024fast}) assume that the boundary of the inhomogeneity is smooth, or at least piecewise smooth.   
However, in applications 
one often encounters 
inhomogeneities whose boundaries are highly non-smooth, 
with complicated multi-scale geometrical features that might be modelled more 
accurately by fractals. 
A specific example in atmospheric physics 
motivating our work 
is the scattering of electromagnetic radiation by ice crystals in clouds, an understanding of which is crucial both for the determination of cloud composition via imaging methods, and for the estimation of the Earth's radiation balance in climate modelling \cite{baran2009review}. 
Here a major computational challenge 
is 
the simulation of scattering by large ice crystal aggregates, which have been observed to exhibit highly complex fractal-like shapes \cite{tyynela2011radar,chukin2012two,stein2015fractal}. 

As a first step towards such applications, in this paper we propose and analyse a 
numerical method for time-harmonic acoustic scattering in $\R^n$, $n=2,3$, by inhomogeneities with fractal boundaries. 
The problem we consider is to find a scattered field $u^{\rm s}$ satisfying the outgoing Sommerfeld radiation condition and the inhomogeneous Helmholtz equation 
\begin{align}
\label{e:FSP0}
\Delta u^s + k^2(1+\m)u^s = -f
\end{align}
in $\R^n$, 
where 
$k>0$ is the wavenumber, $f$ is a compactly supported source term (associated with the incident wave), 
and $\m$ is the refractive index perturbation, which we assume to be bounded and 
supported in a 
compact set%
\begin{align}
\label{e:KDef}
K:=\supp\m,
\end{align}
that we refer to as the \textit{inhomogeneity}. 
This is a standard model for acoustic scattering by penetrable obstacles \cite[Chap.~8]{ck2}, which in two dimensions also models an analogous electromagnetic scattering problem for waves with a certain polarisation  \cite[Rem.~2.5]{moiola2019acoustic}. 

Our method is based on the
standard 
reformulation 
of \eqref{e:FSP0} 
as a volume integral equation, 
commonly referred to as a Lippmann-Schwinger equation (LSE) - see, e.g.,
\cite{martin2003acoustic,costabel2015spectrum} and \cite[Section 8.2]{ck2}, 
and equations \eqref{e:LSE}-\eqref{e:LSE2} below. 
In recent years 
a number of 
efficient LSE solvers have been developed (e.g.\ \cite{vainikko2000fast, chen2002fast,bruno2005higher,hohage2006fast,
duan2009high, 
anand2016efficient,
gopal2022accelerated,
gujjula2022new,
anderson2024fast}), 
based variously on Galerkin, collocation and Nystr\"om discretisations. 
However, these methods all assume that $\m$ is either globally smooth, or is discontinuous across a smooth or piecewise-smooth interface.  %
In contrast, 
our approach 
can efficiently handle a class of inhomogeneities with fractal boundary.

As an illustration of the type of problem we consider, in Figure \ref{f:Scatt} we show total field plots, computed using our method, for the scattering of an incident plane wave in $\R^2$ by 
the Fudgeflake, the Gosper Island, and the Koch Snowflake, each of which has a fractal boundary with Hausdorff dimension strictly between $1$ and $2$ (for details see \S\ref{s:Examp} and \S\ref{s:Numer}). In these plots $\m$ takes a constant non-zero value on the inhomogeneity $K$, meaning that the total refractive index $1+\m$ jumps discontinuously across a fractal curve.

\begin{figure}[t!]
\centering
\includegraphics[height=.3\textwidth]
{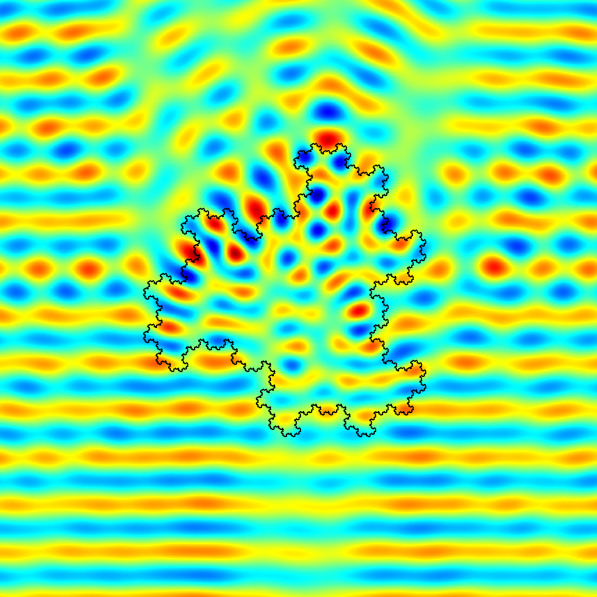}%
\hspace{2mm}
\includegraphics[height=.3\textwidth]
{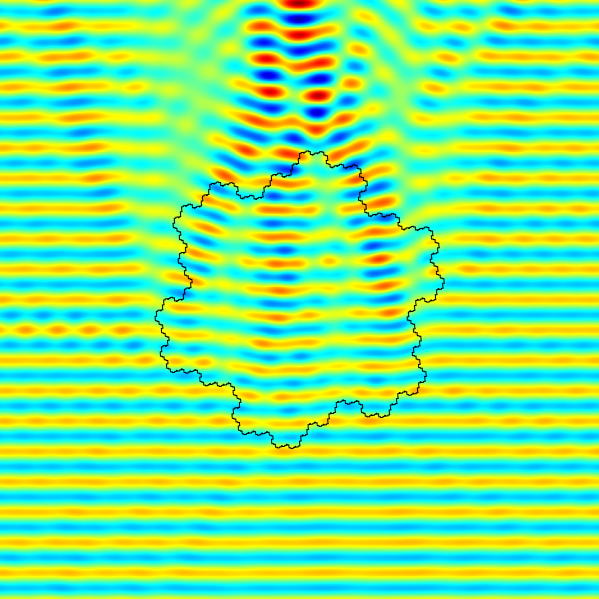}%
\hspace{2mm}
\includegraphics[height=.3\textwidth]
{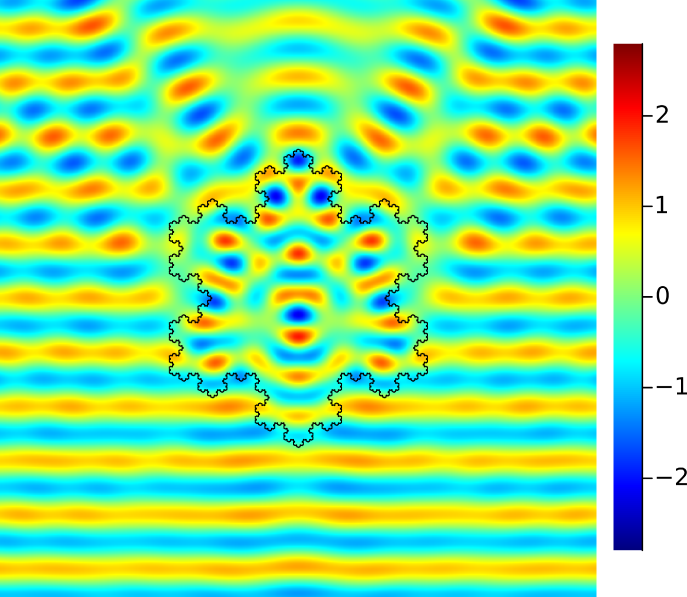}%
\caption{Scattering of a plane wave by the Fudgeflake, Gosper Island and Koch Snowflake, %
with a constant refractive index on each  inhomogeneity. For details see \S\ref{s:Numer}.}
\label{f:Scatt}
\end{figure}

Our method is based on a Galerkin discretisation of the LSE using a discontinuous piecewise-polynomial approximation space on a geometry-conforming mesh, 
such as those illustrated in Figure \ref{f:Meshes}. 
By the term ``geometry-conforming'' we mean that the fractal boundary of the inhomogeneity is captured exactly by the mesh, 
without any geometrical approximation. Necessarily this involves the use of mesh elements with fractal boundary. 

\begin{figure}[t!]
\subfloat[Uniform meshes of the Fudgeflake \label{f:FudgeMesh}]{
\centering
\includegraphics[width=.22\textwidth]{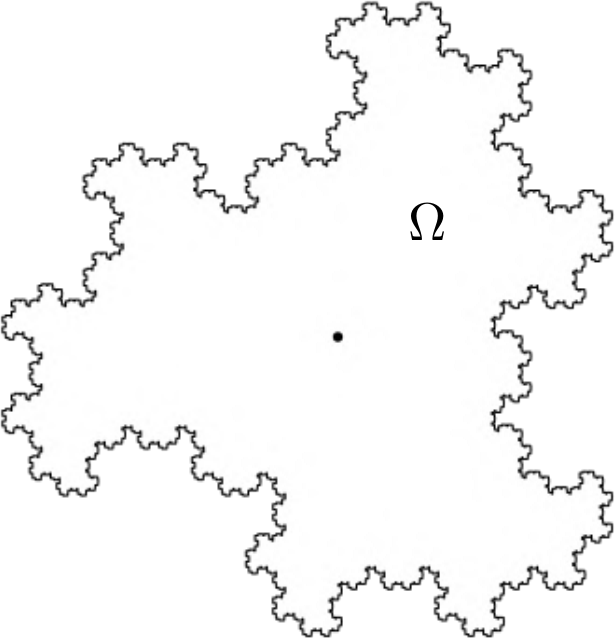}
\hspace{2mm}
\includegraphics[width=.22\textwidth]{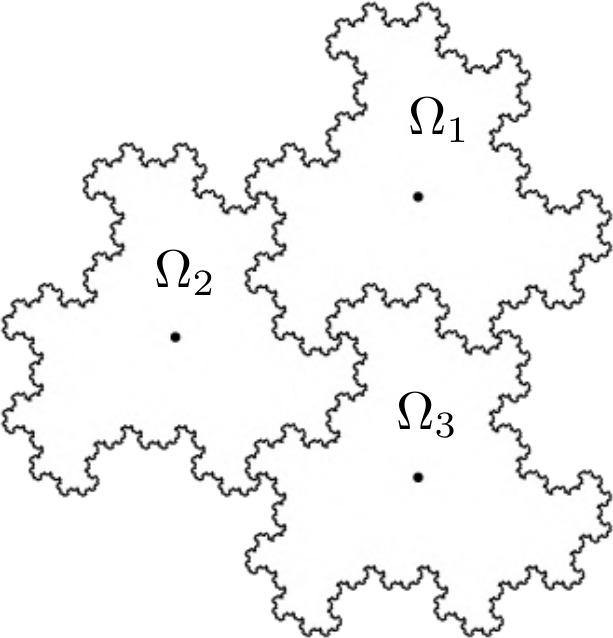}
\hspace{2mm}
\includegraphics[width=.22\textwidth]{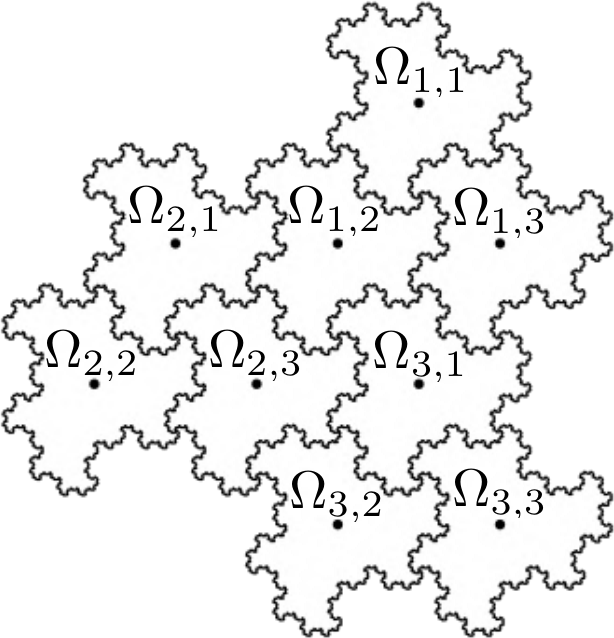}
\hspace{2mm}
\includegraphics[width=.22\textwidth]{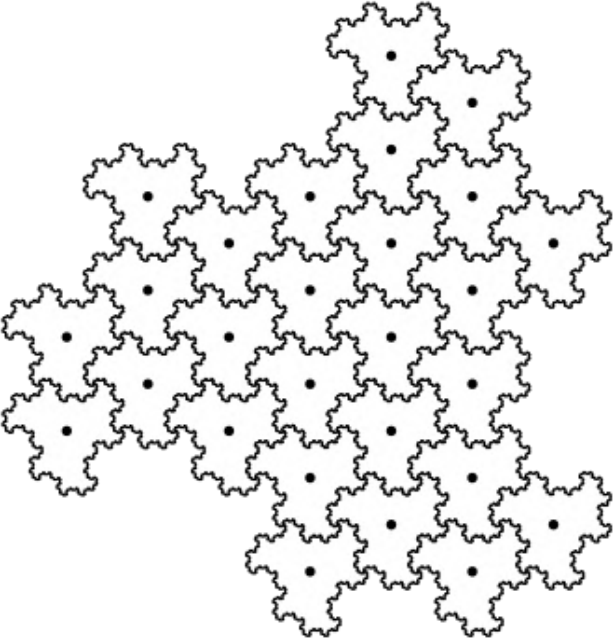}
}

\subfloat[Uniform meshes of the Gosper Island.\label{f:GosperMesh}]{
\centering
\includegraphics[width=.22\textwidth]{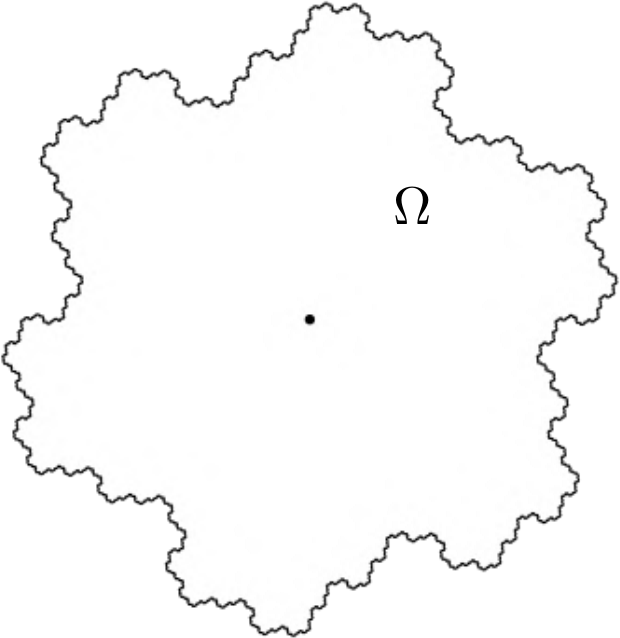}
\hspace{2mm}
\includegraphics[width=.22\textwidth]{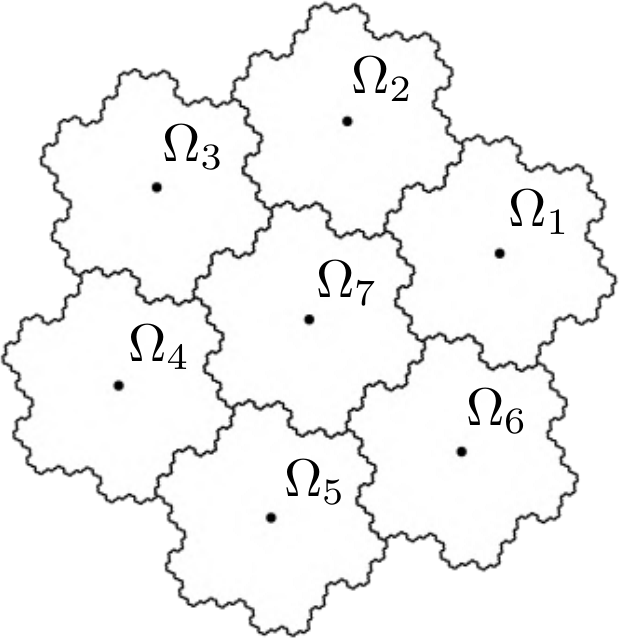}
\hspace{2mm}
\includegraphics[width=.22\textwidth]{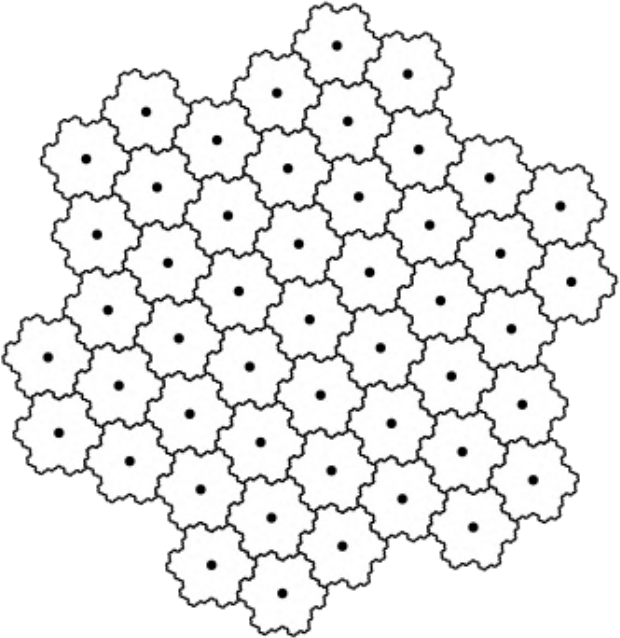}
\hspace{2mm}
\includegraphics[width=.22\textwidth]{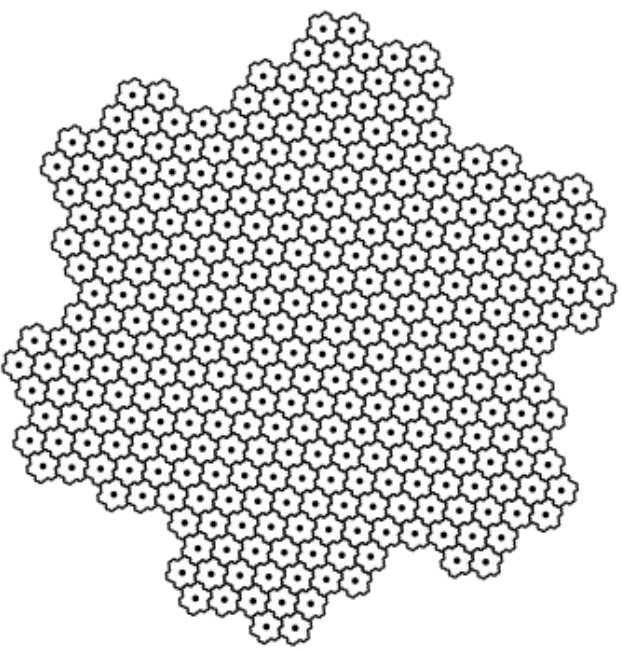}
}

\subfloat[Quasi-uniform meshes of the Koch Snowflake.\label{f:KochMesh}]{
\centering
\includegraphics[width=.22\textwidth]{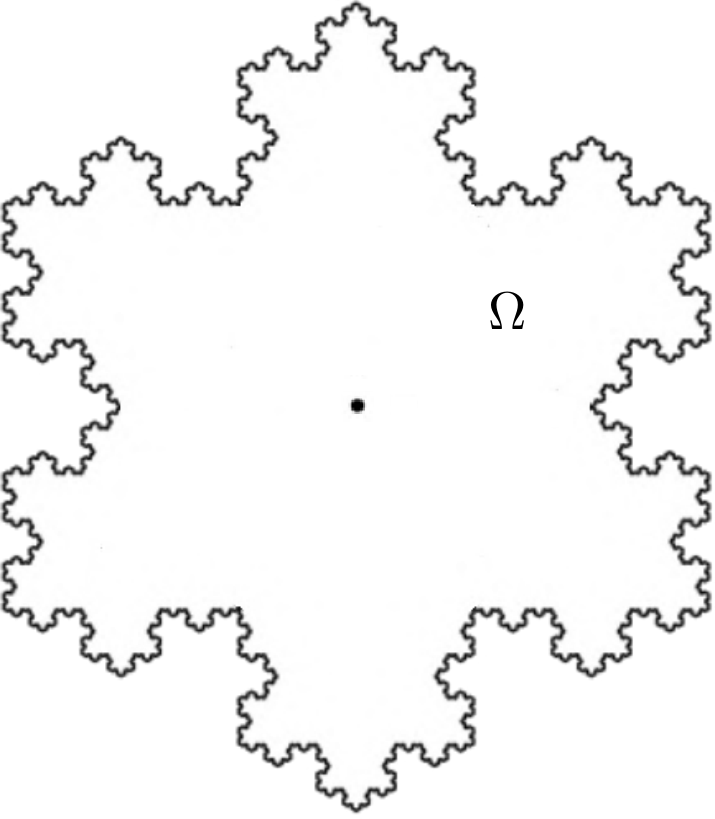}
\hspace{2mm}
\includegraphics[width=.22\textwidth]{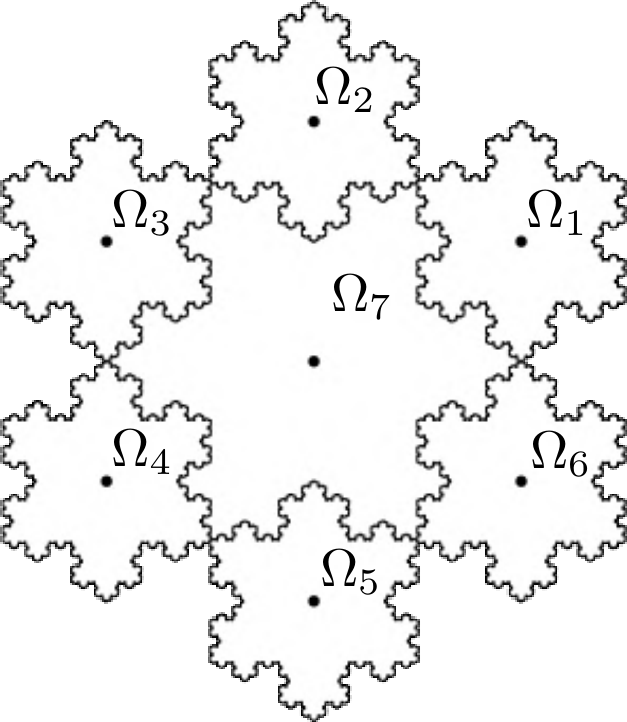}
\hspace{2mm}
\includegraphics[width=.22\textwidth]{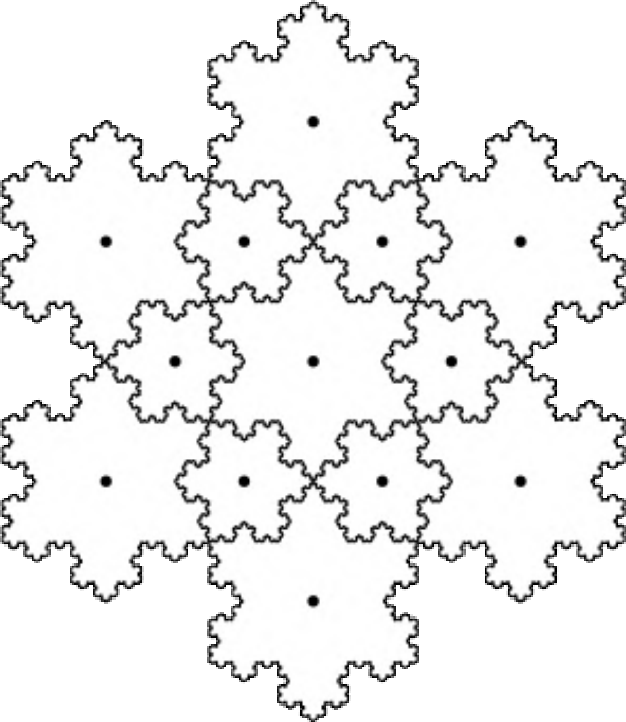}
\hspace{2mm}
\includegraphics[width=.22\textwidth]{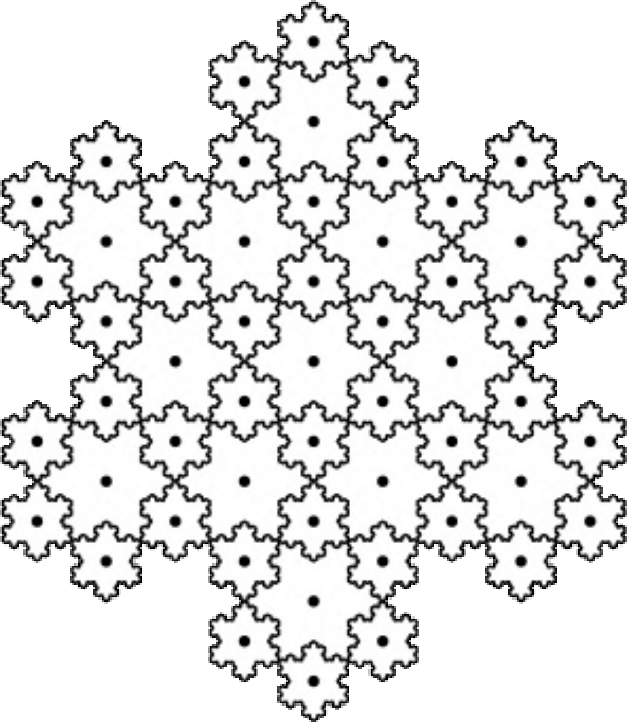}
}
\caption{Examples of geometry-conforming meshes of the Fudgeflake, Gosper Island and Koch Snowflake - for details see \S\ref{s:Examp}. In (a) and  (b) the meshes are uniform, because the associated IFS is homogeneous (all similarities have the same contraction factor), while in (c) the meshes are quasi-uniform.}
\label{f:Meshes}
\end{figure}

Our motivation for considering geometry-conforming meshes is to avoid the geometry-related approximation errors incurred by alternative approaches. Two such alternatives are illustrated in Figure \ref{f:Disc} for $K$ the Koch Snowflake. 
In Figure \ref{f:Disc}(a) we illustrate an ``unfitted'' approach, in which $K$ has been embedded in a bounding box, which has been meshed with a uniform Cartesian mesh. This approach 
is 
used in many volume integral equation solvers
(e.g., \cite{vainikko2000fast, gopal2022accelerated, ying2015sparsifying,
ambikasaran2016fast,
liu2018sparsify, 
tambova2018adiabatic,
giannakopoulos2019memory,
georgakis2020fast,
groth2020accel}) 
and the so-called discrete dipole approximation (DDA) method for the analogous electromagnetic problem \cite{goodman1991application, sarkar2003application, yurkin2007discrete}, 
because it offers the possibility of fast solvers that exploit FFT-based matrix-vector products.  
However, 
the fact that the mesh does not conform to $\partial K$ means that carrying out accurate quadrature on mesh elements intersected non-trivially by $\partial K$ may be challenging. This issue arises even when $\partial K$ is smooth, but is particularly problematic when $\partial K$ is fractal because the number of 
such mesh elements 
grows faster as $h\to 0$ than in the smooth case 
(cf., e.g., \cite[Eqn.~(3.7)]{bannisterthesis} for the Koch Snowflake case).
In Figure \ref{f:Disc}(b) we illustrate a ``semi-fitted'' approach in which $K$ has been approximated by a polygonal ``prefractal'' (the third in the sequence of standard prefractal approximations to $K$ shown in Figure \ref{f:KochPreFrac}), which has been meshed with triangular elements with a meshwidth $h$ comparable to the smallest side length in the prefractal. 
In certain cases (including that pictured in Figure \ref{f:Disc}(b)) it is still possible to develop FFT-based solvers on such meshes (see, e.g., \cite{bannister2022acoustic}). However, truncating the fractal geometry of $\partial K$ in this way will lead to solution errors that are worse than those arising for the approximation of a smooth domain by a polygon, again because of the fractality of $\partial K$, 
which implies in particular that the volume of the set difference between $K$ and its polygonal 
approximation typically decays 
more slowly as $h\to 0$ than in the smooth case 
(cf.~equation \eqref{e:GapArea} in Appendix \ref{a:AppB} below for the Koch Snowflake case). 
Finally, in Figure \ref{f:Disc}(c) for comparison we illustrate the ``fitted'', or ``geometry-conforming'', strategy that we adopt in this paper, in which $K$ is meshed exactly, with no approximation of 
$\partial K$, using a quasi-uniform mesh comprising elements with fractal boundary. This approach does not suffer from the geometrical approximation errors of the previous two approaches, and hence one might expect it to provide more accurate approximations. A key achievement of this paper is to show, both theoretically and numerically, that this is indeed the case. However, the reader might reasonably ask how such meshes can be generated in practice, and how one might go about doing accurate quadrature on them. 
A further key achievement of the paper is to show that, for a certain class of fractal inhomogeneities (including those considered in Figures \ref{f:Scatt} and \ref{f:Meshes}), such meshes can be constructed, and numerical quadrature on them implemented, in such a way that  
our method achieves essentially the same accuracy for a fractal inhomogeneity as one would obtain when discretising a polygonal/polyhedral inhomogeneity using a classical simplicial geometry-conforming mesh. 

\begin{figure}
\centering
\begin{subfigure}[t]{0.28\textwidth}
\centering
\includegraphics[height=40mm]{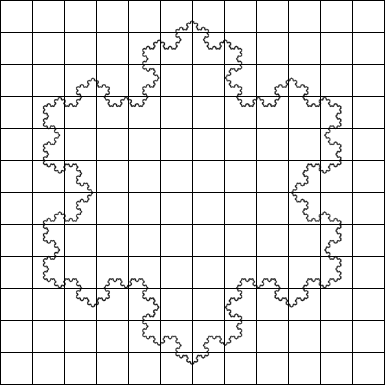}
\caption{Bounding box}
\end{subfigure}
\begin{subfigure}[t]{0.28\textwidth}
\centering
\includegraphics[height=40mm]{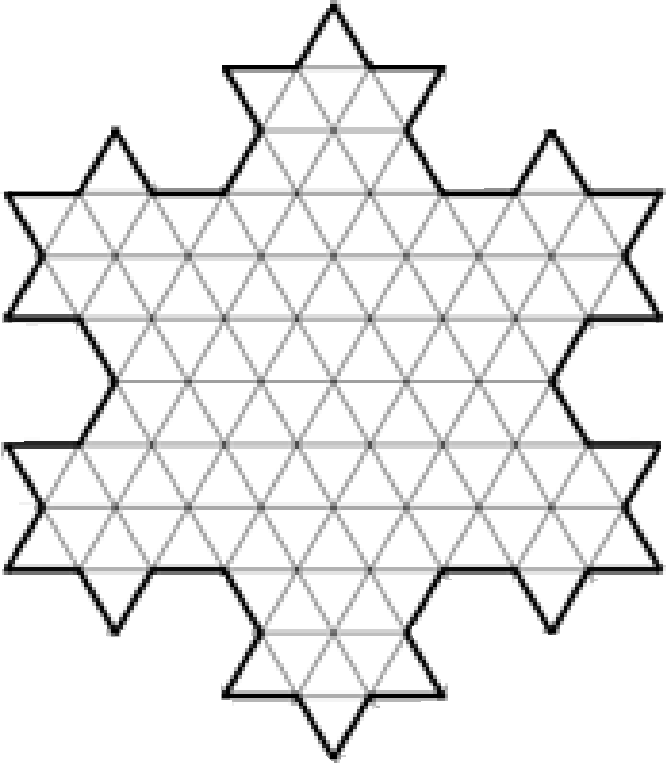}
\caption{Prefractal}
\end{subfigure}
\begin{subfigure}[t]{0.28\textwidth}
\centering
\includegraphics[height=40mm]{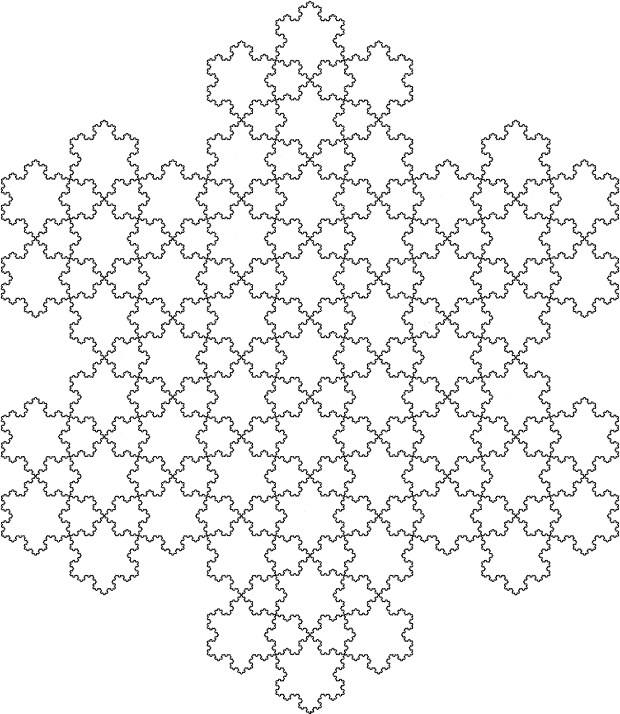}
\caption{Geometry-conforming}
\end{subfigure}
\caption{
Different discretisations of the Koch Snowflake. In this paper we study the 
``geometry-conforming'' approach of (c), using a mesh with fractal elements.}
\label{f:Disc}
\end{figure}

\begin{figure}[t!]
\centering
\includegraphics[width=.9\textwidth]{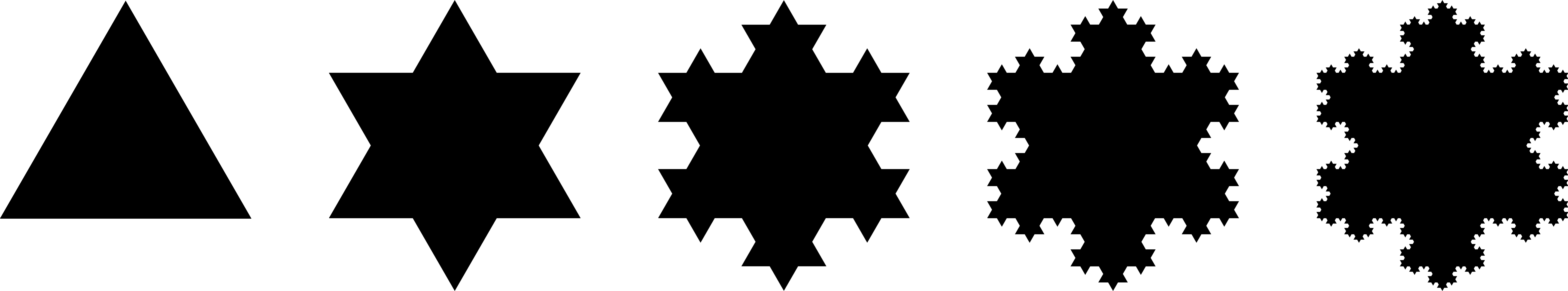}
\caption{The first five prefractal approximations of the Koch Snowflake.}
\label{f:KochPreFrac}
\end{figure}

An overview of the paper is as follows. 
In \S\ref{s:Formulation} we present the problem formulation, and address basic questions of well-posedness and solution regularity. 
In more detail, in \S\ref{s:FSP-LSE} we prove well-posedness of 
\eqref{e:FSP0} and 
review its equivalent reformulation as a well-posed Lippmann-Schwinger equation. This is classical for smooth $\m$ and/or smooth $K$, and our novel contribution here is to clarify that the classical results hold for essentially arbitrary $\m$ and $K$. In \S\ref{s:Restr} we prove the well-posedness of the restriction of the LSE to a bounded domain, and in \S\ref{s:SolReg} we clarify how the mapping properties of the associated integral operators depend on the regularity of $\m$ and $K$. In \S\ref{s:scatt} we discuss certain simplifications that arise when \eqref{e:FSP0} models a scattering problem. 

In \S\ref{s:VIM} we introduce our Galerkin method for the restricted LSE and analyse it at an abstract level. In \S\ref{s:Error} we provide a semi-discrete analysis of the method for completely arbitrary $K$, and completely arbitrary geometry-conforming meshes. 
We consider both $h$- and $p$-version convergence, 
using $h$- and $p$-explicit piecewise polynomial approximation results presented recently in \cite{polyapprox} 
to prove 
convergence estimates for the integral equation solution, and superconvergence estimates for evaluations of the scattered field and far-field pattern, with 
rates that depend explicitly on the regularity of $\m$ and $K$. 
In \S\ref{s:AbstFD} we provide a framework for a fully discrete analysis of our method, applicable whenever suitably accurate quadrature rules are available for its implementation. 

In \S\ref{s:IFS} we then specialise to the case where $K$
belongs to a certain class of compact sets 
that we call ``$n$-attractors'' (following \cite{dequalsnpaper}).  
{As detailed in Definition \ref{d:nattract} below, a compact set is an $n$-attractor if it has 
Hausdorff dimension $n$ 
and is the attractor of an iterated function system (IFS) satisfying the open set condition. The basic theory of IFS attractors was developed in \cite{HUTCHINSONJOHNE.1981FaSS} and numerous examples can be found at \cite{riddleIFS}, including the three examples shown in Figure \ref{f:Scatt}.} 
These sets typically have fractal boundaries, 
and the sets themselves possess a certain self-similarity, being equal to a finite union of smaller, similar copies of themselves. 
In \S\ref{s:IFSMesh} we show how this self-similarity permits the construction of quasi-uniform geometry-conforming meshes, {as illustrated in Figure \ref{f:Meshes}}, which completely capture the geometry of $K$ by using elements which themselves have a fractal boundary. 
In Remark \ref{r:App4} we observe that, when $\m$ is not globally smooth, our geometry-conforming approach offers higher rates of semi-discrete superconvergence for scattered field and far-field pattern evaluations than can be obtained using a ``bounding box'' approach like that illustrated in Figure \ref{f:Disc}(a). 
In \S\ref{s:IFSquad} we show how self-similarity  can also be further applied to derive accurate quadrature rules for singular integration on 
our geometry-conforming 
meshes, using the results presented recently in \cite{disjoint,nondisjoint}. For simplicity we restrict attention in this section to piecewise-constant approximations, and the case where $\m$ is constant on $K$. 
{In \S\ref{s:IFSFD}, 
we provide a fully-discrete error analysis for our method, with Theorem \ref{t:FD} (the main result of the paper) providing}  conditions on the quadrature parameters guaranteeing that the fully-discrete solution converges at the same rate as the semi-discrete solution. 

In \S\ref{s:Examp} we provide details of the three examples illustrated in Figure \ref{f:Scatt}, and in \S\ref{s:Numer} we report numerical results for these examples that validate our theoretical analysis and show that our method achieves a significantly smaller error than a comparable prefractal-based approach like that illustrated in Figure \ref{f:Disc}(b). In Appendix \ref{a:AppB} we provide a semi-discrete error analysis of the pre-fractal method, and in Appendix \ref{a:AppA} we collect some basic geometrical lemmas.

We end this introduction by 
noting that 
the analysis and numerical solution of PDEs on domains with fractal boundaries/interfaces is a highly active field, with many recent studies relating to prefractal-based approximations for the Poisson equation \cite{achdou2016transmission,vivaldi2007fractal,lancia2002transmission,capitanelli2015quasi,capitanelli2010robin,lancia2004some,lancia2003second}, Helmholtz equation \cite{bannister2022acoustic,
chandler2021boundary,
chandler2018well}, modified Helmholtz equation 
\cite{claret2024convergence}, 
heat equation \cite{lancia2010irregular,lancia2012numerical,cefalo2023fractal}, Stokes equation  \cite{cefalo2021approximation,lancia2020stokes} magnetostatic problems \cite{creo2020magnetostatic}, and for spectral problems for the Laplacian \cite{rosler2024computing,rosler2024res}.  
The geometry-conforming approach of the current paper builds on related work presented recently in \cite{caetano2024hausdorff,caetano2025integral,dequalsnpaper} 
for the solution of Dirichlet problems for the Helmholtz equation. 
We also mention the related recent study \cite{joly2024high} into high-order quadrature rules on fractals, which complements the results of \cite{disjoint,nondisjoint}, and which 
we expect to be relevant to the practical implementation of our method for higher polynomial degrees $(p\geq 1$). 
{Finally, we note that variational formulations of transmission problems in non-Lipschitz extension domains have been studied recently in \cite{claret2026helmholtz}, building on earlier related work in \cite{claret2024layer}}.

\section{Problem formulation, well-posedness and regularity}
\label{s:Formulation}

In this section we define the 
full-space problem (FSP) and 
demonstrate its well-posedness and equivalence to 
the Lippmann-Schwinger equation (LSE). 
These results are well known for a smooth or piecewise-smooth refractive index (see, e.g.,\ \cite[\S8]{ck2}), and our aim here is to clarify that they hold in fact for an arbitrary 
bounded refractive index supported in 
an arbitrary compact inhomogeneity. 
We then consider the restriction of the LSE to a bounded set, and study how the regularity of the refractive index and its support affect the mapping properties of the resulting integral operator and its adjoint. 

Throughout the paper, our standing assumptions are that %
\[n=2,3, \quad k>0, \quad f\in L^{2}_{\rm comp}(\R^n), \quad \text{and } \m \in L^{\infty}_{\rm comp}(\R^n).\]
We also recall from \eqref{e:KDef} the notation $K:=\supp\m$, where $\supp\m$ denotes the essential support of $\m$. Our well-posedness results will be stated under the additional {standard} assumption that $\Im[\m(x)] \geq 0$ 
for a.e.\ $x \in \R^n$, which for brevity we will write simply as $\Im[\m] \geq 0$. 
For a set $E\subset\R^n$ we denote by $E^\circ$ its interior, $\overline{E}$ its closure, $\partial E$ its boundary, and $|E|$ its Lebesgue measure (if $E$ is measurable). 

Our analysis is in the setting of fractional Sobolev spaces. 
For $s\in\R$ we denote by $H^s(\R^n)$ the usual Bessel potential Sobolev space, which coincides (with equivalent norms) with the classical Sobolev space $W^s_2(\R^n)$ defined in terms of square integrability of weak derivatives when $s\in \N_0$. In particular, $H^0(\R^n)=L^2(\R^n)$. We let $H^s_{\rm comp}(\R^n)$ denote the set of $u\in H^s(\R^n)$ with compact support, and $H^s_{\rm loc}(\R^n)$ the set of 
distributions $u$  
such that $\chi u\in H^s(\R^n)$ for every $\chi\in C^\infty_{\rm comp}(\R^n)$ (see e.g.\ \cite[\S2.6]{ss11} for topological properties of these spaces). 
For a non-empty open set $\Omega\subset\R^n$ and for $s\geq 0$ we define $H^s(\Omega):=\{u\in L^2(\Omega):u=U|_\Omega \text{ for some } U\in H^s(\R^n)\}$, which is a Hilbert space with the norm $\|u\|_{H^s(\Omega)}=
\min
\|U\|_{H^s(\R^n)}$, 
where the minimum is over all $U\in H^s(\R^n)$ such that $U|_\Omega=u$. 
When $\Omega$ is Lipschitz and $r\in \N_0$ the space $H^r(\Omega)$ coincides (with equivalent norms) with the space $W^r(\Omega)$ of functions in $L^2(\Omega)$ whose weak derivatives up to order $r$ are in $L^2(\Omega)$. However, for non-Lipschitz $\Omega$ the space $H^r(\Omega)$ may be a proper subset of $W^r(\Omega)$ \cite[\S1.1.1]{nonlip}. 

\subsection{The full-space problem and the Lippmann-Schwinger equation}\label{s:FSP-LSE}

\begin{definition}[Full-space problem (FSP)]\label{d:FSP}
Find $u^{\rm s} \in H^{1}_{\rm loc}(\R^n)$ satisfying the inhomogeneous Helmholtz equation
\begin{equation}\label{e:FSP}
\Delta u^{\rm s} + k^{2}(1+\m)u^{\rm s} = -f, \quad \mbox{in } \R^n,    
\end{equation}
and the Sommerfeld radiation condition (SRC)
\begin{equation}\label{e:SRC}
\frac{\partial u^{\rm s}}{\partial r}(x)-iku^{\rm s}(x)=o\bigg( \frac{1}{r^{(n-1)/2}} \bigg)	, \quad \mbox{uniformly as } r \coloneqq |x| \rightarrow \infty	.
\end{equation}
\end{definition}

\begin{remark}
\label{r:ER}
If $u^{\rm s}$ satisfies \eqref{e:FSP} then $u^{\rm s}\in H^2_{\rm loc}(\R^n)$ by standard elliptic regularity results (e.g., \cite[\S6.3.1, Thm 3]{evans}), and hence 
$u^{\rm s}\in C^0(\R^n)$ by the Sobolev embedding theorem (recall that $n=2,3$). Furthermore, if $R>0$ is such that the open ball $B_R$ centred at the origin contains both 
$K$ 
and $\supp{f}$, then $u^{\rm s}|_{\overline{B_R}^c}\in C^\infty(\overline{B_R}^c)$, so that the SRC \eqref{e:SRC} can be imposed {with $\partial u^{\rm s}/\partial r$ understood} in the classical sense.
\end{remark}

\begin{remark}[Transmission problem in the Lipschitz case]
Suppose that 
$K=
\overline{\Omega}$ 
for some bounded Lipschitz open set $\Omega\subset\R^n$. 
Then by \cite[Lem.~4.19]{mclean} 
the FSP \eqref{e:FSP}-\eqref{e:SRC} 
is equivalent to the following transmission problem: 
\begin{equation}\label{e:Trans}
\begin{aligned}
&\begin{array}{rlrl}
\Delta u^{\rm s}+k^2 (1+\m) u^{\rm s} \hspace{-2mm} &= \hspace{1mm}  -f|_{\Omega} \quad \text { in } \Omega , \\[1mm]
\Delta u^{\rm s}+k^2 u^{\rm s} \hspace{-2mm} &= \hspace{1mm} -f|_{\overline{\Omega}^c} \,\,\,\, \text { in } \overline{\Omega}^c, \\[1mm]
\gamma^{+} u^{\rm s} \hspace{-2mm} &= \hspace{1mm} \gamma^{-} u^{\rm s} \quad \text { on } \partial \Omega, \\[1mm]
\partial_n^{+} u^{\rm s} \hspace{-2mm} &= \hspace{1mm} \partial_n^{-} u^{\rm s} \quad \text { on } \partial \Omega,
\end{array}
\end{aligned}
\end{equation}
along with the radiation condition \eqref{e:SRC}, where $\gamma^\pm$ and $\partial_n^\pm$ denote the exterior($+$)/interior($-$) Dirichlet and Neumann traces on $\partial\Omega$. 
The FSP \eqref{e:FSP}-\eqref{e:SRC} therefore provides 
a generalisation of the transmission problem \eqref{e:SRC}-\eqref{e:Trans} to non-Lipschitz inhomogeneities. 
\end{remark}

\begin{theorem}[FSP well-posedness]\label{t:FSPWP}
Suppose that 
$\Im[\m] \geq 0$. 
Then there exists a unique $u^{\rm s}\in 
H^{1}_{\rm loc}(\R^n)$ satisfying
\eqref{e:FSP}-\eqref{e:SRC}.
\end{theorem}

\begin{proof}
For uniqueness, suppose that $u^{\rm s}$ satisfies \eqref{e:FSP}-\eqref{e:SRC} with $f=0$ and let $R>0$ be large enough that $\supp(\m)\subset B_R$ (the open ball of radius $R$ centred at the origin). An application of Green's first identity in $B_R$ gives that $\langle \partial_n^-u^{\rm s},\gamma^- u^{\rm s}\rangle_{\partial B_R} = \int_{B_R} |\nabla u^{\rm s}|^2-k^2(1+\m)|u^{\rm s}|^2$. Since $u^{\rm s}$ satisfies $\Delta u^{\rm s}+k^2 u^{\rm s}=0 $ in a neighbourhood of $\partial B_R$ we have $\partial_n^+u^{\rm s}=\partial_n^-u^{\rm s}$ and $\gamma^+u^{\rm s}=\gamma^-u^{\rm s}$ on $\partial B_R$, so that $\Im[\langle \partial_n^+u^{\rm s},\gamma^+ u^{\rm s}\rangle_{\partial B_R}] = -k^2 \int_{B_R} \Im[\m]|u^{\rm s}|^2\leq 0$, and then by \cite[Lem.~9.9]{mclean} $u^{\rm s}=0$ in $\overline{B_R}^c$. Finally, let $R'>R$ and note that $u^{\rm s}|_{B_{R'}}$ belongs to $H^1_0(B_{R'})$, satisfies \eqref{e:FSP} in $B_{R'}$ with $f=0$, and vanishes on $B_{R'}\setminus \overline{B_{R}}$, so that by the unique continuation principle (e.g., \cite[Thm~2.1]{ucp}) $u^{\rm s}=0$ in $B_{R'}$. Hence $u^{\rm s}=0$ on $\R^n$. 

For existence, let $R>0$ be such that $ B_R$ contains both $\supp(\m)$ and $\supp(f)$ and consider the problem: find $u\in H^1(B_R)$ such that 
\begin{align}
\label{e:VarProb}
a_R(u,v):=\int_{B_R} \nabla u \cdot \overline{\nabla u} - k^2 (1+\m) u \overline{v} - \langle T_R \gamma^-u,\gamma^-v\rangle_{\partial B_R} = \int_{B_R}f\overline{v}, \qquad \forall v\in H^1(B_R),
\end{align}
where $T_R:H^{1/2}(\partial B_R) \to H^{-1/2}(\partial B_R)$ denotes the Dirichlet-to-Neumann map for the exterior Dirichlet problem for $\Delta u+k^2 u=0$ in $\overline{B_R}^c$ with the SRC at infinity, and $H^{\pm 1/2}(\partial B_R)$ denote the standard boundary Sobolev spaces. 
Arguing as in \cite[Lemma 3.3]{gps}, we note that if $u$ satisfies \eqref{e:VarProb} then the function 
\begin{align}
\label{e:Ext}
u^{\rm s}:=\begin{cases}
u, & \text{on }B_R,\\
U, & \text{on }\overline{B_R}^c,
\end{cases} 
\end{align}
where $U$ denotes the solution of the exterior Dirichlet problem mentioned above, with Dirichlet data equal to $\gamma^-u$ on $\partial B_R$, satisfies \eqref{e:FSP}-\eqref{e:SRC}. 
The sesquilinear form $a_R(\cdot,\cdot)$ is continuous on $H^1(B_R)$, by the Cauchy-Schwarz inequality, the boundedness of $T_R:H^{1/2}(\partial B_R) \to H^{-1/2}(\partial B_R)$ and the fact that $\m\in L^\infty(\R^n)$. Furthermore, it satisfies a G{\aa}rding inequality, since for $u\in H^1(B_R)$ the fact that $\Re[-\langle T_R \gamma^-u,\gamma^-v\rangle_{\partial B_R}]\geq 0$ (see, e.g., \cite[Lemma 3.2]{gps}) implies that 
\begin{align}
\label{e:Garding}
\Re[a_R(u,u)]\geq \|\nabla u\|^2_{L^2(B_R)} - k^2\|\Re[1+\m]\|_{L^\infty(B_R)}\| u\|^2_{L^2(B_R)} \geq \|\nabla u\|^2_{H^1(B_R)} - \beta \| u\|^2_{L^2(B_R)},
\end{align}
where $\beta:=k^2\|\Re[1+\m]\|_{L^\infty(B_R)} + 1$. The form $a_R(\cdot,\cdot)$ is also injective because if $u\in H^1(B_R)$ and $a_R(u,v)=0$ for every $v\in H^1(B_R)$ then $u^{\rm s}$ defined as in \eqref{e:Ext} satisfies \eqref{e:FSP}-\eqref{e:SRC} with $f=0$ (cf.\ \cite[Lemma 2.3]{gps}). Hence by uniqueness $u^{\rm s}=0$ on $\R^n$, so that $u=u^{\rm s}|_{B_R}=0$. Finally, since $H^1(B_R)$ is compactly embedded in $L^2(B_R)$ we deduce the existence of a solution of \eqref{e:VarProb} (using e.g.\ \cite[Thm~5.18]{waef}), and hence by \eqref{e:Ext} a solution of \eqref{e:FSP}-\eqref{e:SRC}.
\end{proof}

Let 
$\Phi$ 
denote the outgoing fundamental solution of the Helmholtz equation in $\R^n$, viz.
$$
\Phi(x,y) := \frac{\ri}{4}H_{0}^{(1)}(k|x-y|), \quad n=2,
\quad \text{ and }\quad
\Phi(x,y) := \frac{\re^{\ri k|x-y|}}{4\pi|x-y|}, \quad n=3, 
$$
where $H_0^{(1)}$ denotes the Hankel function of the first kind of order $0$. 
We note that $\Phi(x,\cdot)\in L^2_{\rm loc}(\R^n)$ for both $n=2,3$.
Let $\mathcal{V}$ denote the acoustic Newtonian potential %
defined by
\begin{equation} \label{e:Newt}
\mathcal{V}\psi(x)=\int_{\R^n}\Phi(x,y)\psi(y)\,dy, \qquad \mbox{for } \quad x\in \R^n \quad \mbox{and} \quad \psi\in L^2_{\mathrm{comp}}(\R^n).
\end{equation}
For any $s\in\R$, $\mathcal{V}$ extends to a continuous operator (see, e.g., \cite[Cor.~6.5, Eqn (6.10)]{mclean}, combined with \cite[Thm~2.6.7]{ss11})
\begin{align}
 \label{e:NewtReg}
 \mathcal{V}:H^s_{\rm comp}(\R^n) \to H^{s+2}_{\rm loc}(\R^n),
\end{align}
and (see, e.g., \cite[p197]{mclean})
\begin{align}
 \label{e:NewtEqs}
 (\Delta + k^2)\mathcal{V}\psi =  \mathcal{V}(\Delta + k^2)\psi = -\psi, \qquad \psi \in H^s_{\mathrm{comp}}(\R^n), \quad s\in \R.
 \end{align}
 We define 
$\cM_{\m}:L^2_{\rm loc}(\R^n)\to L^2_K:=\{u\in L^2(\R^n):\supp(u)\subset K\}$ (a continuous operator) 
by
\begin{align}
\label{e:MDef}
\cM_{\m}\psi(x):=\m(x)\psi(x), \qquad \psi\in L^2_{\rm loc}(\R^n).
\end{align}
The following equivalence theorem is a generalisation of \cite[Theorem 8.3]{ck2}.

\begin{theorem}[FSP-LSE equivalence]
\label{t:FSP-LSE}
If $u^{\rm s}\in H^{1}_{\rm loc}(\R^n)$ satisfies the FSP \eqref{e:FSP}-\eqref{e:SRC} then $u^{\rm s}\in H^{2}_{\rm loc}(\R^n)\cap C^{0}(\R^n)$ and $u^{\rm s}$ satisfies the Lippmann-Schwinger equation (LSE)
\begin{equation}\label{e:LSE} 
	u^{\rm s} = \mathcal{V}f+k^{2}\mathcal{V}
	\cM_{\m} u^{\rm s},
\end{equation}
i.e.
\begin{equation}\label{e:LSE2} 
	u^{\rm s}(x) = \int_{\R^n}\left( f(y)+k^{2}\m(y)u^{\rm s}(y) \right) \Phi(x,y) \ \mathrm{d}y, \quad x\in \R^n.
\end{equation}
\noindent Conversely, if $u^{\rm s}\in L^{2}_{\rm loc}(\R^n)$ satisfies the LSE \eqref{e:LSE}, then $u^{\rm s}\in H^{2}_{\rm loc}(\R^n)\cap C^{0}(\R^n)$ and $u^{\rm s}$ satisfies the FSP \eqref{e:FSP}-\eqref{e:SRC}. 
\end{theorem}

\begin{proof}
Suppose that $u^{\rm s}\in H^{1}_{\rm loc}(\R^n)$ satisfies  \eqref{e:FSP}-\eqref{e:SRC}. Then
$\Delta u^{\rm s}+k^s u^{\rm s} = -f - k^2\m u^{\rm s}$ as an equality in $L^2_{\rm comp}(\R^n)$, and we noted already in Remark \ref{r:ER} that $u^{\rm s}\in H^{2}_{\rm loc}(\R^n)\cap C^{0}(\R^n)$. Hence, letting $R>0$ be such that $B_R$ contains both $K$ and $\supp(f)$, an application of Green's representation theorem in $B_R$ (see, e.g., \cite[Theorem 2.1]{ck2} and the comments after \cite[Theorem 8.2]{ck2} arguing that \cite[Theorem 2.1]{ck2} holds also for elements of $H^2(B_R)$, which are expanded in \cite[Theorem 2.3.1]{bannisterthesis}), gives that
\begin{align*}
u^{\rm s}(x) &= \int_{B_R}\left(f(y)+k^{2}\m(y)u^{\rm s}(y) \right) \Phi(x,y) \ \mathrm{d}y,\notag\\
&  - \int_{\partial B_R} \partial_{n(y)}^-u^{\rm s}\, \gamma^-\Phi(x,y) - \gamma^-u^{\rm s}\, \partial_{n(y)}^-\Phi(x,y) \mathrm{d}s(y), 
\qquad x\in B_R.
\end{align*}
{Since $u^{\rm s}$ satisfies the SRC \eqref{e:SRC}, a standard argument (see, e.g., \cite[Thm~2.5]{ck2}) shows that the integral over $\partial B_R$ tends to zero as $R\to \infty$, from which we obtain \eqref{e:LSE2}.} 

Now suppose that $u^{\rm s}\in L^{2}_{\rm loc}(\R^n)$ satisfies \eqref{e:LSE}. Then $u^{\rm s}\in H^{2}_{\rm loc}(\R^n)\cap C^{0}(\R^n)$ by \eqref{e:NewtReg} and the Sobolev embedding theorem, $u^{\rm s}$ satisfies \eqref{e:FSP} by \eqref{e:NewtEqs}, and $u^{\rm s}$ satisfies \eqref{e:SRC} because $\Phi$ does (see, e.g., \cite[p.~282]{mclean}).
\end{proof}

\begin{corollary}[LSE well-posedness]
\label{c:LSEWP}
Let $\Im[\m] \geq 0$. %
Then the LSE \eqref{e:LSE} has a unique solution $u^{\rm s}$, which 
coincides with the unique solution of the FSP \eqref{e:FSP}-\eqref{e:SRC}.
\end{corollary}

The LSE \eqref{e:LSE} can be viewed as a ``direct'' integral equation formulation since the unknown $u^s$ is the scattered field. One can also consider an equivalent ``indirect'' formulation.

\begin{proposition}[LSE-ILSE equivalence]
\label{p:Indirect}
If $u^{\rm s}\in L^2_{\rm loc}(\R^n)$ satisfies the LSE \eqref{e:LSE} then $\sigma^{\rm s}:=f+k^2\cM_\m u^{\rm s}\in L^2_{\rm comp}(\R^n)$ satisfies the indirect LSE (ILSE)
\begin{align}
\label{e:ILSE}
\sigma^{\rm s} = f + k^2\cM_\m \cV \sigma^{\rm s},
\end{align}
and 
\begin{align}
\label{e:ILSERep}
u^{\rm s}=\cV\sigma^{\rm s}.
\end{align}
Conversely, if $\sigma^{\rm s}\in L^2_{\rm comp}(\R^n)$ satisfies \eqref{e:ILSE} then $u^{\rm s}$ defined by \eqref{e:ILSERep} satisfies the LSE \eqref{e:LSE}, and 
$\sigma^{\rm s}=f+k^2\cM_\m u^{\rm s}$. 
\end{proposition}
\begin{proof}
Follows by straightforward direct calculations. 
\end{proof}

\subsection{The restricted LSE}
\label{s:Restr}

We now consider the restriction of the LSE \eqref{e:LSE} and ILSE \eqref{e:ILSE} to a bounded set $\Omega\subset\R^n$, as a step towards numerical discretisations. The set $\Omega$ could be, for instance, the inhomogeneity $K$, its interior, or some larger set containing it (such as a bounding box, cf.\ Figure \ref{f:Disc}(a)). 

Let $\Omega\subset \R^n$ be bounded and measurable. %
We define the restriction operator 
$|_\Omega:L_{\rm loc}^2(\R^n)\to L^2(\Omega)$ (a continuous operator) by $(u|_\Omega)(x):=u(x)$ for $u\in L^2_{\rm loc}(\R^n)$ and $x\in \Omega$, and 
the extension by zero operator
$\mathcal{Z}_{\Omega}:L^{2}(\Omega)\rightarrow L^{2}(\R^n)$ 
(a bounded operator)
by 
\begin{equation*}
	(\mathcal{Z}_{\Omega}\phi)(x) := \begin{cases}   \phi(x), & \text { for   } x\in \Omega, \\  
	 0, & \text { for } x\in \R^n\setminus \Omega, \end{cases}
	 \qquad \phi\in L^2(\Omega),
\end{equation*}
and note that 
$(\cZ_\Omega \phi)|_\Omega=\phi$ for all $\phi\in L^2(\Omega)$. 
We then define
$\mathcal{V}_\Omega:L^2(\Omega)\to L^2(\Omega)$ and $\mathcal{V}_\Omega^*:L^2(\Omega)\to L^2(\Omega)$ (both bounded operators) by
\begin{align}
 \label{e:VVstar}
 \mathcal{V}_\Omega := |_\Omega \circ \mathcal{V} \circ \cM_{\m}\circ \mathcal{Z}_\Omega, 
\qquad 
\mathcal{V}_\Omega^* := |_\Omega \circ \cM_{\m}\circ \cV \circ \mathcal{Z}_\Omega, 
 \end{align} 
which have integral representations
\[ \mathcal{V}_\Omega \phi (x) = \int_\Omega \Phi(x,y) \m(y)\phi(y)\, \rd y, \qquad \mathcal{V}_\Omega^* \phi (x) = \m(x)\int_\Omega \Phi(x,y) \phi(y)\, \rd y, \qquad \phi\in L^2(\Omega).\]
We note that $\cV_\Omega^*$ is the adjoint of $\cV_\Omega$, modulo complex conjugation, in the sense that
\begin{align}
\label{e:Adj}
(\cV_\Omega \phi,\psi) = (\phi,\overline{\cV_\Omega^*\overline{\psi}}), \qquad \phi,\psi\in L^2(\Omega).
\end{align}

\begin{lemma}[Compactness of $\cV_\Omega$ and $\cV_\Omega^*$]
\label{l:Comp}
Let $\Omega\subset \R^n$ be bounded and measurable. 
Then $\mathcal{V}_\Omega:L^2(\Omega)\to L^2(\Omega)$ and $\mathcal{V}_\Omega^*:L^2(\Omega)\to L^2(\Omega)$ are compact. 
\end{lemma}
\begin{proof}
Let $D\subset\R^n$ be any bounded open set such that $D\supset \Omega$. Then $\mathcal{V}_\Omega$ can be written as $ \mathcal{V}_\Omega = |_{D\to \Omega} \circ |_D \circ \mathcal{V} \circ \cM_{\m} \circ \mathcal{Z}_\Omega$, where $|_{D\to \Omega}$ denotes the restriction operator from $L^2(D)$ to $L^2(\Omega)$. 
The composition $\mathcal{V} \circ \cM_{\m} \circ \mathcal{Z}_\Omega$ is continuous from $L^2(\Omega)\to H^2_{\rm loc}(\R^n)$, so that by \cite[Thm 2.6.7(c)]{ss11} the map $\phi\mapsto \chi (\mathcal{V} ( \cM_{\m} (\mathcal{Z}_\Omega(\phi)))$ is bounded from $L^2(\Omega)\to H^2(\R^n)$ for every $\chi\in C^\infty_{\rm comp}(\R^n)$. By choosing $\chi$ to satisfy $\chi\equiv 1$ on $D$ we can obtain boundedness of $|_D \circ \mathcal{V} \circ \cM_{\m} \circ \mathcal{Z}_\Omega$ from $L^2(\Omega)\to H^2(D)$, and the claimed compactness then follows from the 
compactness of the embedding $H^2(D)\subset L^2(D)$ (see, e.g., \cite[Thm~3.27]{mclean}) and the boundedness of $|_{D\to \Omega}:L^2(D)\to L^2(\Omega)$. The compactness of $\cV_\Omega^*$ follows from the adjoint property \eqref{e:Adj}. 
\end{proof}

The following propositions provide conditions under which the LSE and ILSE are equivalent to their restricted versions. Both propositions assume that $|K\setminus\Omega|=0$, which ensures that $\Omega$ is ``large enough'' that $\m$ is determined by its values on $\Omega$, so that in particular 
\begin{align}
\label{e:KOmega}
\text{if } |K\setminus\Omega|=0 \text{ then } \cZ_\Omega\circ |_\Omega \circ \cM_\m =\cM_\m \circ \cZ_\Omega\circ |_\Omega = \cM_\m  \text{ on }L^2_{\rm loc}(\R^n).
\end{align}
Also, if $|K\setminus\Omega|=0$ then we can write $\m=m\chi_\Omega$ for some $m\in L^\infty_{\rm comp}(\R^n)$, where $\chi_\Omega$ is the characteristic function of $\Omega$, 
a fact we shall use in \S\ref{s:SolReg} (see in particular \eqref{e:smOmega}). 

\begin{proposition}
[LSE-RLSE equivalence]
\label{p:IntEqn}
Let 
$\Omega\subset \R^n$ be bounded and measurable with $|K\setminus \Omega|=0$. Suppose that $u^{\rm s}$ satisfies the LSE \eqref{e:LSE}. Then $\tilde{u}:=u^{\rm s}|_\Omega$ satisfies the 
restricted LSE (RLSE):
\begin{align}
\label{e:IntEqn}
(\mathcal{I} - k^2\mathcal{V}_\Omega)\tilde{u} = 
\mathcal{V}f|_\Omega.
\end{align}
Conversely, suppose that $\tilde{u}\in L^2(\Omega)$ satisfies the RLSE \eqref{e:IntEqn}.
Then 
\begin{align}
\label{e:usDef}
u^{\rm s}:=\mathcal{V}f + k^2\mathcal{V}(\cM_{\m}(\mathcal{Z}_\Omega(\tilde{u})))
\end{align}
satisfies the LSE \eqref{e:LSE} and $u^{\rm s}|_\Omega=\tilde{u}$. 
\end{proposition}
\begin{proof}
The first statement follows by direct calculation, applying $|_\Omega$ to \eqref{e:LSE} and using \eqref{e:KOmega}. 
For the second, taking the restriction of \eqref{e:usDef} to $\Omega$ and combining with \eqref{e:IntEqn} gives $u^{\rm s}|_\Omega=\tilde{u}$.
Combining this with \eqref{e:usDef} and \eqref{e:KOmega} proves that $u^{\rm s}$ satisfies  \eqref{e:LSE}. 
\end{proof}

\begin{proposition}
[ILSE-RILSE equivalence]
\label{p:IntEqn2}
Let 
$\Omega\subset \R^n$ be bounded and measurable with $|K\setminus \Omega|=0$ and $|\supp f \setminus \Omega|=0$. Suppose that $\sigma^{\rm s}$ satisfies the ILSE \eqref{e:ILSE}. Then $\sigma:=\sigma^{\rm s}|_\Omega$ satisfies the 
restricted ILSE (RILSE):
\begin{align}
\label{e:RILSE}
(\mathcal{I} - k^2\mathcal{V}_\Omega^*)\sigma = 
f|_\Omega.
\end{align}
Conversely, suppose that $\sigma\in L^2(\Omega)$ satisfies the RILSE \eqref{e:RILSE}. %
Then 
\begin{align}
\label{e:sigsDef}
\sigma^{\rm s}:=\mathcal{Z}_\Omega(\sigma)
\end{align}
satisfies the ILSE \eqref{e:ILSE} and $\sigma^{\rm s}|_\Omega=\sigma$. 
\end{proposition}
\begin{proof}
The first statement follows by direct calculation, noting that 
the assumptions on $\Omega$ imply that
$\cZ_\Omega(\sigma^{\rm s}|_\Omega)=\sigma^{\rm s}$. 
For the second, we first note that $\sigma^{\rm s}|_\Omega=(\cZ_\Omega\sigma)|_\Omega=\sigma$. Then, applying $\cZ_\Omega$ to \eqref{e:RILSE} and using \eqref{e:KOmega}, and the fact that $\cZ_\Omega(f|_\Omega)=f$ (which holds because $|\supp f\setminus \Omega|=0$), gives \eqref{e:ILSE}.
\end{proof}

\begin{corollary}[RLSE and RILSE well-posedness]
\label{c:IntEqn}
Let 
$\Im[\m] \geq 0$. %
Let $\Omega\subset \R^n$ be bounded and measurable with $|K\setminus \Omega|=0$. Then for any $g,\tau\in L^2(\Omega)$ the equations 
\begin{align}
\label{e:IntEqn2}
(\mathcal{I} - k^2\mathcal{V}_\Omega)u = 
g
\quad \text{and} \quad
(\mathcal{I} - k^2\mathcal{V}_\Omega^*)\sigma = 
\tau
\end{align}
both have unique solutions $u,\sigma\in L^2(\Omega)$. 
\end{corollary}
\begin{proof}
By Lemma \ref{l:Comp} the operator $(\mathcal{I} - k^2\mathcal{V}_\Omega)$ is Fredholm of index zero, so to prove it is invertible we just need to prove it is injective. For this, suppose that $u\in L^2(\Omega)$ satisfies \eqref{e:IntEqn2} with $g=0$. Then by Proposition \ref{p:IntEqn} the function $u^{\rm s}:= k^2\mathcal{V}(\cM_{\m}(\mathcal{Z}_\Omega(u)))$ satisfies the LSE \eqref{e:LSE} with $f=0$, and furthermore $u^{\rm s}|_\Omega=u$. By Corollary \ref{c:LSEWP} we deduce that $u^{\rm s}=0$, so that also $u=u^{\rm s}|_\Omega=0$. Invertibility of $(\mathcal{I} - k^2\mathcal{V}_\Omega^*)$ follows using the adjoint relation \eqref{e:Adj}.
\end{proof}

\subsection{Solution regularity}
\label{s:SolReg}
The results of \S\ref{s:FSP-LSE}-\S\ref{s:scatt} show that the FSP \eqref{e:FSP}-\eqref{e:SRC} is well posed for general $\m\in L^\infty_{\rm comp}$ (with $K=\supp\m$ a general compact set), provided that $\Im[\m]\geq 0$. They also show that the FSP can be equivalently formulated by the RLSE \eqref{e:IntEqn}-\eqref{e:usDef} 
on any bounded measurable set $\Omega$ such that $|K\setminus \Omega|=0$, 
or by the RILSE \eqref{e:RILSE}-\eqref{e:sigsDef} under the additional assumption that $|\supp f\setminus \Omega|=0$.
We now consider how the smoothness of the FSP, RLSE and RILSE solutions is affected by additional regularity assumptions on $\m$ and $\Omega$.

Our regularity assumptions will be stated in the language of Sobolev multipliers. For a function $\tau\in L^\infty_{\rm comp}(\R^n)$ and 
$s\in\R$ 
we say $\tau$ is a Sobolev multiplier on $H^s(\R^n)$, and write $\tau\in \mathfrak{M}(H^s(\R^n))$, if the pointwise multiplication $\phi\mapsto \tau\phi$ defines a bounded operator from $H^s(\R^n)$ to $H^s(\R^n)$ 
(cf.\ e.g.\ 
\cite{Sickel99a,RunstSickel}). 
It is a standard result (see, e.g., \cite[Lemma 7]{sickel1999pointwise}) that 
\begin{align}
\label{e:Mult}
\text{if }\;0<s<t 
\;\text{ then }\;
\mathfrak{M}(H^t(\R^n)) \subset \mathfrak{M}(H^s(\R^n)).
\end{align}
Recalling that $\m\in L^\infty_{\rm comp}(\R^n)$, 
let
\begin{align}
\label{e:sm}
s_\m:=\sup
\{s\geq 0: \m\in\mathfrak{M}(H^s(\R^n))\}\in [0,\infty]. 
\end{align}
Given a bounded measurable set $\Omega\subset \R^n$, let
\begin{align}
\label{e:sOmega}
s_\Omega:=\sup
\{s\geq 0: \chi_\Omega\in\mathfrak{M}(H^s(\R^n))\}\in [0,\infty],
\end{align}
and if $|K\setminus \Omega|=0$, so that $\m=m\chi_\Omega$ for some $m\in L^\infty_{\rm comp}(\R^n)$, let
\begin{align}
\label{e:smOmega}
s_{\m,\Omega}:=\sup
\{s\geq 0: \m=m\chi_\Omega \text{ for some } m\in \mathfrak{M}(H^s(\R^n))\}\in [0,\infty].
\end{align}
By \eqref{e:Mult}, 
if $0\leq s<s_\m$ then 
$\m\in\mathfrak{M}(H^s(\R^n))$. 
When $s_\m<\infty$ it is not guaranteed that $\m\in \mathfrak{M}(H^{s_\m}(\R^n))$;  
however, if $s_\m=0$ then $\m\in \mathfrak{M}(H^{s_\m}(\R^n))=\mathfrak{M}(L^2(\R^n))$ because $\m\in L^\infty_{\rm comp}(\R^n)$. 
Similarly, if $0\leq s<s_\Omega$ then 
$\chi_\Omega\in\mathfrak{M}(H^s(\R^n))$, and it may or may not be true that $\chi_\Omega \in \mathfrak{M}(H^{s_\Omega}(\R^n))$, but definitely if $s_\Omega=0$ then $\chi_\Omega\in \mathfrak{M}(H^{s_\Omega}(\R^n))=\mathfrak{M}(L^2(\R^n))$.
Similarly, if $0\leq s<s_{\m,\Omega}$ then there exists $m\in \mathfrak{M}(H^s(\R^n))$ such that $\m=m\chi_\Omega$, and if $s_{\m,\Omega}<\infty$ it may or may not be true that there exists $m\in \mathfrak{M}(H^{s_{\m,\Omega}}(\R^n))$ such that $\m=m\chi_\Omega$, but definitely if $s_{\m,\Omega}=0$ then there exists $m\in \mathfrak{M}(H^{s_{\m,\Omega}}(\R^n))=\mathfrak{M}(L^2(\R^n)$ such that $\m=m\chi_\Omega$ (just take $m=\m$). 

\begin{remark}
We comment on the values of $s_\m$, $s_\Omega$ and $s_{\m,\Omega}$ in certain cases:
\label{r:MCont}
\begin{enumerate}[(i)]
\item
\label{i:Csigma}
If ${s'}\in\N$ and $\m\in C^{{s'}}_{\rm comp}(\R^n)$ then $\m\in\mathfrak{M}(H^s(\R^n))$ for $|s|\leq {s'}$ (e.g.,\ \cite[Thm 3.20]{mclean}), so that $s_\m\geq {s'}$, and $s_{\m,\Omega}\geq {s'}$ for any non-empty bounded measurable set $\Omega\subset\R^n$. 
\item
\label{i:CInfty}
If $\m\in C^{\infty}_{\rm comp}(\R^n)$ then 
$s_\m=\infty$, and $s_{\m,\Omega}=\infty$ for any non-empty bounded measurable set $\Omega\subset\R^n$.
\item 
\label{i:Charac}
If $\Omega$ is a non-empty bounded open set then $s_\Omega\leq 1/2$, and $\chi_\Omega\not\in \mathfrak{M}(H^{1/2}(\R^n))$. This holds by \cite[Lem.~3.2]{sickel1999pointwise} (and see also the references given before it). 
\item 
\label{i:Lip}
If 
$\Omega$ is a non-empty bounded Lipschitz open set 
then $s_\Omega=1/2$; indeed, 
$\chi_\Omega\in \mathfrak{M}(H^s(\R^n))$ if and only if $|s|<1/2$ (see, e.g., \cite[Prop.~3.2]{dequalsnpaper}, which is based on \cite[Props~5.1 \& 5.3]{Tri:02}). 
This is a situation where $\chi_\Omega\not\in \mathfrak{M}(H^{s_\Omega}(\R^n))$.
\item 
\label{i:smIneq}
For any $\m\in L^\infty_{\rm comp}(\R^n)$ and any non-empty bounded measurable set $\Omega\subset\R^n$ we have that $s_{\m,\Omega}\geq s_\m$, 
since $m=\m$ is a valid choice for $m$ in the set defining $s_{\m,\Omega}$. 
\item
\label{i:chiOmega} 
If $\m = m\chi_\Omega$ for some $m\in \C\setminus\{0\}$ and some non-empty bounded open set $\Omega\subset\R^n$ 
then $s_\m = s_\Omega$; indeed, $\m\in \mathfrak{M}(H^s(\R^n))$ if and only if $\chi_\Omega\in \mathfrak{M}(H^s(\R^n))$. Furthermore, $s_{\m,\Omega}=\infty$. 
In light of point \eqref{i:Charac}, this is a situation where $s_\m<s_{\m,\Omega}$. 
Furthermore, by point \eqref{i:Lip}, if $\Omega$ is Lipschitz then this is a situation where $\m\not\in \mathfrak{M}(H^{s_\m}(\R^n))$.
\item 
\label{i:Bounding}
Let $\m$ and $\Omega$ be as in point \eqref{i:chiOmega}, and let $\Omega'$ be a bounded open set 
such that $\overline{\Omega}\subset\Omega'$. Then 
$s_{\m,\Omega'}=s_\m =s_\Omega \in [0,1/2]$, while, as noted in point \eqref{i:chiOmega}, $s_{\m,\Omega} =\infty$. 
\end{enumerate} 
\end{remark}

Below we shall prove results under the assumption of certain statements $P$ involving the symbol $s_\m^-$, for instance ``$s\leq s_\m^-$'' or ``$u\in H^{s_\m^-}(\R^n)$''. This indicates that the result is valid if $P$ holds with $s_\m^-$ replaced by any $s\in[0,s_\m)$, and that if $s_\m<\infty$ and $\m\in \mathfrak{M}(H^{s_\m}(\R^n))$ then the result is also valid if $P$ holds with $s_\m^-$ replaced by $s_\m$. 
Statements involving the symbols $s_\Omega^-$ and $s_{\m,\Omega}^-$ should be interpreted similarly.

The following lemma is an immediate consequence of \eqref{e:NewtReg} and \eqref{e:MDef}.%
\begin{lemma}[Mapping properties of $\cV\circ\cM_{\m}$]
\label{l:MapProp}
Let $s\geq 0$. Then $\mathcal{V}\circ\cM_{\m}:H^s_{\rm loc}(\R^n)\to H^{\min(s,s_\m^-)+2}_{\rm loc}(\R^n)$ is continuous.
\end{lemma}

As a consequence, we have the following elementary regularity result. 
\begin{lemma}[FSP/LSE solution regularity]
\label{l:Reg}
Let $0\leq s\leq s_\m^-$, $f\in H^{s}_{\rm comp}(\R^n)$, %
$\Im[\m] \geq 0$. 
Then the solution of the FSP \eqref{e:FSP}-\eqref{e:SRC} (equivalently, of the LSE \eqref{e:LSE}) satifies $u^{\rm s}\in H^{s+2}_{\rm loc}(\R^n)$.
\end{lemma}
\begin{proof}
The case $s=0$ is already covered by Theorem \ref{t:FSPWP}, Theorem \ref{t:FSP-LSE} and Corollary \ref{c:LSEWP}. 
The case $s>0$ follows 
by repeated applications of \eqref{e:LSE} and Lemma \ref{l:MapProp}.
\end{proof}

We now prove mapping properties of the operators $\cV_\Omega$ and $\cV_\Omega^*$ in the case %
where $\Omega$ is a bounded open set.
The properties of $\cV_\Omega^*$ are different to those of $\cV_\Omega$ because the volume potential and multiplication operator are applied in the opposite order. 
Since $\cV_\Omega$ and $\cV_\Omega^*$ are defined on $L^2(\Omega)$, it is natural to consider their mapping properties in the spaces $H^s(\Omega)\subset L^2(\Omega)$ for $s\geq0$. %
We shall also consider mapping properties of $\cV_\Omega^*$ in the spaces 
\[ H^s_{\rm ze}(\Omega):=\{\phi\in H^s(\Omega):\cZ_\Omega(\phi)\in H^s(\R^n)\}, \quad s\geq 0,\]
equipped with the norm
\[ \|\phi \|_{H^s_{\rm ze}(\Omega)}:=\|\cZ_\Omega(\phi)\|_{H^s(\R^n)}, \quad \phi\in H^s_{\rm ze}(\Omega).\]
We note that $H^s_{\rm ze}(\Omega)\subset H^s(\Omega)$ is a Hilbert space, unitarily isomorphic to the closed subspace $\cirH^s(\Omega):=\{u\in H^s(\R^n):u=0 \text{ a.e.\ on } \Omega^c\}$ of $H^s(\R^n)$ (see, e.g., \cite[Rem.~3.1]{nonlip}), and that $H^s_{\rm ze}(\Omega)$ is continuously embedded in $H^s(\Omega)$ with 
\begin{align}
\label{e:HzeEmb}
\|\phi\|_{H^s(\Omega)}\leq \|\phi \|_{H^s_{\rm ze}(\Omega)}, \qquad \phi\in H^s_{\rm ze}(\Omega). 
\end{align}
We note that $H^s_{\rm ze}(\Omega)=H^s(\Omega)$ (with equivalent norms) if and only if $\cZ_\Omega:H^s(\Omega)\to H^s(\R^n)$ is bounded, which holds if and only if $\chi_\Omega\in\mathfrak{M}(H^s(\R^n))$. 
The latter equivalence holds because $\chi_\Omega u = \cZ_\Omega(u|_\Omega)$ for $u\in L^2(\R^n)$, and because $\cZ_\Omega(\phi)=\chi_\Omega \mathcal{E}^s_\Omega(\phi)$ for $\phi\in H^s(\Omega)$, where $\mathcal{E}^s_\Omega:H^s(\Omega)\to H^s(\R^n)$ is the ``minimal norm extension operator'' on $H^s(\Omega)$, mapping $\phi\in H^s(\Omega)$ to the unique $\tilde{\phi}\in H^s(\R^n)$ such that $\tilde{\phi}|_\Omega=u$ and $\|\phi\|_{H^s(\Omega)}=\|\tilde{\phi}\|_{H^s(\R^n)}$ (see, e.g., \cite[\S 3.1.4]{nonlip}).  
We note also that $H^0(\Omega)=H_{\rm ze}^0(\Omega)=L^2(\Omega)$.

\begin{lemma}[Mapping properties of $\cV_\Omega$ and $\cV_\Omega^*$]
\label{l:MapProp2}
Let $\Omega\subset\R^n$ be bounded and open with $|K\setminus \Omega|=0$. 
For $s\geq 0$ the following operators are bounded:
\begin{align}
\cV_\Omega&:H^s(\Omega)\to H^{\min(s,s_\m^-)+2}(\Omega),\label{e:MP1}\\
\cV_\Omega^*&:H^s(\Omega)\to H^{\min(\min(s,s_\Omega^-)+2,s_{\m,\Omega}^-)}(\Omega),\label{e:MP3}\\
\cV_\Omega^*&:H_{\rm ze}^s(\Omega)\to H_{\rm ze}^{\min(s+2,s_\m^-)}(\Omega)\label{e:MP4}.
\end{align}
If $\Im[\m] \geq 0$ then the following operators are boundedly invertible:
\begin{align}
\label{}
&\mathcal{I}-k^2\mathcal{V}_\Omega:H^{s}(\Omega)\to H^{s}(\Omega), 
&&  %
0\leq s\leq s_\m^-+2,\label{e:MP5}\\ 
&\mathcal{I}-k^2\mathcal{V}_\Omega^*:H^{s}(\Omega)\to H^{s}(\Omega),
&&  %
0\leq s\leq \min(s_\Omega^- + 2,s_{\m,\Omega}^-),\label{e:MP7}\\
&\mathcal{I}-k^2\cV_\Omega^*:H_{\rm ze}^s(\Omega)\to H_{\rm ze}^s(\Omega),
&&  %
0\leq s\leq s_\m^-\label{e:MP8}.
\end{align}

\end{lemma}

\begin{proof}
The operator $\mathcal{V}_\Omega$ was defined as $\mathcal{V}_\Omega := |_\Omega \circ \mathcal{V} \circ \cM_{\m} \circ \mathcal{Z}_\Omega$. However, since 
$|K\setminus \Omega|=0$, 
for $s\geq 0$ its action on $H^s(\Omega)$ can be equivalently represented as $\mathcal{V}_\Omega = |_\Omega \circ \mathcal{V} \circ \cM_{\m} \circ \mathcal{E}^s_\Omega$, where $\mathcal{E}^s_\Omega$
is the minimal norm extension operator discussed above. 
Then \eqref{e:MP1} follows from 
Lemma \ref{l:MapProp}. 
For \eqref{e:MP5}, 
suppose that $g\in H^s(\Omega)$ for some $0\leq s\leq s_\m^-+2$. By Corollary \ref{c:IntEqn} there exists a unique $u\in L^2(\Omega)$ such that $u = g+k^2\cV_\Omega u$. That $u\in H^s(\Omega)$ follows from repeated applications of the latter equation and \eqref{e:MP1}.  

The operator $\mathcal{V}_\Omega^*$ was defined as $\mathcal{V}_\Omega^* := |_\Omega \circ \cM_\m \circ \cV \circ \mathcal{Z}_\Omega$.
By 
the discussion after \eqref{e:HzeEmb}, 
$\cZ_\Omega:H^s(\Omega)\to H^s(\R^n)$ is bounded for $0\leq s\leq s_\Omega^-$. 
Furthermore, if $0\leq t\leq s_{\m,\Omega}^-$ then there exists a %
function 
$m$ 
such that 
$\m=m\chi_\Omega$, 
with multiplication by $m$ a continuous operator from 
$H^{t}_{\rm loc}(\R^n)$ to $H^{t}(\R^n)$.
Then $|_\Omega \circ \cM_\m =|_\Omega \circ \cM_m$, where $\cM_m$ denotes the operator of multiplication by $m$, so that 
$|_\Omega\circ \cM_\m:H^t_{\rm loc}(\R^n) \to H^t(\Omega)$ for $0\leq t\leq s_{\m,\Omega}^-$. 
On the other hand, 
$\cZ_\Omega:H^s_{\rm ze}(\Omega)\to H^s(\R^n)$ is bounded for all $s\geq 0$, and by \eqref{e:KOmega} the assumption $|K\setminus\Omega|=0$ implies that $\cZ_\Omega\circ|_\Omega\circ\cM_\m=\cM_\m:H^s_{\rm loc}(\R^n)\to H^s(\R^n)$ is continuous for $0\leq s\leq s_\m^-$. 
Then \eqref{e:MP3} and \eqref{e:MP4} follow by combining these observations with \eqref{e:NewtReg}. 
For \eqref{e:MP7} and \eqref{e:MP8}, by Corollary \ref{c:IntEqn} we know that 
for every $\tau \in L^2(\Omega)$ 
there exists a unique $\sigma\in L^2(\Omega)$ such that $\sigma = \tau +k^2\cV_\Omega^* \tau$. Then \eqref{e:MP7}/\eqref{e:MP8} follow from repeated applications of the latter equation and 
\eqref{e:MP3}/\eqref{e:MP4}.  
\end{proof}

The following example presents 
the classical case when $\m\in C^\infty_{\rm comp}(\R^n)$ (as in, e.g., \cite{anderson2024fast}). Examples with non-smooth $\m$ will be considered in \S\ref{s:IFS}. 

\begin{example}
\label{r:Cinfty}
Suppose that $\m\in C^\infty_{\rm comp}(\R^n)$ with $\Im[\m] \geq 0$, and that $\Omega$ is a bounded open set with $|K\setminus \Omega|=0$. Then the following operators are boundedly invertible:
\begin{align}
\label{}
&\mathcal{I}-k^2\mathcal{V}_\Omega:H^{s}(\Omega)\to H^{s}(\Omega), 
&&  
0\leq s<\infty,\label{e:MP5a}\\ 
&\mathcal{I}-k^2\mathcal{V}_\Omega^*:H^{s}(\Omega)\to H^{s}(\Omega),
&&  
0\leq s\leq s_{\Omega}^-+2\leq 5/2,\label{e:MP7a}\\
&\mathcal{I}-k^2\cV_\Omega^*:H_{\rm ze}^s(\Omega)\to H_{\rm ze}^s(\Omega),
&&  
0\leq s<\infty\label{e:MP8a}.
\end{align}

\end{example}

\subsection{Scattering problems and total field formulation}
\label{s:scatt}
Let $u^{\rm i}$ be an \textit{incident field}, by which we mean a 
function that is $C^\infty$ and satisfies $\Delta u^{\rm i}+k^{2}u^{\rm i}=0$ in some open set $D^{\rm i}\subset \R^n$ such that $K\subset D^{\rm i}$.
Then we refer to the FSP \eqref{e:FSP}-\eqref{e:SRC} with 
\begin{align}
\label{e:fscatt}
f:=k^{2}\m u^{\rm i}
\end{align} 
as the \emph{scattering problem} for $u^{\rm i}$. 
Assuming that $\Im[\m] \geq 0$,  
we refer to the resulting unique solution $u^{\rm s}$ of \eqref{e:FSP}-\eqref{e:SRC} as the \emph{scattered field}, and note that  
the \emph{total field} $u^{\rm t}:=u^{\rm i}+u^{\rm s}$ satisfies the equation $\Delta u^{\rm t} +k^{2}(1+\m)u^{\rm t}=0$ in $D^{\rm i}$. 
Examples of incident fields include a plane wave $u^{\rm i}(x) := e^{ik\vartheta\cdot x}$, where $\vartheta\in \R^n$ is a unit direction vector, for which $D^{\rm i}=\R^n$, and a point source $u^{\rm i}(x):=\Phi(x,x_0)$ for some $x_0\in \R^n\setminus K$, for which $D^{\rm i}=\R^n\setminus\{x_0\}$.

In this context, 
given a bounded measurable set 
$\Omega$ such that 
$\overline{\Omega}\subset D^{\rm i}$,   
by substituting 
\eqref{e:fscatt} 
into the RLSE \eqref{e:IntEqn} and adding $u^{\rm i}|_\Omega$ to both sides we find that  
$u:=u^{\rm t}|_\Omega = \tilde{u}+u^{\rm i}|_\Omega$, the restriction of the total field to $\Omega$, is an element of $L^2(\Omega)$ and satisfies the 
RLSE with a modified right-hand side:%
\begin{align}
\label{e:IntEqn3}
(\mathcal{I} - k^2\mathcal{V}_\Omega) u  = 
u^{\rm i}|_\Omega.
\end{align}
If $|K\setminus \Omega|=0$, then, by Corollary \ref{c:IntEqn}, \eqref{e:IntEqn3} has a unique solution $u$, and, by Proposition \ref{p:IntEqn}, the resulting fields can be written in terms of $u$ using \eqref{e:usDef}, which gives
\begin{align}
\label{e:usRep}
u^{\rm s}=k^2\mathcal{V}(\cM_{\m}(\mathcal{Z}_\Omega(u)))\qquad \text{in }\R^n,\\
\label{e:utRep}
u^{\rm t}:=u^{\rm i}+k^2\mathcal{V}(\cM_{\m}(\mathcal{Z}_\Omega(u)))\qquad \text{in }D^{\rm i}.
\end{align}
An attraction of the total field formulation \eqref{e:IntEqn3}-\eqref{e:usRep} for the RLSE 
over the scattered field formulation 
\eqref{e:IntEqn}-\eqref{e:usDef}  
is that we avoid the need to evaluate $\cV f$. 
For completeness we note that under the assumption of \eqref{e:fscatt} the RILSE \eqref{e:RILSE} becomes 
\begin{align}
\label{e:RILSE2}
(\mathcal{I} - k^2\mathcal{V}_\Omega^*)\sigma = 
k^2\m|_\Omega u^{\rm i}|_\Omega.
\end{align}

\section{The volume integral method (VIM)}
\label{s:VIM}
In \S\ref{s:Restr} we showed that the FSP/LSE can be formulated equivalently in terms of either the RLSE or the RILSE. Either formulation can be used as the starting point for a numerical discretisation, but for brevity we shall focus only on discretisation of the RLSE 
\begin{align}
\label{e:IntEqn2a}
(\mathcal{I} - k^2\mathcal{V}_\Omega)u = 
g.
\end{align}
We assume henceforth that
\begin{align}
\label{e:nAssumption}
\Im[\m](x)\geq 0 \text{ a.e.\ on }\R^n,
\end{align}
and that %
\begin{align}
\label{e:OmegaAssumption}
\Omega\subset\R^n \text{ is bounded and open, with } 
|K\setminus \Omega|=0.
\end{align}
The following result is an immediate consequence of Corollary \ref{c:IntEqn}.

\begin{corollary}
[Variational problem for RLSE]
Under the
assumptions \eqref{e:nAssumption} and \eqref{e:OmegaAssumption}, 
the integral equation \eqref{e:IntEqn2a} has a unique solution $u \in L^2(\Omega)$, which can be characterized as the unique solution of the 
following 
variational problem: %
Find 
$u \in L^2(\Omega)$ such that
\begin{equation}\label{e:VarProb2}
b(u, v)=F(v), \quad \forall v \in L^{2}(\Omega),	
\end{equation}
where, for $u,v\in L^2(\Omega)$,
\begin{equation*}
\begin{aligned}
b(u, v) := 	\left(\left(\mathcal{I}-k^2\mathcal{V}_{\Omega}\right) u, v \right)_{L^2(\Omega)}  
=\int_{\Omega}\left(u(x) -  k^2\int_{\Omega} \Phi(x,y)\m(y)u(y)\,\mathrm{d} y\right) \overline{v(x)}\, \mathrm{d} x
\end{aligned}
\end{equation*}
and 
\begin{equation*}
F(v):= \left(g, v\right)_{L^2(\Omega)}=\int_{\Omega} g(x) \overline{v(x)} \,\mathrm{d} x.
\end{equation*}
\end{corollary}

To solve \eqref{e:VarProb2} we use a Galerkin volume integral method (VIM) based on piecewise-polynomial approximation on a 
mesh of $\Omega$.

\begin{definition}[Mesh]
\label{d:Mesh}
Let $\Omega\subset\R^n$ be non-empty, bounded and open.
We say that a finite collection $\cT=\{T\}$ of non-empty bounded open subsets of $\Omega$ is a \emph{mesh} of $\Omega$ if 
\[ 
T\cap T'= \emptyset \text{ for }T\neq T' \quad \text{and} \quad
\bigg| \Omega \setminus \bigcup_{T\in\cT}  T \bigg|=0. \]
We define the meshwidth of $\cT$ by $h:=\max_{T\in\cT}\diam(T)$. 
\end{definition}
Given 
$\Omega$ and $h>0$ we can always find a mesh $\cT$ of $\Omega$ of meshwidth $\leq h$, e.g.\ 
by considering the intersection with $\Omega$ of a set of dyadic cubes with diameter $\leq h$. 

\begin{definition}[Piecewise polynomial space $V_N$]
\label{d:Poly}
Given a mesh $\cT$ of $\Omega$ and $p\in \N_0$ we denote the space of piecewise polynomial functions of total degree at most $p$ on $\cT$ by
\begin{align}
\label{e:VhDef}
V_N:= \{v\in L^2(\Omega):v|_{ T} \text{ is a polynomial of total degree $\leq p$ for each } T\in\cT\}\subset L^2(\Omega),
\end{align}
and we denote by $P_N:L^2(\Omega)\to V_N$ the $L^2$-orthogonal projection onto $V_N$. 
\end{definition}

\begin{definition}[Galerkin VIM]
Our Galerkin VIM is as follows: find 
$u_N\in V_N $ such that
\begin{equation}\label{e:VIM}
b(u_N, v_N)=F(v_N), \quad \forall v_N \in V_N ,	
\end{equation}
or, equivalently, 
\begin{align}
\label{e:VIMProj}
(\cI - k^2 P_N \cV_\Omega)u_N = P_N g.
\end{align}
\end{definition}

Let $\{\psi_j\}_{j=1}^N$ be an $L^2$-orthonormal basis of $V_N $.  
Then, writing $u_N=\sum_{j=1}^N c_j\psi_j$, the VIM \eqref{e:VIM} is equivalent to 
solving for $\bc=(c_1,\ldots,c_N)^T\in \C^N$
the linear system 
\begin{align}
\label{e:LinSys}
A\bc = \bg,
\end{align}
where $A=(A_{ij})_{i,j=1}^N \in \C^{N\times N}$ and $\bg=(g_1,\ldots,g_N)^T\in \C^N$ are given by
\begin{align}
\label{e:LinSys2}
A_{ij} = \delta_{ij} - k^2\int_{ \Omega}\int_{ \Omega}\Phi(x,y)\m(y)\psi_j(y)\overline{\psi_i(x)}\,\rd y\rd x, 
\qquad 
g_{i} = \int_{\Omega}g(x)\overline{\psi_i(x)}\,\rd x.
\end{align}

\begin{remark}[Scattered and total field approximations]
\label{r:Scatt}
In the scattering problem 
discussed in \S\ref{s:scatt}, we have $g=u^{\rm i}|_\Omega$ and $u=u^{\rm t}|_\Omega$. 
Then if $u_N$ solves \eqref{e:VIM} 
we can insert $u_N$ into \eqref{e:usRep} to obtain approximations $u^{\rm s}_N$ and $u^{\rm t}_N$ of the scattered field $u^{\rm s}$ and the total field $u^{\rm t}$, viz.
\begin{align}
\label{e:ushDef}
u^{\rm s}_N:=k^2\mathcal{V}(\cM(\mathcal{Z}_\Omega(u_N))),\qquad \text{in }\R^n,\\
\label{e:uthDef}
u^{\rm t}_N:=u^{\rm i}+k^2\mathcal{V}(\cM(\mathcal{Z}_\Omega(u_N))),\qquad \text{in }D^{\rm i}.
\end{align}
\end{remark}

\subsection{Abstract semi-discrete error analysis}
\label{s:Error}

We now carry out a semi-discrete error analysis of our Galerkin VIM for arbitrary meshes of arbitrary $\Omega$ satisfying \eqref{e:OmegaAssumption}. 
For efficiency of presentation we shall present 
results simultaneously for both 
the $h$- and $p$-versions of our Galerkin VIM, which are defined as follows:
\begin{align}
\text{\underline{$h$-version}:}& \text{ fix $p\in \N_0$ and consider a sequence of meshes of $\Omega$}
\label{e:hversion}\\
&\text{ with associated meshwidths $h\to 0$;}
\notag\\
\text{\underline{$p$-version}:}& \text{ fix a mesh of $\Omega$ with some meshwidth $h>0$ and let $p\to\infty$.}\label{e:pversion}
\end{align}

Our analysis follows the classical approach to the analysis of projection methods for integral equations, a historical overview of which can be found in e.g.\ \cite[\S3]{atkinson1997numerical}. However, to apply this approach in the current context we need certain results that are not available in the classical literature. 

We first note the following $h$- and $p$-explicit best approximation results, taken from 
\cite{polyapprox}. (The $p=0$, $r=1$ version of Proposition \ref{p:BA} appeared already in \cite[Cor.~3.25]{dequalsnpaper}.) We emphasize that Proposition \ref{p:BA} holds for arbitrary meshes of arbitrary domains. 

\begin{proposition}[Best approximation results, 
{\cite[Cor.~2.7]{polyapprox}}]
\label{p:BA}
Let $n,r\in \N$ and let $h_0>0$. There exists a constant $C>0$, depending only on $n$, $r$ and $h_0$, such that if $\Omega\subset\R^n$ is any non-empty bounded open set, $\cT=\{ T\}$ is any mesh of $\Omega$ with meshwidth $h\in (0,h_0]$, 
$p\geq r-1$ 
and $0\leq s\leq r$, 
then
\begin{align}
\label{e:BAHs}
\min_{v_N\in V_N }
\|v-v_N\|_{L^2(\Omega)} 
\leq 
\left(\frac{Ch}{p+1}\right)^s
\|v\|_{H^s(\Omega)},
\qquad v\in H^s(\Omega).
\end{align} 
\end{proposition}

\begin{theorem}[Semi-discrete convergence]
\label{t:GalWP}
Let $\m$ and $\Omega$ satisfy \eqref{e:nAssumption} and \eqref{e:OmegaAssumption}, and suppose that $g\in H^s(\Omega)$ for some $0\leq s\leq s_\m^-+2$. Then the unique solution of \eqref{e:IntEqn2a} satisfies $u\in H^s(\Omega)$. 
For sufficiently small $h$ (for the $h$-version) or sufficiently large $p$ (for the $p$-version) the discrete problem \eqref{e:VIM} has a unique solution $u_N\in V_N $ and %
\begin{align}
\label{e:GalErr}
\|u-u_N\|_{L^2(\Omega)}%
=O\left(\frac{h^{\min(p+1,s)}}{(p+1)^{s}}\right), \quad \text{as }\, h\to 0 \; \text{ or } \; p\to \infty.
\end{align}
\end{theorem}
\begin{proof}
As already noted in Lemma \ref{l:Comp} and Corollary \ref{c:IntEqn}, the 
operator 
$\cI-k^2\cV_\Omega:L^2(\Omega)\to L^2(\Omega)$ 
is injective and a compact perturbation of the identity, from which it follows that \eqref{e:IntEqn2a} 
has a unique solution $u\in L^2(\Omega)$. 
The fact that $u\in H^s(\Omega)$ under the stated assumptions on $g$ and $s$ follows by Lemma \ref{l:MapProp2} (specifically, by \eqref{e:MP5}). 
By assumption, either we have a fixed maximum polynomial degree $p$ and a sequence of meshes of $\Omega$ with associated meshwidths $h$ tending to zero, or we have a fixed mesh of some meshwidth $h$ and we consider the limit as the polynomial degree $p\to \infty$. In either case, the associated sequence of approximation spaces $V_N$ 
forms a dense sequence in $L^2(\Omega)$, as follows from the following standard argument. Let $v\in L^2(\Omega)$ and $\eps>0$. Then by the density of $C^\infty_{\rm comp}(\Omega)$ in $L^{2}(\Omega)$ \cite[Corollary 3.5]{mclean} there exists 
$\phi \in C^\infty_{\rm comp}(\Omega)$ such that 
$\|u-\phi\|_{L^2(\Omega)}<\eps/2$, and by 
Proposition \ref{p:BA} we can find $h>0$ sufficiently small and a mesh $\cT$ of meshwidth $\leq h$ (for the $h$-version), or $p$ sufficiently large (for the $p$-version), such that $\|\phi-P_N\phi\|_{L^2(\Omega)}<\eps/2$, so that by the triangle inequality $\|u-P_N\phi\|_{L^2(\Omega)}<\eps$.
Hence by 
\cite[Thm~3.1.1 \& Lem.~3.1.2]{atkinson1997numerical} 
there exist $C_{qo}>0$ and $h_0>0$ (or $p_0\in \N_0$) such that for each $0<h\leq h_0$ (or for each $p\geq p_0$) the operator $\cI-k^2P_N\cV_\Omega$ is invertible 
on $L^2(\Omega)$ 
with 
\begin{align}
\label{e:UnifBd}
\|(\cI-k^2P_N\cV_\Omega)^{-1}\|\leq C_{qo},
\end{align}
the discrete problem \eqref{e:VIM} has a unique solution $u_N\in V_N $, and
\begin{align}
\label{e:QO}
\|u-u_N\|_{L^2(\Omega)}\leq 
C_{qo}
\|u-P_N u\|_{L^2(\Omega)},
\end{align}
from which \eqref{e:GalErr} follows by 
Proposition \ref{p:BA}.
\end{proof}

In applications one is often interested in the evaluation of bounded linear functionals of $u$, rather than $u$ itself. 
We recall that if $J$ is a bounded linear functional on $L^2(\Omega)$ then by the Riesz representation theorem there exists $\jj \in L^2(\Omega)$ such that 
\begin{align}
\label{e:JDef}
J(v) = (v,\overline{\jj }) = \int_\Omega \jj (y)v(y)\,\mathrm{d} y, \qquad v\in L^2(\Omega).
\end{align}
Whenever $J$ is of this form,%
the Cauchy-Schwarz inequality gives 
\begin{align}
\label{e:JErr}
|J(u)-J(u_N)|\leq 
\|u-u_N\|_{L^2(\Omega)}\|\jj \|_{L^2(\Omega)},
\end{align}
so that $J(u_N)$ converges to $J(u)$ at least as fast as $u_N$ converges to $u$.
Under additional assumptions, we can prove %
superconvergence. 

\begin{theorem}[Semi-discrete superconvergence of functionals]
\label{t:SupCon}
Under the assumptions of Theorem \ref{t:GalWP}, 
suppose that either $\jj \in H^t(\Omega)$ for some $0\leq t\leq \min(s_\Omega^-+2,s_{\m,\Omega}^-)$, or $\jj \in H_{\rm ze}^t(\Omega)$ for some $0\leq t\leq s_\m^-$.  
Then
\begin{align}
\label{e:SupCon}
|J(u)-J(u_N)|
=O\left(\frac{h^{\min(p+1,s)+\min(p+1,t)}}{(p+1)^{s+t}}\right), 
\quad \text{as }\, h\to 0 \; \text{ or } \; p\to \infty.
\end{align}
\end{theorem}
\begin{proof}
The case $t=0$ is covered by Theorem \ref{t:GalWP} and \eqref{e:JErr}. 
For the case $t>0$ we introduce the dual problem: given a 
functional $J$ of the form \eqref{e:JDef}, find $\zeta\in L^2(\Omega)$ such that 
\begin{align}
 \label{e:DualProb}
 b(v,\zeta)=J(v), \qquad \forall v\in L^2(\Omega).
 \end{align}
By \eqref{e:Adj}, \eqref{e:DualProb} is equivalent to the equation
\[ (\cI - k^2\cV_\Omega^*) \overline{\zeta} = \jj . \]
Hence, by Lemma \ref{l:MapProp2}, if either 
$\jj \in H^t(\Omega)$ for some $0\leq t\leq \min(s_\Omega^-+2,s_{\m,\Omega}^-)$ or $\jj \in H_{\rm ze}^t(\Omega)$ for some $0\leq t\leq s_\m^-$ then the solution $\zeta$ of \eqref{e:DualProb} satisfies $\zeta\in H^t(\Omega)$ (recall that $H^t_{\rm ze}(\Omega)$ is continuously embedded in $H^t(\Omega)$ for $t\geq 0$). By \eqref{e:VarProb2} and \eqref{e:VIM} we have Galerkin orthogonality, i.e.\ that 
\begin{align}
\label{e:GalOrth}
b(u-u_N,v_N)=0, \qquad \text{for all }v_N\in V_N.
\end{align}
This gives 
\begin{align*}
\label{}
|J(u)-J(u_N)| = |b(u-u_N,\zeta)| &= |b(u-u_N,\zeta-P_N\zeta)| \\
&\leq \|\cI-k^2\cV_\Omega\|_{L^2(\Omega)\to L^2(\Omega)} \|u-u_N\|_{L^2(\Omega)}\|\zeta-P_N\zeta\|_{L^2(\Omega)},
\end{align*}
from which \eqref{e:SupCon} follows by Theorem \ref{t:GalWP} and Proposition \ref{p:BA}.
\end{proof}

By adapting the arguments from \cite{sloan1984iterated} (see also \cite[\S3.4.1]{atkinson1997numerical}) we can also prove superconvergence of the ``iterated'' Galerkin solution. 

\begin{theorem}[Semi-discrete superconvergence of iterated Galerkin solution]
\label{t:SupConI}
Under the assumptions of Theorem \ref{t:GalWP}, %
define the iterated Galerkin solution by 
\begin{align}
\label{e:IterGal}
u_N^{(1)}:= g + k^2\cV_\Omega u_N.
\end{align}
Then 
\begin{align}
\label{e:SupCon2}
\|u-u_N^{(1)}\|_{L^2(\Omega)}
=O\left(\frac{h^{\min(p+1,s)+\min(p+1,t)}}{(p+1)^{s+t}}\right), 
\quad \text{as }\, h\to 0 \; \text{ or } \; p\to \infty,
\end{align}
where 
\[ t=\max(\min(s_\Omega^-+2,s_{\m,\Omega}^-),s_\m^-).\]
\end{theorem}
\begin{proof}
We follow the idea of the proof of \cite[Thm~1]{sloan1984iterated}. 
For ease of presentation we introduce some further notation, defining $\overline{\cV}_\Omega^*:L^2(\Omega)\to L^2(\Omega)$ by $\overline{\cV}_\Omega^*(\phi):=\overline{\cV_\Omega^*(\overline{\phi})}$ for $\phi\in L^2(\Omega)$, which by \eqref{e:Adj} is the adjoint of $\cV_\Omega:L^2(\Omega)\to L^2(\Omega)$. 
From \eqref{e:IntEqn2a} and \eqref{e:IterGal} we have that
\begin{align}
\|u-u_N^{(1)}\|_{L^2(\Omega)} = k^2 \|\cV_\Omega(u-u_N)\|_{L^2(\Omega)} 
&= k^2\sup_{0\neq\phi\in L^2(\Omega)}\frac{|(\cV_\Omega(u-u_N),\phi)|}{\|\phi\|_{L^2(\Omega)}}\notag\\
&= k^2\sup_{0\neq\phi\in L^2(\Omega)}\frac{|(u-u_N,\overline{\cV}_\Omega^*\phi)|}{\|\phi\|_{L^2(\Omega)}}.
\label{e:IterGal3}
\end{align}

For $0\leq t\leq \min(s_\Omega^-+2,s_{\m,\Omega}^-)$, 
by Lemma \ref{l:MapProp2} we know that $\cV_\Omega^*:L^2(\Omega)\to H^{t}(\Omega)$ is bounded, which implies that $\overline{\cV}_\Omega^*:L^2(\Omega)\to H^{t}(\Omega)$ is also bounded, with the same norm. Noting that if $\overline{\cV}_\Omega^*\phi=0$ then $(u-u_N,\overline{\cV}_\Omega^*\phi)=0$, combining the above with \eqref{e:IterGal3} gives 
\begin{align}
\|u-u_N^{(1)}\|_{L^2(\Omega)} 
&\leq
k^2 \sup_{\substack{\phi\in L^2(\Omega)\notag\\\overline{\cV}_\Omega^*\phi \neq 0}}\frac{|(u-u_N,\overline{\cV}_\Omega^*\phi)|}{\|\phi\|_{L^2(\Omega)}}\\
 &\leq
k^2 \|\overline{\cV}_\Omega^*\|^{-1}_{L^2(\Omega)\to H^t(\Omega)}
 \sup_{\substack{\phi\in L^2(\Omega)\notag\\\overline{\cV}_\Omega^*\phi \neq 0}}\frac{|(u-u_N,\overline{\cV}_\Omega^*\phi)|}{\|\overline{\cV}_\Omega^*\phi\|_{H^{t}(\Omega)}}\\
&\leq
k^2 \|\overline{\cV}_\Omega^*\|^{-1}_{L^2(\Omega)\to H^t(\Omega)}
 \sup_{0\neq\psi\in H^{t}(\Omega)}\frac{|(u-u_N,\psi)|}{\|\psi\|_{H^{t}(\Omega)}}.
\label{e:IterGal4}
\end{align}
Under the same assumptions on $t$, Lemma \ref{l:MapProp2} also implies that $\cI - k^2\cV_\Omega^*:H^{t}(\Omega)\to H^{t}(\Omega)$ is invertible, so that also $\cI - k^2\overline{\cV}_\Omega^*:H^{t}(\Omega)\to H^{t}(\Omega)$ is invertible. Hence, for each $\psi\in H^{t}(\Omega)$ there exists $\varphi\in H^{t}(\Omega)$ such that $(\cI - k^2\overline{\cV}_\Omega^*)\varphi = \psi$, and, by the invertibility of $(\cI - k^2\cV_\Omega):L^2(\Omega)\to L^2(\Omega)$ and the Galerkin orthogonality \eqref{e:GalOrth}, %
\begin{align}
(u-u_N,\psi) &= ((\cI - k^2\cV_\Omega)^{-1}(\cI - k^2\cV_\Omega)(u-u_N),\psi) \notag\\
&= ((\cI - k^2\cV_\Omega)(u-u_N),(\cI - k^2\overline{\cV}_\Omega^*)^{-1}\psi) \notag\\
&= ((\cI - k^2\cV_\Omega)(u-u_N),\varphi)\notag\\
&= b(u-u_N,\varphi)\notag\\
&= b(u-u_N,\varphi-P_N\varphi)\notag\\
&\leq \|\cI-k^2\cV_\Omega\|_{L^2(\Omega)\to L^2(\Omega)}\|u-u_N\|_{L^2(\Omega)} \|\varphi-P_N\varphi\|_{L^2(\Omega)}.
\label{e:L2Bound}
\end{align}
Combining this with \eqref{e:IterGal4}, Theorem \ref{t:GalWP} and Proposition \ref{p:BA} gives 
\eqref{e:SupCon2}. 

For $0\leq t\leq s_\m^-$ we argue similarly, albeit in the context of the $H^t_{\rm ze}(\Omega)$ spaces. By Lemma \ref{l:MapProp2} we deduce that $\overline{\cV}_\Omega^*:L^2(\Omega)\to H_{\rm ze}^{t}(\Omega)$ is bounded, and then using \eqref{e:HzeEmb} we obtain the following modified version of \eqref{e:IterGal4}:
\begin{align}
\|u-u_N^{(1)}\|_{L^2(\Omega)} 
&\leq
k^2 \|\overline{\cV}_\Omega^*\|^{-1}_{L^2(\Omega)\to H_{\rm ze}^t(\Omega)}
 \sup_{0\neq\psi\in H_{\rm ze}^{t}(\Omega)}\frac{|(u-u_N,\psi)|}{\|\psi\|_{H_{\rm ze}^{t}(\Omega)}}\notag\\
&\leq k^2 \|\overline{\cV}_\Omega^*\|^{-1}_{L^2(\Omega)\to H_{\rm ze}^t(\Omega)}
 \sup_{0\neq\psi\in H_{\rm ze}^{t}(\Omega)}\frac{|(u-u_N,\psi)|}{\|\psi\|_{H^{t}(\Omega)}}.
\label{e:IterGal4a}
\end{align}
By Lemma \ref{l:MapProp2} we also deduce that 
$\cI - k^2\overline{\cV}_\Omega^*:H_{\rm ze}^{t}(\Omega)\to H_{\rm ze}^{t}(\Omega)$ is invertible. Hence, for each $\psi\in H_{\rm ze}^{t}(\Omega)$ there exists $\varphi\in H_{\rm ze}^{t}(\Omega)\subset H^{t}(\Omega)$ such that $(\cI - k^2\overline{\cV}_\Omega^*)\varphi = \psi$, and then we finish the proof as in the previous case, combining 
\eqref{e:IterGal4a} with \eqref{e:L2Bound}, Theorem \ref{t:GalWP} and Proposition \ref{p:BA} to obtain \eqref{e:SupCon2}. 
\end{proof}

\subsection{Application to the scattering problem}
\label{s:Scatt2}

We now consider the application of the results of \S\ref{s:Error} to the scattering problem 
discussed in \S\ref{s:scatt}, in which we solve \eqref{e:IntEqn2a} with $g:=u^{\rm i}|_\Omega$ for some incident field $u^{\rm i}$, and $u$ represents the restriction to $\Omega$ of the total field $u^{\rm t}$. 
In this context,
two important functionals 
are:
\begin{itemize}
\item 
the \emph{scattered field 
functional} 
$J^{\rm SF}_{x_0}$, 
given for an observation point $x_0\in \R^n$ 
by $J^{\rm SF}_{x_0}(u)=u^{\rm s}(x_0)$, 
which 
(recalling \eqref{e:usRep}) 
takes 
the form 
\eqref{e:JDef} with $\jj (y) = k^2\m(y)\Phi(x_0,y)$;
\item
the \emph{far-field pattern %
functional} 
$J^{\rm FFP}_{\hat{x}}$, 
given for an observation direction $\hat{x}\in \mathbb{S}^{n-1}$ by \eqref{e:JDef} with 
$\jj (y) = ck^{(1+n)/2}\m(y)\re^{-\ri k \hat{x}\cdot y}$ for a constant $c$ depending only on $n$ (see, e.g., \cite[\S2.5, \S3.5]{ck2}); 
\end{itemize}
We also mention that, in this context, the iterated Galerkin solution $u_N^{(1)}$ defined in \eqref{e:IterGal} satisfies $u_N^{(1)}=u^{\rm t}_N|_\Omega = u^{\rm s}_N|_\Omega + u^{\rm i}|_\Omega$, where $u^{\rm t}_N|_\Omega$ and $u^{\rm s}_N|_\Omega$ are defined by \eqref{e:ushDef} and \eqref{e:uthDef}, so 
\begin{align}
\label{e:IterRep}
\|u-u_N^{(1)}\|_{L^2(\Omega)}=\|u^{\rm s}-u^{\rm s}_N\|_{L^2(\Omega)}=
\|J^{\rm SF}_{\cdot}(u)-J^{\rm SF}_{\cdot}(u_N)\|_{L^2(\Omega)}.
\end{align}

\begin{corollary}[Semi-discrete analysis for the scattering problem]
\label{c:Scatt}
Let $\m$ and $\Omega$ satisfy \eqref{e:nAssumption} and \eqref{e:OmegaAssumption}, let $g=u^{\rm i}|_\Omega$ for some incident field $u^{\rm i}$ that is 
$C^\infty$
in an open set $D^{\rm i}$ containing $\overline\Omega$, and let 
\begin{align}
 \label{e:stDef}
 s:= s_\m^-+2 \text{ and }t:=\max(\min(s_\Omega^-+2,s_{\m,\Omega}^-),s_\m^-). 
 \end{align} 
Then:
\begin{itemize}
\item[(i)] The solution of \eqref{e:IntEqn2a} satisfies 
$u\in H^{s}(\Omega)$, 
and for sufficiently small $h$ (for the $h$-version) or sufficiently large $p$ (for the $p$-version) the Galerkin VIM solution $u_N$ exists and satisfies 
\eqref{e:GalErr}; 
\item[(ii)]
The far-field pattern functional $J^{\rm FFP}_{\hat{x}}$ satisfies \eqref{e:SupCon}, 
uniformly for $\hat{x}\in \mathbb{S}^{n-1}$.  
\item[(iii)]
The scattered field functional $J^{\rm SF}_{x_0}$ also satisfies \eqref{e:SupCon}, 
uniformly for $x_0\in \R^n\setminus K_\eps$,  where $K_\eps:=\{x\in \R^n:\dist(x,K)< \eps\}$, for every $\eps>0$.
\item[(iv)]   
$ \|u-u_N^{(1)}\|_{L^2(\Omega)}=\|u^{\rm s}-u^{\rm s}_N\|_{L^2(\Omega)}=
\|J^{\rm SF}_{\cdot}(u)-J^{\rm SF}_{\cdot}(u_N)\|_{L^2(\Omega)}$ satisfies \eqref{e:SupCon2}.
\end{itemize}
\end{corollary}
\begin{proof}
Since the incident field $u^{\rm i}$ is $C^\infty$ in an open set $D^{\rm i}$ containing $\overline\Omega$, we have $g=u^{\rm i}|_\Omega\in H^s(\Omega)$ for every $s\geq 0$, so that part (i) follows by Theorem \ref{t:GalWP}. 

For part (ii) we note first that the function $y\mapsto\Phi^\infty(\hat{x},y):=\re^{-\ri k \hat{x}\cdot y}$ is an element of $H^t(\Omega)$ for all $t\geq 0$, and that for any $t\geq 0$ the norm $\|\Phi^\infty(\hat{x},\cdot)\|_{H^t(\Omega)}$ can be bounded above, uniformly for $\hat{x}\in \mathbb{S}^{n-1}$. 
Hence $\jj (y) = ck^{(1+n)/2}\m(y)\re^{-\ri k \hat{x}\cdot y}$ satisfies 
\begin{align}
\label{e:HtReg}
\jj \in H^t(\Omega) \text{ for all } 0\leq t\leq s_{\m,\Omega}^-, \text{ and } \jj \in H_{\rm ze}^t(\Omega) \text{ for all } 0\leq t\leq s_\m^-,
\end{align}
and, in all cases considered in \eqref{e:HtReg}, $\|\jj \|_{H^t(\Omega)}$ can be bounded above, uniformly in $\hat{x}\in \mathbb{S}^{n-1}$, so that part (ii) follows by Theorem \ref{t:SupCon}. 

For part (iii) we note similarly that, by the smoothness of the function $\Phi(x_0,\cdot)$, 
the function $\jj (y) = k^2\m(y)\Phi(x_0,y)$ satisfies \eqref{e:HtReg}, and that, in all cases considered in \eqref{e:HtReg}, for every $\eps>0$ the norm $\|\jj \|_{H^t(\Omega)}$ can be bounded above, uniformly for $x_0\in \R^n\setminus K_\eps$, so that part (iii) also follows by Theorem \ref{t:SupCon}.

Finally, part (iv) follows by Theorem \ref{t:SupConI}, in light of \eqref{e:IterRep}. 
\end{proof}

\begin{remark}[Pointwise superconvergence of $J^{\rm SF}_{x_0}$ on $K$]
\label{r:SupConF}
Part (iii) of Corollary \ref{c:Scatt} provides pointwise superconvergence estimates for $J^{\rm SF}_{x_0}$ for $x_0\in\R^n\setminus K$.
One can also derive such estimates for $x_0\in K$, but the rate of superconvergence may be reduced compared to the case $x_0\in\R^n\setminus K$ because of the singularity in $\Phi(x_0,\cdot)$.

By standard calculations one can show that if $t<2-n/2$ then $\Phi(x_0,\cdot)\in H^t_{\rm loc}(\R^n)$ for every $x_0\in \R^n$. From this it follows that, for every $x_0\in \R^n$, $\jj (y) = k^2\m(y)\Phi(x_0,y)$ satisfies
\begin{align}
&\jj \in H^t(\Omega) \text{ for all } 0\leq t\leq s_{\m,\Omega}^-,\,\,t<2-n/2,\notag\\ \text{ and } &\jj \in H_{\rm ze}^t(\Omega) \text{ for all } 0\leq t\leq s_\m^-,\,t<2-n/2.
\label{e:HtRegK}
\end{align}
Furthermore, in all cases considered in \eqref{e:HtRegK}, $\|\jj \|_{H^t(\Omega)}$ can be bounded above, uniformly for $x_0\in \R^n$.  
It then follows by Theorem \ref{t:SupCon} that
$J^{\rm SF}_{x_0}$ satisfies \eqref{e:SupCon}, uniformly for $x_0\in \R^n$, with $t$ replaced by any element of the set
\begin{align}
 \label{e:SupConvR}
  \{t:t\leq \max(\min(s_\Omega^-+2,s_{\m,\Omega}^-),s_\m^-),\, t<2-n/2\}.  
 \end{align}

Part (iv) of Corollary \ref{c:Scatt} shows that the restriction $t<2-n/2$ can be removed if one considers $L^2$ convergence of $J^{\rm SF}_{x_0}$ rather than pointwise convergence. 
\end{remark}

In the following example, we discuss the implications of these results for the classical case with smooth $\m$ considered in Example \ref{r:Cinfty}. 
Here, when we mention convergence rates, we have in mind the norms considered in Corollary \ref{c:Scatt}. So for $u$ we are referring to $L^2$ error on $\Omega$, 
for $J^{\rm FFP}_{\hat{x}}$ we are referring to pointwise error, uniformly for $\hat{x}\in \mathbb{S}^{n-1}$, and for $J^{\rm SF}_{x_0}$ we are referring to both pointwise error, uniformly for $x_0\in \R^n\setminus K_\eps$ for any $\eps>0$, and $L^2$ error on $\Omega$.
We note that examples for non-smooth $\m$ will be considered in \S\ref{s:IFS}.
\begin{example}
\label{r:Cinfty2}
Suppose that $\m\in C^\infty_{\rm comp}(\R^n)$ with $\Im[\m] \geq 0$, and that $\Omega$ is a bounded open set with $|K\setminus \Omega|=0$. Then $s$ and $t$ defined by \eqref{e:stDef} satisfy $s=t=\infty$, and hence: 
\begin{itemize}
\item For the $h$-version VIM, we have algebraic convergence, with $O(h^{p+1})$ error for $u$ and $O(h^{2(p+1)})$ error for $J^{\rm FFP}_{\hat{x}}$ and $J^{\rm SF}_{x_0}$;%
\item For the $p$-version VIM, we have super-algebraic convergence, with $O((p+1)^{-\infty})$ error for $u$, $J^{\rm FFP}_{\hat{x}}$ and $J^{\rm SF}_{x_0}$. %
\end{itemize}
\end{example}

\subsection{Abstract fully discrete error analysis}
\label{s:AbstFD}

Suppose that the matrix and right-hand side in \eqref{e:LinSys} are approximated using quadrature rules as $A\approx \tilde{A} \in \C^{N\times N}$ and $\bg\approx \tilde{\bg}\in \C^N$. 
Then we seek $\tilde{\bc}=(\tilde{c}_1,\ldots,\tilde{c}_N)^T\in\C^N$ such that
\begin{align}
\label{e:LinSysFD}
\tilde{A}\tilde{\bc} = \tilde{\bg},
\end{align}
and define the fully discrete approximation 
\begin{align}
\label{e:FDDef}
\tilde{u}_N:=\sum_{j=1}^N \tilde{c}_j \psi_j\in V_N .
\end{align}
Equivalently, we seek $\tilde{u}_N\in V_N $ as the solution of 
\begin{align}
\label{e:uFDVar}
\tilde{b}(\tilde{u}_N,v_N) = \tilde{F}(v_N), \qquad v_N\in V_N ,
\end{align}
where $\tilde{b}(\tilde{u}_N,v_N) := \overline{\bd}^T \tilde{A} \tilde{\bc}$ and $\tilde{F}(v_N):=\overline{\bd}^T \tilde{\bg}$ for $v_N=\sum_{j=1}^N d_j \psi_j\in V_N $ and $\bd=(d_1,\ldots,d_N)^T$.

Having computed $\tilde{u}_N$, one can also consider quadrature approximations to linear functionals $J$ of the form \eqref{e:JDef} applied to $\tilde{u}_N$. 
We note that the action of a such a functional $J$ 
on 
$v_N=\sum_{j=1}^N d_j \psi_j\in V_N $ can be written as $J(v_N)=\bj^T \bd$ where $\bj\in \C^N$ is defined by 
\begin{align}
\label{e:JSD}
j_i = J(\psi_i)=
\int_{ T_i} \jj (y) \psi_i(y)\,\rd y , \qquad i=1,\ldots,N.
\end{align} 
If $\bj$ is approximated using quadrature as $\bj\approx \tilde{\bj}\in \C^N$ then we can obtain a fully discrete functional approximation $\tilde{J}(v_N):=\tilde{\bj}^T \bd$. 

The error in the fully discrete solution and fully discrete functional evaluations 
can be analysed using Strang's Lemma (e.g.\ \cite[Thm~4.2.11]{ss11}). 
Furthermore, as is standard for second kind integral equations (cf.\ e.g.\ \cite[\S3.6]{atkinson1997numerical}), the fully discrete system is well conditioned. 

\begin{theorem}[Fully discrete convergence, conditioning and superconvergence]
\label{t:FDErr}
Under the assumptions of Theorem \ref{t:GalWP}: 
\begin{itemize}
\item[(i)]
If $\|A-\tilde{A}\|_2\to 0$ and $\|\bg-\tilde{\bg}\|_2\to 0$ as $h\to 0$ (for the $h$-version) or $p\to \infty$ (for the $p$-version), then for sufficiently small $h$, or sufficiently large $p$, 
the system \eqref{e:LinSysFD} is invertible, with
\begin{align}
\label{e:ConvFD}
\|u-\tilde{u}_N\|_{L^2(\Omega)}\to 0, \quad \text{as }\, h\to 0 \; \text{ or } \; p\to \infty,
\end{align} 
and
\begin{align}
\label{e:Cond}
\kappa_2(\tilde{A}) = \|\tilde{A}\|_2\|\tilde{A}^{-1}\|_2 =O(1), \quad \text{as }\, h\to 0 \; \text{ or } \; p\to \infty;
\end{align}
\item[(ii)]
If 
\begin{align}
\label{e:AfOh}
\|A-\tilde{A}\|_2\;\text{ and } \; \|\bg-\tilde{\bg}\|_2 \; \text{ are }\; O\left(\frac{h^{\min(p+1,s)}}{(p+1)^{s}}\right), \quad \text{as }\, h\to 0 \; \text{ or } \; p\to \infty,
\end{align}
then 
\begin{align}
\label{e:FDErr}
\|u-\tilde{u}_N\|_{L^2(\Omega)}=O\left(\frac{h^{\min(p+1,s)}}{(p+1)^{s}}\right), \quad \text{as }\, h\to 0 \; \text{ or } \; p\to \infty; 
\end{align}
\item[(iii)]
Under the assumptions of Theorem \ref{t:SupCon}, if 
\begin{align}
\label{e:AfOh2}
\|A-\tilde{A}\|_2, \;\;
\|\bg-\tilde{\bg}\|_2\;\text{ and } \; 
\|\bj-\tilde{\bj}\|_2
\; \text{ are }\; 
O\left(\frac{h^{\min(p+1,s)+\min(p+1,t)}}{(p+1)^{s+t}}\right), 
\quad \text{as }\, h\to 0 \; \text{ or } \; p\to \infty,
\end{align}
then 
\begin{align}
\label{e:FDErr2}
|J(u)-\tilde{J}(\tilde{u}_N)|=O\left(\frac{h^{\min(p+1,s)+\min(p+1,t)}}{(p+1)^{s+t}}\right), 
\quad \text{as }\, h\to 0 \; \text{ or } \; p\to \infty. 
\end{align}
\end{itemize}
\end{theorem}
\begin{proof}
For (i) and (ii), we first note that, for $v_N,w_N\in V_N $,
\begin{align}
\label{e:Strang1}
|b(v_N,w_N)-\tilde{b}(v_N,w_N)| \leq \|A-\tilde{A}\|_2 \|v_N\|_{L^2(\Omega)}\|w_N\|_{L^2(\Omega)}, 
\end{align}
and 
\begin{align}
\label{e:Strang2}
|F(v_N)-\tilde{F}(v_N)| \leq \|\bg-\tilde{\bg}\|_2 \|v_N\|_{L^2(\Omega)}. 
\end{align}
Then, under the assumptions of (i), by \cite[Thm~4.2.11]{ss11} 
the system \eqref{e:LinSysFD} is invertible 
for sufficiently small $h$ (or sufficiently large $p$) 
and
\begin{align}
\|u-\tilde{u}_N\|_{L^2(\Omega)}
&\lesssim \min_{v_N\in V_N } \left( \|u-v_N\|_{L^2(\Omega)} + \|A-\tilde{A}\|_2\|v_N\|_{L^2(\Omega)} \right) + \|\bg-\tilde{\bg}\|_2 \notag\\
&\leq \|u-P_Nu\|_{L^2(\Omega)} + \|A-\tilde{A}\|_2\|P_Nu\|_{L^2(\Omega)}  +\|\bg-\tilde{\bg}\|_2\notag\\
&\leq \|u-P_Nu\|_{L^2(\Omega)} + \|A-\tilde{A}\|_2\|u\|_{L^2(\Omega)}  +\|\bg-\tilde{\bg}\|_2,
\label{e:Strang3}
\end{align}
by the fact that $\|P_Nu\|_{L^2(\Omega)}\leq \|u\|_{L^2(\Omega)}$. 
From \eqref{e:Strang3} and Proposition \ref{p:BA} it follows that if $\|A-\tilde{A}\|_2\to 0$ and $\|\bg-\tilde{\bg}\|_2\to 0$ then $\|u-\tilde{u}_N\|_{L^2(\Omega)}\to 0$, and also that \eqref{e:AfOh} implies \eqref{e:FDErr}. 

To prove the conditioning statement \eqref{e:Cond}, by standard arguments (e.g.\ \cite[Thm 3.1.1, Lem.~3.1.2, \S3.6]{atkinson1997numerical}) one can show that, for sufficiently small $h$ (or sufficiently large $p$), both $\|A\|_2$ and $\|A^{-1}\|_2$ are bounded above, independently of $h$ (respectively, of $p$). That the same is true for $\|\tilde{A}\|_2$ and $\|\tilde{A}^{-1}\|_2$ when $\|A-\tilde{A}\|_2\to 0$ and $\|\bg-\tilde{\bg}\|_2\to 0$ then follows by a standard perturbation argument, using e.g.\ \cite[Thm~A.1]{atkinson1997numerical}. 

For (iii), to prove \eqref{e:FDErr2} we first use the triangle inequality to bound
\begin{align}
\label{e:Strang4}
|J(u)-\tilde{J}(\tilde{u}_N)| \leq & \underbrace{|J(u)-J(u_N)|}_{(I)} 
+\underbrace{|J(u_N)-\tilde{J}(u_N)|}_{(II)}
+ \underbrace{|\tilde{J}(u_N)-\tilde{J}(\tilde{u}_N)|}_{(III)}.
\end{align}
Let $\mu:=\frac{h^{\min(p+1,s)+\min(p+1,t)}}{(p+1)^{s+t}}$. 
We know (I) is $O(\mu)$ by Theorem \ref{t:SupCon}. For (II) we have, using \eqref{e:VIMProj}, \eqref{e:UnifBd}, the fact that $\|P_Ng\|_{L^2(\Omega)}\leq \|g\|_{L^2(\Omega)}$, and the third assumption in \eqref{e:AfOh2}, that
\begin{align}
\label{e:Strang5}
|J(u_N)-\tilde{J}(u_N)|\leq \|\bj-\tilde{\bj}\|_2 \|u_N\|_{L^2(\Omega)}  
\lesssim
\|\bj-\tilde{\bj}\|_2\|g\|_{L^2(\Omega)} = O(\mu).
\end{align}
Finally, for (III) we can bound
\begin{align}
\label{e:Strang6}
|\tilde{J}(u_N)-\tilde{J}(\tilde{u}_N)| \leq \|\tilde{\bj}\|_2 \|u_N-\tilde{u}_N\|_{L^2(\Omega)}
\lesssim \|\jj \|_{L^2(\Omega)}\|u_N-\tilde{u}_N\|_{L^2(\Omega)},
\end{align}
where we used the third assumption in \eqref{e:AfOh2} to bound $\|\tilde{\bj}\|_2\leq\|\bj\|_2+\|\bj-\tilde{\bj}\|_2\lesssim \|\bj\|_2=\|\jj \|_{L^2(\Omega)}$. 
Then, by the argument at the top of p235 of \cite{ss11} (which uses \cite[Eqn~(4.156)]{ss11}), 
for sufficiently small $h$ (or sufficiently large $p$) we can use \eqref{e:Strang1} and \eqref{e:Strang2} to bound
\begin{align}
\|u_N-\tilde{u}_N\|_{L^2(\Omega)} 
&\lesssim \sup_{0\neq v_N\in V_N }\frac{|b(u_N,v_N) - \tilde{b}(u_N,v_N)| + |F(v_N)-\tilde{F}(v_N)|}{\|v_N\|_{L^2(\Omega)}}\notag \\
&\leq\left(\|A-\tilde{A}\|_2 \|u_N\|_{L^2(\Omega)} + \|\bg-\tilde{\bg}\|_2 \right),
\label{e:Strang7}
\end{align}
which is $O(\mu)$ by the first two assumptions in \eqref{e:AfOh2} and by \eqref{e:UnifBd}, which gives uniform boundedness of $\|u_N\|_{L^2(\Omega)}$ as in \eqref{e:Strang5}.
\end{proof}

\section{Iterated Function Systems and $n$-attractors}
\label{s:IFS}

We now specialise to a class of inhomogeneities that we will refer to as ``$n$-attractors'' (following \cite{dequalsnpaper}). We first recall the notion of an iterated function system (IFS) and some associated concepts. 
Our sources include the seminal paper by Hutchinson \cite{HUTCHINSONJOHNE.1981FaSS}, and Falconer's textbook
\cite{falconer2014fractal}. 
We note that the standard definition of an IFS involves general contraction mappings; here we restrict our attention to contracting similarities. 
In the following, $\dimH(E)$ denotes the Hausdorff dimension of a set $E\subset\R^n$ (see \cite[\S3]{falconer2014fractal}). 

\begin{definition}[Iterated function system (IFS)]
\label{d:IFS}
Let $n,M\in \mathbb{N}$ with $M\geq 2$. An \emph{iterated function system} (IFS) 
is a collection $\{s_{1},s_{2},\ldots, s_{M}\}$ of maps such that, for each $i=1,\ldots,M,$ $s_{i}:\mathbb{R}^{n}\rightarrow \mathbb{R}^{n}$ is a contracting similarity of the form
\begin{equation}
\label{e:IFS}
	s_{i}(x) = \rho_{i} A_{i}x +\delta_{i}, \quad x\in \mathbb{R}^{n},
\end{equation}
for some contraction factor $\rho_{i} \in (0,1)$, orthogonal matrix $A_{i}\in \mathbb{R}^{n\times n}$, and translation vector $\delta_{i}\in \mathbb{R}^{n}$. 
\end{definition}

\begin{definition}[Open set condition (OSC)]
\label{d:osc}
We say that the IFS $\{s_1,\ldots,s_M\}$ satisfies the \emph{open set condition} (OSC) if there exists a non-empty open set $O\subset \mathbb{R}^{n}$ such that $s_{i}(O)\cap s_{j}(O)=\emptyset$ for all $i,j=1,\ldots M$ such that $i\neq j$, and 
$\cup_{i=1}^{M}s_{i}(O)\subset O $. 
\end{definition}

\begin{theorem}[IFS attractor and fractal dimension]
\label{t:IFS}
Let $\{s_{i}\}_{i=1}^{M}$ be an IFS. 
Then there exists a unique non-empty compact set $K\subset\R^n$, called the \emph{attractor} of the IFS, such that 
\begin{align}
\label{e:Attractor}
K = \bigcup_{i=1}^M s_i(K).
\end{align}
Furthermore, if the OSC holds then the Hausdorff dimension $d=\dimH(K)$ of $K$ equals the \emph{similarity dimension} of the IFS, defined to be the unique positive solution of the equation
\begin{equation}\label{e:SimDim}
	\sum_{i=1}^{M}\rho_{i}^{d} =1.
\end{equation}
\end{theorem}
\begin{proof}
See \cite[Thm~3.1, Thm~5.3]{HUTCHINSONJOHNE.1981FaSS}. 
\end{proof}

\begin{definition}[$n$-attractor]
\label{d:nattract}
We say a non-empty compact $K\subset\R^n$ is a \emph{$n$-attractor} if it is the attractor of an IFS satisfying the OSC for which the similarity dimension (i.e.\ the solution of \eqref{e:SimDim}), and hence also the Hausdorff dimension of $K$, equals $n$.
\end{definition}

The concept of an $n$-attractor is closely related to the concept of a ``fractal tiling'' --- see, e.g., \cite{bandt1991self,grochenig1994self,
strichartz1999geometry,Keesling:99,BaVi:14}. 
In particular, the set of ``homogeneous'' $n$-attractors (i.e.,\ those for which all contraction factors are equal) coincides with the set of compact ``$M$-rep tiles'', as defined in \cite{bandt1991self} (see \cite[Rem.~2.3]{dequalsnpaper}). 
Some examples of $n$-attractors will be studied in \S\ref{s:Examp}. 
Some basic properties of $n$-attractors can be found in \cite[Prop.~2.8]{dequalsnpaper}; here we highlight some properties that will be particularly relevant for us.
\begin{proposition}
\label{p:IFSProp}
Let $K\subset\R^n$ be an $n$-attractor and let $\Omega:=K^\circ$. Then:
$|K|>0$;
$\Omega\neq \emptyset$; 
$\overline{\Omega}=K$; 
$\partial\Omega=\partial K$; 
$\dimH(\partial K)=\dimH(\partial \Omega)\in[n-1,n)$, so in particular $|\partial K|=|\partial\Omega|=0$; 
the OSC is satisfied with $O=\Omega$;
and $|s_i(K)\cap s_j(K)|=0$ for $i\neq j\in\{1,\ldots,M\}$. 
\end{proposition}

Also of relevance will be the following result, proved in \cite{dequalsnpaper}, about multiplier properties of the characteristic function of the interior of an $n$-attractor.
We note that, for the case considered in Proposition \ref{p:sOmega}, determining whether 
$\chi_\Omega\in \mathfrak{M}(H^{s_\Omega}(\R^n))$ 
appears to be 
an open problem.

\begin{proposition}[{\cite[Cor.~3.4]{dequalsnpaper}}]
\label{p:sOmega}
Let $\Omega\subset\R^n$ denote the interior of an $n$-attractor. Then 
$s_\Omega = (n-d)/2$ and 
$\chi_\Omega\in \mathfrak{M}(H^s(\R^n))$ 
for $|s|< (n-d)/2$,
where $d:=\dim_{\rm H}(\partial\Omega)\in[n-1,n)$ is the Hausdorff dimension of $\partial\Omega$. 
\end{proposition}

By combining Proposition \ref{p:sOmega} with results from \S\ref{s:VIM} we can derive semi-discrete convergence results for our Galerkin VIM when the inhomogeneity $K$ is an $n$-attractor. As an example we consider the case where $\m$ is constant on the inhomogeneity, as in Figure \ref{f:Scatt}. 
The following result follows by combining Corollary \ref{c:Scatt},  Proposition \ref{p:sOmega} and Remark \ref{r:MCont}\eqref{i:chiOmega}.

\begin{corollary}[Semi-discrete convergence on an $n$-attractor]
\label{c:App3}
Let $\Omega$ be interior of an $n$-attractor, let $\m = m\chi_\Omega$ for some $m\in \C\setminus\{0\}$ with $\Im[m] \geq 0$, and let $d:=\dim_{\rm H}(\partial\Omega)$.  Then 
\begin{align}
\label{e:App4}
s_\m = s_\Omega = \tfrac{n-d}{2} \quad \text{and} \quad s_{\m,\Omega}=\infty,
\end{align}
the following operators are boundedly invertible:
\begin{align}
\label{}
&\mathcal{I}-k^2\mathcal{V}_\Omega:H^{s}(\Omega)\to H^{s}(\Omega), 
&&  
0\leq s\leq(\tfrac{n-d}{2})^-+2,\label{e:MP5b}\\ 
&\mathcal{I}-k^2\mathcal{V}_\Omega^*:H^{s}(\Omega)\to H^{s}(\Omega),
&&  
0\leq s\leq(\tfrac{n-d}{2})^-+2,\label{e:MP7b}\\
&\mathcal{I}-k^2\cV_\Omega^*:H_{\rm ze}^s(\Omega)\to H_{\rm ze}^s(\Omega),
&&  
0\leq s\leq(\tfrac{n-d}{2})^-,\label{e:MP8b}
\end{align}
and the statements of Corollary \ref{c:Scatt} hold with
\begin{align}
 \label{e:App5}
 s=t=(\tfrac{n-d}{2})^-+ 2. 
 \end{align} 
In particular, for the $h$-version VIM on $\Omega$ 
we have $O(h^{\min(p+1,(\tfrac{n-d}{2})^-+2)})$ error for $u$ and $O(h^{2\min(p+1,(\tfrac{n-d}{2})^-+2)})$ error for $J^{\rm FFP}_{\hat{x}}$ and $J^{\rm SF}_{x_0}$ (in the sense explained before Example \ref{r:Cinfty2}). 
\end{corollary}

\begin{remark}[Semi-discrete convergence on a bounding box]
\label{r:App4}
Let $\m$, $\Omega$ and $d$ be as in Corollary \ref{c:App3}.
Let $\Omega'$ be a bounded 
open set (e.g.\ a bounding box) such that $\overline{\Omega}\subset\Omega'$. 
Then by Proposition \ref{p:sOmega} and Remark \ref{r:MCont}\eqref{i:Bounding}
\begin{align}
\label{e:App4}
s_{\m,\Omega'}=s_\Omega=\tfrac{n-d}{2},
\end{align}
the following operators are boundedly invertible:
\begin{align}
\label{}
&\mathcal{I}-k^2\mathcal{V}_{\Omega'}:H^{s}(\Omega')\to H^{s}(\Omega'), 
&&  
0\leq s\leq(\tfrac{n-d}{2})^-+2,\label{e:MP5c}\\ 
&\mathcal{I}-k^2\mathcal{V}_{\Omega'}^*:H^{s}(\Omega')\to H^{s}(\Omega'),
&&  
0\leq s\leq(\tfrac{n-d}{2})^-,\label{e:MP7c}\\
&\mathcal{I}-k^2\cV_{\Omega'}^*:H_{\rm ze}^s(\Omega')\to H_{\rm ze}^s(\Omega'),
&&  
0\leq s\leq(\tfrac{n-d}{2})^-\label{e:MP8c}.
\end{align}
and, with $\Omega$ replaced by $\Omega'$, 
the statements of Corollary \ref{c:Scatt} hold with 
\begin{align}
 \label{e:App6}
 s=(\tfrac{n-d}{2})^-+ 2 \quad \text{and} \quad t=(\tfrac{n-d}{2})^-. 
 \end{align}
In particular, for the $h$-version VIM on $\Omega'$  
we have $O(h^{\min(p+1,(\tfrac{n-d}{2})^-+2)})$ error for $u$ and $O(h^{\min(p+1,(\tfrac{n-d}{2})^-+2)+(\tfrac{n-d}{2})^-})$ error for $J^{\rm FFP}_{\hat{x}}$ and $J^{\rm SF}_{x_0}$ (in the sense explained before Example \ref{r:Cinfty2}).  

Comparing this with the results in Corollary \ref{c:App3}, we see that, for this choice of $\m$, applying our Galerkin VIM on $\Omega$ (via a 
``geometry-conforming'' method, as illustrated in Figure \ref{f:Disc}(c)) leads to an improved rate of convergence for linear functionals, compared to that obtained by applying the same approach on $\Omega'$ (via an 
``bounding box'' method, as illustrated in Figure \ref{f:Disc}(a)).
\end{remark}

We now consider the practical implementation of our Galerkin VIM on an $n$-attractor. In \S\ref{s:IFSMesh} we discuss the generation of geometry-conforming meshes, and in \S\ref{s:IFSquad} we discuss numerical quadrature rules. 
We combine the results of these two sections  in \S\ref{s:IFSFD} (specifically, in Theorem \ref{t:FD}), where we provide a fully discrete error analysis of our method in the case of piecewise constant approximations. 

\subsection{Meshing an $n$-attractor}
\label{s:IFSMesh}

For an $n$-attractor $K$, the fact that $|K|>0$ and $|s_i(K)\cap s_j(K)|=0$ for $i\neq j\in\{1,\ldots,M\}$ is often referred to as the ``self-similarity'' of $K$ (see, e.g., \cite[\S5]{HUTCHINSONJOHNE.1981FaSS}). Combined with \eqref{e:Attractor}, 
it means that $K$ can be decomposed into a union of $M$ smaller copies of $K$, with zero measure overlaps. 
This induces a corresponding decomposition of the interior $\Omega:=K^\circ$ into $M$ disjoint smaller copies of $\Omega$, with
\begin{align}
\label{e:OmegaDecomp}
\Bigg| \Omega\setminus \bigcup_{i=1}^M s_i(\Omega)\Bigg|=0, 
\qquad s_i(\Omega)\cap s_j(\Omega)=\emptyset, \,\,i\neq j\in\{1,\ldots,M\}.
\end{align}
Hence $\{s_i(\Omega)\}_{i=1}^M$ 
constitutes 
a mesh of $\Omega$ in the sense of Definition \ref{d:Mesh}. 
The 
decomposition 
\eqref{e:OmegaDecomp} 
can be applied recursively to generate 
meshes of $\Omega$ comprising elements similar to $\Omega$ and 
of arbitrarily small diameter. To describe such meshes we adopt the vector index notation of \cite{HUTCHINSONJOHNE.1981FaSS}.

\begin{definition}[Vector indices]%
\label{d:VectInd}
Let $\{s_{i}\}_{i=1}^{M}$ be an IFS on $\mathbb{R}^{n}$. 
Given $l\in \N$ and $\bm{m} = (m_{1},\ldots,m_{l})\in \left\{1, \ldots, M \right\}^{l}$ 
we define 
\begin{equation}\label{e:sm}
s_{\bm{m}}(x) \coloneqq s_{m_1}\circ s_{m_2} \circ \ldots \circ s_{m_l}(x),\quad x\in \mathbb{R}^{n},
\end{equation}
which is a contracting similarity with contraction factor
\begin{equation*}
	\rho_{\bm{m}} := \prod_{i=1}^{l}\rho_{m_{i}},
\end{equation*} 
where, for each $i=1,\ldots,l$, $\rho_{m_i}$ denotes the contraction factor of the map $s_{m_i}$.
Given $E\subset \mathbb{R}^{n}$ we define $E_{\bm{m}} := s_{\bm{m}}(E)$. When 
$l=1$, so that 
$\bm{m}=(i)$ for some $i\in\{1,\ldots N\}$, we denote $s_{\bm{m}}(E) =s_{i}(A)$ simply by $E_{i}$. For convenience we also define $s_\emptyset$ to be the identity map on $\R^n$, $\rho_\emptyset:=1$ and $E_\emptyset:=E$. 
Finally, we define 
\begin{equation}
\cI := \{\emptyset\}\cup \left(\bigcup_{l\in \mathbb{N}} \left\{1, \ldots, M \right\}^{l}\right).
\end{equation}
\end{definition}

The following trivial scaling results will be used repeatedly in our analysis.

\begin{lemma}[Scaling properties]
\label{l:Scaling}
Let $K$ be an $n$-attractor with IFS $\{s_{i}\}_{i=1}^{M}$, and let $\Omega:=K^\circ$. 
Then, for any 
$\bm{m} \in \mathcal{I}$, 
\begin{equation}
\label{e:Kmdiam}
\operatorname{diam}\left(\Omega_{\bm{m}}\right)=\rho_{\bm{m}} \operatorname{diam}(\Omega)
\end{equation}
and 
\begin{equation}\label{e:Scaling}
	\left|\Omega_{\bm{m}}\right|=\rho_{\bm{m}}^{n}|\Omega|
	=\left(\frac{\diam(\Omega_{\bm{m}})}{\diam(\Omega)}\right)^n|\Omega|.
\end{equation}
\end{lemma} 

Our focus will be on the quasi-uniform meshes described in the following definition.  

\begin{definition}[$L_{h}$ index set and $L_{h}$ mesh]
\label{d:Lh}
Let $K$ be an $n$-attractor and let $\Omega:=K^\circ$. Given 
$h>0$ 
we define the $L_{h}$ \emph{index set} 
by $L_h(\Omega)=\{\emptyset\}$ when $h\geq \diam(\Omega)$ and by 
\begin{equation}\label{lhindexset}
L_{h}(\Omega) \coloneqq \{ \bm{m}=(m_1,\dots,m_{l}) \in \mathcal{I} :\operatorname{diam}(\Omega_{\bm{m}})\leq h \mbox{ and } \operatorname{diam}\left(\Omega_{(m_{1},\dots,m_{l-1})}\right)>h\} 
\end{equation}
when $h<\diam(\Omega)$, 
replacing $\Omega_{(m_{1},\dots,m_{l-1})}$ by $\Omega_\emptyset=\Omega$ when $l=1$.
We refer to $\{\Omega_{\bm{m}}\}_{\bm{m}\in L_h(\Omega)}$ as the $L_h$ \emph{mesh} of $\Omega$. 
\end{definition}

\begin{proposition}[$L_h$ mesh properties]
\label{p:LhMesh}
Let $K$ be an $n$-attractor, $\Omega:=K^\circ$ and %
\begin{align}
\label{e:rhominDef}
\rho_{\rm min}:=\min_{i=1,\ldots,M}\rho_i.
\end{align}
Suppose that $h\in(0,\diam(\Omega)]$. 
Then the $L_h$ mesh $\{\Omega_{\bm{m}}\}_{\bm{m}\in L_h(\Omega)}$ is a 
mesh of $\Omega$ (in the sense of Definition \ref{d:Mesh}). 
Furthermore, the elements of the mesh satisfy 
\begin{align}
\label{e:LhProp1}
\rho_{\rm min}h < \diam(\Omega_{\bm{m}})\leq h, \qquad \bm{m}\in L_h(\Omega),
\end{align}
and
\begin{align}
\label{e:tauDef}
\left(\frac{\rho_{\rm min}h}{\diam(\Omega)}\right)^{n} |\Omega| 
\leq |\Omega_{\bm{m}}|\leq 
\left(\frac{h}{\diam(\Omega)}\right)^{n} |\Omega|,
\qquad \bm{m}\in L_h(\Omega),
\end{align}
and the number of elements $N=\#L_h(\Omega)$ in the mesh satisfies
\begin{equation}
\label{e:LhProp2}
\left(\frac{\diam(\Omega)}{h}\right)^{n}\leq N < \left(\frac{\diam(\Omega)}{\rho_{\rm min}h}\right)^{n}.
\end{equation}
\end{proposition}
\begin{proof}
That $\{\Omega_{\bm{m}}\}_{\bm{m}\in L_h(\Omega)}$ defines a mesh of $\Omega$ follows by repeated applications of  \eqref{e:OmegaDecomp} - see the results in \cite[Section 5.1.1]{joly2024high}. 
The bounds in \eqref{e:LhProp1} 
follow immediately from the definition of $L_h(\Omega)$, and \eqref{e:tauDef} and \eqref{e:LhProp2} follow from \eqref{e:LhProp1} and \eqref{e:Scaling}, since
\[ |\Omega| = \sum_{\bm{m}\in L_h(\Omega)} |\Omega_{\bm{m}}| = |\Omega|\sum_{\bm{m}\in L_h(\Omega)}\left(\frac{\diam(\Omega_{\bm{m}})}{\diam(\Omega)}\right)^n, \qquad \text{so that }\quad 1= \sum_{\bm{m}\in L_h(\Omega)}\left(\frac{\diam(\Omega_{\bm{m}})}{\diam(\Omega)}\right)^n,\]
which by \eqref{e:LhProp1} gives that
\[
N\left(\frac{\rho_{\rm min}h}{\diam(\Omega)}\right)^n 
\leq 1=\sum_{\bm{m}\in L_h(\Omega)}\left(\frac{\diam(\Omega_{\bm{m}})}{\diam(\Omega)}\right)^n
\leq N\left(\frac{h}{\diam(\Omega)}\right)^n,
\]
from which \eqref{e:LhProp2} follows immediately.
\end{proof}

In practice, the $L_{h}$ mesh can be constructed via repeated subdivision of $\Omega$ using the IFS structure - see the algorithm presented in \cite[Section 5.1.1]{joly2024high}, which is guaranteed to terminate in a finite number of steps which grows logarithmically with decreasing $h$ \cite[Remark 5.1]{joly2024high}. 
Furthermore, because the diameters of elements in the $L_h$ meshes can only take values from a discrete set (namely the diameter of $\Omega$ multiplied by all possible products of powers of the contraction factors, cf.\ \eqref{e:Kmdiam}), there are only countably many distinct $L_h$ meshes. Indeed, for any $n$-attractor there exists a decreasing sequence $(h_l)_{l\in\N_0}$, tending to zero as $l\to\infty$, and with $h_0=\diam(\Omega)$, such that, for each $l\in\N$, the $L_{h_l}$ mesh is a strict refinement of the $L_{h_{l-1}}$ mesh and 
\begin{align}
\label{e:hkProp}
L_{h}(\Omega)=L_{h_l}(\Omega) \quad \text{for }h\in [h_{l}, h_{l-1}). 
\end{align}
We emphasize that the $L_h$ meshes comprise elements whose boundaries are in general fractal. 
Some examples were presented in Figure \ref{f:Meshes}.

\subsection{Numerical quadrature 
- piecewise constant case}
\label{s:IFSquad}

To implement our VIM 
we need to evaluate the integrals \eqref{e:LinSys2} arising in the linear system \eqref{e:LinSys}. 
For simplicity we focus on the case where the refractive index is constant on the inhomogeneity, with 
\begin{align}
\label{e:nConst}
\m=m\chi_\Omega,\qquad \text{for some }m\in\C \text{ with }\Im[m]\geq 0.
\end{align}
We also restrict attention to the case of the $h$-version VIM with piecewise constant basis functions ($p=0$). 

Under these assumptions, if the elements of the $L_h$ mesh $\{\Omega_{\bm{m}}\}_{\bm{m}\in L_h(\Omega)}$ are indexed as $\{\Omega_{\bm{m}(1)},\ldots,\Omega_{\bm{m}(N)}\}$, where $N=\#L_h(\Omega)$ (recall 
\eqref{e:LhProp2}), 
we work with the basis of $V_N$ 
\begin{align}
\label{}
\psi_j  = | \Omega_{\bm{m}(j)}|^{-1/2}\chi^{}_{\Omega_{\bm{m}(j)}}, \qquad j=1,\ldots,N.
\end{align}
The integrals \eqref{e:LinSys2} to be evaluated are then, for $i,j=1,\ldots,N$, 
\begin{align}
\label{e:LinSys2a}
A_{ij} = \delta_{ij} - \frac{mk^2}
{|\Omega_{\bm{m}(i)}|^{1/2}|\Omega_{\bm{m}(j)}|^{1/2}}
\int_{\Omega_{\bm{m}(i)}}\int_{\Omega_{\bm{m}(j)}}\Phi(x,y)\,\rd y\rd x, 
\qquad 
g_{i} = \frac{1}
{|\Omega_{\bm{m}(i)}|^{1/2}}
\int_{\Omega_{\bm{m}(i)}}u^{\rm i}(x)\,\rd x.
\end{align}

Our approach to the evaluation of \eqref{e:LinSys2a} will be to combine singularity subtraction with the self-similarity techniques of \cite{disjoint,nondisjoint} to write the (weakly) singular integrals arising in the matrix $A$ in terms of ``regular'' integrals (i.e.\ integrals with a smooth integrand). These regular integrals can be evaluated using the composite barycentre rule. 
The following definitions and results arise from special cases of the definitions and results found in \cite{disjoint,nondisjoint}, where more general measures on fractals were considered. We note that the higher order quadrature rules 
recently proposed and analysed in \cite{joly2024high} 
could be used in place of the composite barycentre rule in the following, %
but, at least for the piecewise constant case, the composite barycentre rule is sufficient for our purposes. %

\subsubsection{The composite barycentre rule}
\label{s:Bary}

We begin by reviewing the definition and convergence theory for the composite barycentre rule considered in \cite{disjoint}.
Suppose that $K$ is an $n$-attractor and let $\Omega:=K^\circ$. Given a measurable function $F:\Hull(\Omega)\to \mathbb{C}$ and some 
$h_q>0$ 
we define the composite barycentre rule
approximation of the integral
\begin{equation}\label{e:Intf}
I_{\Omega}[F]\coloneqq \int_{\Omega}F(x)\,\mathrm{d}x 
\end{equation}
by 
\begin{equation}
\label{e:Quadf}
Q^{h_{q}}_{\Omega}[F] \coloneqq \sum_{\bm{m}\in L_{h_q}(\Omega)}|\Omega_{\bm{m}}|\,F(x_{\bm{m}}),
\end{equation}
where, for $\bm{m}\in \cI$, $x_{\bm{m}}$ is the barycentre of $\Omega_{\bm{m}}$, defined by
\begin{align}
\label{}
x_{\bm{m}} \coloneqq \frac{1}{|\Omega_{\bm{m}}|}\int_{\Omega_{\bm{m}}} x \,\mathrm{d} x.
\end{align}
We recall that formulas for the weight $|\Omega_{\bm{m}}|$ appearing in \eqref{e:Quadf} were given in \eqref{e:Scaling}. 
One can prove that (see \cite[Proposition 3.3]{disjoint})
\[x_{\bm{m}}=s_{\bm{m}}(x^{}_{\Omega}),\] 
where $x^{}_{\Omega}$ is the barycentre of $\Omega$, which is given by the formula
\begin{equation}\label{e:Bary}
x^{}_{\Omega} = \left[ I - \sum_{j=1}^{M}\rho_{j}^{n+1}A_{j} \right]^{-1}\left( \sum_{j=1}^{M}\rho_{j}^{n}\delta_{j} \right),
\end{equation}
where $I$ denotes the $n\times n$ identity matrix.

Now suppose that $K'$ is another $n$-attractor, $\Omega':=(K')^\circ$, and $G:\Hull(\Omega)\times\Hull(\Omega')\to \C$ is measurable. 
Then for the double integral 
\begin{equation}\label{e:Intg}
I_{\Omega,\Omega'}[G]\coloneqq \int_{\Omega}\int_{\Omega'}G(x,x')\,\mathrm{d}x \mathrm{d}x' 
\end{equation}
we define the iterated composite barycentre approximation 
\begin{equation}
\label{e:Quadg}
Q^{h_{q}}_{\Omega,\Omega'}[G] \coloneqq \sum_{\bm{m}\in L_{h_q}(\Omega)}\sum_{\bm{m}'\in L_{h_q}(\Omega')}|\Omega_{\bm{m}}||\Omega_{\bm{m}'}'| \, G(x_{\bm{m}},x'_{\bm{m}'}).
\end{equation}

\noindent The following theorem collects results from \cite[Section 3]{disjoint}, providing error estimates for the above approximations under appropriate smoothness assumptions on the integrands. 
Here, and henceforth, we define, for $h_q>0$, 
\[ \Hull(\Omega,h_q):=\bigcup_{\bm{m}\in L_{h_q}(\Omega)}\Hull(\Omega_{\bm{m}})\subset \Hull(\Omega).\]

\begin{proposition}
\label{p:BaryErr}
Let $h_q>0$ and 
suppose that 
$f:U \rightarrow \mathbb{C}$ is twice differentiable, where $U$ is some open set containing 
$\Hull(\Omega,h_q)$.
Then 
\begin{equation}\label{e:ErrSing}
	\left|I_{\Omega}[F]-Q^{h_q}_{\Omega}[F]\right| \leq \frac{nh_q^{2}}{2} |\Omega|
\sup_{\substack{x\in \Hull(\Omega,h_q)}}
\max _{\substack{\alpha \in \mathbb{N}_0^n \\|\alpha|=2}}\left|D^\alpha F(x)\right|.
\end{equation}
Suppose that 
$G:U\times U'\to \C$ is twice differentiable, where 
$U$ is as above and  
$U'$ is some open set containing 
$\Hull(\Omega',h_q)$. 
Then 
\begin{equation}\label{e:ErrDoub}
\left|I_{\Omega,\Omega'}[G]-Q^{h_q}_{\Omega,\Omega'}[G]\right| \leq 2n h_q^{2}|\Omega||\Omega'| 
\sup_{\substack{x\in \Hull(\Omega,h_q)\\{y\in \Hull(\Omega',h_q) }}}
\max _{\substack{\alpha \in \mathbb{N}_0^{2n} \\|\alpha|=2}}\left|D^\alpha G(x,y)\right|.
\end{equation}
Under the weaker assumption that 
$G:\Hull(\Omega,h_q)\times \Hull(\Omega',h_q)\rightarrow \mathbb{C}$ 
is Lipschitz continuous with Lipschitz constant 
$\mathcal{L}_{\Hull(\Omega,h_q)\times \Hull(\Omega',h_q)}[G]$, 
we have 
\begin{equation}
\label{e:ErrDoub2}
\left|I_{\Omega,\Omega'}[G]-Q^{h_q}_{\Omega,\Omega'}[G]\right| \leq \sqrt{2} h_q|\Omega||\Omega'| 
\mathcal{L}_{\Hull(\Omega,h_q)\times \Hull(\Omega',h_q)}[G].
\end{equation}
\end{proposition}
\begin{proof}
See \cite[Thms 3.6 \& 3.7]{disjoint}\footnote{We note that the final bound of \cite[Theorem 3.7(iii)]{disjoint} is missing a factor of $\mu(K)\mu(K')$.}.
\end{proof}

In the implementation of our VIM we will apply these quadrature rules to integrals over mesh elements. For every $\bm{m}\in \cI$ the mesh element $\Omega_{\bm{m}}$ is the interior of $K_{\bm{m}}$, which is the attractor of the IFS $\{s_{\bm{m}}\circ s_{i}\circ s_{\bm{m}}^{-1}\}_{i=1}^{M}$, which satisfies the OSC because $\{s_i\}_{i=1}^M$ does. Furthermore, for any $\bm{p}\in\cI$ we have that $(\Omega_{\bm{m}})_{\bm{p}}=\Omega_{(\bm{m},\bm{p})}$, where $(\bm{m},\bm{p})$ denotes the concatenation of $\bm{m}$ and $\bm{p}$ (with $(\emptyset,\bm{p}):=\bm{p}$ and $(\bm{m},\emptyset):=\bm{m}$). Also, $L_{h_q}(\Omega_{\bm{m}})=L_{h_q/\rho_{\bm{m}}}(\Omega)$. Hence, 
for $F:\Hull(\Omega_{\bm{m}})\to \C$ 
and 
$h_q>0$ 
the composite barycentre rule approximation to the integral $I_{\Omega_{\bm{m}}}[F]$ is given by
\[ Q^{h_q}_{\Omega_{\bm{m}}}[F]=\sum_{\bm{p}\in L_{h_q/\rho_{\bm{m}}}(\Omega)}|\Omega_{(\bm{m},\bm{p})}|\,F(x_{(\bm{m},\bm{p})}).\]
Similarly, if $\bm{n}\in \cI$ 
then for $G:\Hull(\Omega_{\bm{m}})\times \Hull(\Omega_{\bm{n}})\to \C$ our approximation to the integral $I_{\Omega_{\bm{m}},\Omega_{\bm{n}}}[G]$ is given by
\[ Q^{h_q}_{\Omega_{\bm{m}},\Omega_{\bm{n}}}[G]=\sum_{\bm{p}\in L_{h_q/\rho_{\bm{m}}}(\Omega)}\sum_{\bm{q}\in L_{h_q/\rho_{\bm{n}}}(\Omega)}|\Omega_{(\bm{m},\bm{p})}|\,|\Omega_{(\bm{n},\bm{q})}|\,G(x_{(\bm{m},\bm{p})},x_{(\bm{n},\bm{q})}).\]

For the evaluation of the matrix entry $A_{ij}$ in \eqref{e:LinSys2a} we adopt different approaches, depending on whether 
$K_{\bm{m}(i)}=\overline{\Omega_{\bm{m}(i)}}$ and 
$K_{\bm{m}(j)}= \overline{\Omega_{\bm{m}(j)}}$ are disjoint or not. 
To distinguish these cases we introduce the following index sets:
\[ 
\cI_{\rm reg} := \{(i,j)\in \{1,\ldots,N\}^2: K_{\bm{m}(i)}\cap K_{\bm{m}(j)}= \emptyset\},
\qquad 
\cI_{\rm sing} := 
\{1,\ldots,N\}^2 \setminus \cI_{\rm reg}.
\] 

\subsubsection{Regular interactions}
\label{s:Reg}
If $(i,j)\in \cI_{\rm reg}$ the integral in \eqref{e:LinSys2a} is regular, and amenable to direct approximation by the composite barycentre rule.  
With $h_r>0$ we define our approximation of $A_{ij}$ in this case by
\begin{align}
\label{e:DisQuad}
\tilde{A}_{ij}:=\delta_{ij} - \frac{mk^2}
{|\Omega_{\bm{m}(i)}|^{1/2}|\Omega_{\bm{m}(j)}|^{1/2}}
Q^{h_r}_{\Omega_{\bm{m}(i)},\Omega_{\bm{m}(j)}}[\Phi].
\end{align}
To derive error estimates we need to assume a certain separation property for the $L_h$ meshes. 
We recall that, for a given $\bm{m}\in \cI$, $K_{\bm{m}}$ equals the closure of the open set $\Omega_{\bm{m}}$.

\begin{assumption}[Separation assumption]
\label{a:Sep}
We assume that: 
\begin{itemize}
\item[(i)] If $\bm{m},\bm{n}\in \cI$ and $K_{\bm{m}}\cap K_{\bm{n}}=\emptyset$ then 
$\Hull(K_{\bm{m}})\cap \Hull(K_{\bm{n}})=\emptyset$;
\item[(ii)] There exists a constant $C_{\rm sep}>0$ such that if $h\in(0,\diam(K)]$, $\bm{m},\bm{n}\in L_h(\Omega)$ and $K_{\bm{m}}\cap K_{\bm{n}}=\emptyset$ then $\dist(\Hull(K_{\bm{m}}),\Hull(K_{\bm{n}}))\geq C_{\rm sep}h$.
\end{itemize}
\end{assumption}

We also need a regularity result about the fundamental solution $\Phi$. In \cite[Lemma 5.1]{disjoint} it was shown that for $n=2,3$ there is a constant $C_n>0$, depending only on $n$, such that 
\begin{equation}\label{e:PhiBnd}
	\left| D^\alpha \Phi(x,y)\right| \leq C_n \frac{\left(1+(k|x-y|)^{(n+1)/2}\right)}{|x-y|^{n}}, \qquad \forall x\neq y, \quad \forall \alpha \in \mathbb{N}_0^{2n}, \, |\alpha|=2.
\end{equation}

\begin{lemma}[Disjoint quadrature error]
\label{l:DisQuad}
Suppose that 
Assumption \ref{a:Sep}(i) holds.
Then there exists a constant $C_{\rm dis}>0$ such that if $h\in (0,\diam(K)]$ and $\bm{m}(i),\bm{m}(j)\in L_{h}(K)$ are such that $K_{\bm{m}{(i)}}\cap K_{\bm{m}{(j)} } = \emptyset$, then $A_{ij}$ and $\tilde{A}_{ij}$ defined by \eqref{e:LinSys2a} and \eqref{e:DisQuad} satisfy 
\begin{equation}\label{e:DisErr}
|A_{i j}-\tilde{A}_{ij}| \leq 
\frac{C_{\rm dis}h_r^2h^n}
{\left(d_{ij}\right)^n}, 
\qquad 
\forall h_r>0,
\end{equation}
where
\[ d_{ij}:=\dist(\Hull(K_{\bm{m}(i)}),\Hull(K_{\bm{m}(j)})).\]
\end{lemma}
\begin{proof}
We note first that
\begin{align}
\label{e:DQ1}
|A_{ij}-\tilde{A}_{ij}| =  \frac{|m|k^2}
{|\Omega_{\bm{m}(i)}|^{1/2}|\Omega_{\bm{m}(j)}|^{1/2}}
\left|I_{\Omega_{\bm{m}(i)},\Omega_{\bm{m}(j)}}[\Phi]-Q^{h_j}_{\Omega_{\bm{m}(i)},\Omega_{\bm{m}(j)}}[\Phi]\right|.
\end{align}
Assumption \ref{a:Sep}(i)
implies the existence of two disjoint open sets $U,U' \subset \R^{n}$ such that 
$\Hull(\Omega_{\bm{m}{(i)}}) \subset U$ and $\Hull(\Omega_{\bm{m}{(j)}}) \subset U'$.   
By the smoothness of $\Phi(x,y)$ away from $x=y$ we have that $\Phi:U\times U'\to\C$ is infinitely differentiable. Furthermore, by 
\eqref{e:PhiBnd} and the fact that the function $r\mapsto (1+r^{(n+1)/2})/r^n$ is decreasing on $(0,\infty)$, we have that
\[ 
\sup_{\substack{x\in \Hull(\Omega_{\bm{m}(i)})\\{y\in \Hull(\Omega_{\bm{m}(j)})}}}
\max _{\substack{\alpha \in \mathbb{N}_0^{2n} \\|\alpha|=2}}\left|D^\alpha \Phi(x,y)\right| \leq 
C_n\frac{\left(1+(kd_{ij})^{(n+1)/2}\right)}{(d_{ij})^{n}}.
\] 
Hence by \eqref{e:ErrDoub} we have that, for $h_r>0$,
\begin{align}
\label{e:DQ2}
|A_{ij}-\tilde{A}_{ij}| \leq  \frac{2n|m|k^2h_r^2|\Omega_{\bm{m}(i)}|^{1/2}|\Omega_{\bm{m}(j)}|^{1/2}
C_n\left(1+(kd_{ij})^{(n+1)/2}\right)}{(d_{ij})^{n}},
\end{align}
from which \eqref{e:DisErr} follows, by \eqref{e:Scaling},  \eqref{e:LhProp1} and the fact that $d_{ij}\leq \diam(\Omega)$, with
\begin{align*}
 \label{}
C_{\rm dis} = \frac{2n|m|k^2|\Omega|C_n(1+(k\diam(\Omega))^{(n+1)/2})}{(\diam{\Omega})^n}.
 \end{align*} 
\end{proof}

\subsubsection{Singular interactions}
\label{s:Sing}

If $(i,j)\in \cI_{\rm sing}$ then integral in \eqref{e:LinSys2a} is singular, and to approximate $A_{ij}$ we adopt a singularity subtraction approach. By decomposing the fundamental solution as
\begin{align}
\label{e:PhiDec}
\Phi(x,y)= c_n\Phi_0(x,y) + \Phi_*(x,y),
\end{align}
where
\begin{align}
c_n:=
\begin{cases}
-\dfrac{1}{2\pi}, & n=2,\\[2mm]
\dfrac{1}{4\pi}, & n=3,
\end{cases}
\qquad\qquad
\label{e:Phi0Def}
\Phi_0(x,y):=
\begin{cases}
\log(|x-y|), & n=2,\\[2mm]
\dfrac{1}{|x-y|}, & n=3,
\end{cases}
\end{align} 
and 
\begin{align}
 \label{e:PhiStDef}
\Phi_*(x,y):= \Phi(x,y)-c_n\Phi_0(x,y),
\end{align}
we can decompose $A_{ij}$ as
\begin{align}
\label{e:SingDec}
A_{ij} = \delta_{ij} - 
mk^2(c_nG^0_{ij} + G^*_{ij}),
\end{align}
where 
\begin{align}
\label{e:GijDef}
G^0_{ij}:= 
\frac{I_{\Omega_{\bm{m}(i)},\Omega_{\bm{m}(j)}}[\Phi_0]}{|\Omega_{\bm{m}(i)}|^{1/2}|\Omega_{\bm{m}(j)}|^{1/2}}
\qquad
\text{and} 
\qquad
G^*_{ij}:= 
\frac{I_{\Omega_{\bm{m}(i)},\Omega_{\bm{m}(j)}}[\Phi_*]}{|\Omega_{\bm{m}(i)}|^{1/2}|\Omega_{\bm{m}(j)}|^{1/2}}.
\end{align}

For the evaluation of $G^*_{ij}$ we apply the composite barycentre rule with a quadrature mesh width $h_*>0$ to obtain an approximation
\begin{align}
\label{e:GstijQ}
\tilde{G}^*_{ij}:= 
\frac{Q^{h_*}_{\Omega_{\bm{m}(i)},\Omega_{\bm{m}(j)}}[\Phi_*]}{|\Omega_{\bm{m}(i)}|^{1/2}|\Omega_{\bm{m}(j)}|^{1/2}}.
\end{align}
By \cite[Eqn~(62)]{disjoint} there exists $C_*>0$, independent of $k$, such that
\begin{align}
\label{e:GstBnd}
|D^\alpha\Phi_*(x,y)|\leq C_*k, \qquad \forall x,y\in\R^2, \quad \forall \alpha \in \mathbb{N}_0^{4}, \, |\alpha|=1,
\end{align}
which allows us to prove the following error estimate.

\begin{lemma}
\label{l:Gst}
There exists a constant $C_{\rm sing}^*>0$ such that 
if $h\in (0,\diam(K)]$ and $\bm{m}(i),\bm{m}(j)\in L_h(K)$ then 
$G^*_{ij}$ and $\tilde{G}^*_{ij}$ defined by \eqref{e:GijDef} and \eqref{e:GstijQ} satisfy
\begin{align}
\label{e:GstErr}
|G^*_{ij}-\tilde{G}^*_{ij}| \leq 
C^*_{\rm sing} h^2 h_*, \qquad \forall h_*>0.
\end{align}

\end{lemma}
\begin{proof}
The bound \eqref{e:GstBnd} implies that $\Phi_*$ is Lipschitz on $\R^2\times \R^2$ with Lipschitz constant $C^*k$, and hence by \eqref{e:ErrDoub2} we have that
\begin{align}
\label{e:GstErr1}
|G^*_{ij}-\tilde{G}^*_{ij}| \leq 
\sqrt{2} \,|\Omega_{\bm{m}(i)}|^{1/2}|\Omega_{\bm{m}(j)}|^{1/2}C_* kh_*,
\end{align}
from which \eqref{e:GstErr} follows, by \eqref{e:Scaling} and \eqref{e:LhProp1}, with
\[ 
C^*_{\rm sing} = \sqrt{2} C_*k \frac{|\Omega|}{(\diam(\Omega))^2}.
\]
\end{proof}

For the evaluation of $G^0_{ij}$ we adopt the approach of \cite{nondisjoint}.
For this approach to be applicable we need to assume that the IFS defining the $n$-attractor $K$ is such that the singular integral arising in $G^0_{ij}$ can be written explicitly, via a suitable coordinate transformation, in terms of one of a finite number of ``canonical singular integrals'', which should satisfy an invertible linear system of equations, the right-hand side of which should involve a finite number of ``canonical regular integrals'' that are amenable to approximation using the composite barycentre rule. An algorithm for deriving such a linear system was presented in \cite[\S4]{nondisjoint}.  
We make the above statements more precise in the following assumption, referring the reader to \cite{nondisjoint} for full details. 

\begin{assumption}
[Canonical integrals]
\label{a:nondisj}
We assume that, when applied to the IFS $\{s_{i}\}_{i=1}^{M}$ defining the $n$-attractor $K$, the algorithm detailed in \cite[\S4]{nondisjoint} terminates, producing a linear system of the form 	
\begin{equation}\label{E:CanLS}
	\cA\bm{s}=\cB\bm{r}+\bm{t},
\end{equation}
where, for some $n_s,n_r \in \N$,
\begin{itemize}
\item $\cA\in \mathbb{C}^{n_s \times n_s}$ is known and invertible;
\item $\cB\in \mathbb{C}^{n_s\times n_r}$ is known;
\item $\bm{t}\in \mathbb{C}^{n_{s}}$ is known (and present only when $n=2$); 
\item $\bm{s}=\left( I_{\Omega_{\bm{a}_1},\Omega_{\bm{b}_1}}[\Phi_0],\ldots,I_{\Omega_{\bm{a}_{n_s}},\Omega_{\bm{b}_{n_s}}}[\Phi_0] \right)\in \mathbb{C}^{n_{s}}$ 
is a vector of ``canonical singular integrals'', 
with $\bm{a}_\nu,\bm{b}_\nu\in  \mathcal{I}$ and $K_{\bm{a}_\nu}\cap K_{\bm{b}_\nu}\neq \emptyset$ for each $\nu=1,\ldots,n_s$;
\item $\bm{r}=\left( I_{\Omega_{\bm{c}_{1}},\Omega_{\bm{d}_{1}}}[\Phi_0],\ldots,I_{\Omega_{\bm{c}_{n_r}},\Omega_{\bm{d}_{n_r}}}[\Phi_0] \right)\in \mathbb{C}^{n_r}$ 
is a vector of ``canonical regular integrals'', 
with $\bm{c}_\eta,\bm{d}_\eta\in  \mathcal{I}$ and $K_{\bm{c}_\eta}\cap K_{\bm{d}_\eta}=\emptyset$ for each $\eta=1,\ldots,n_r$.
\end{itemize}
We also assume that, if $h\in(0,\diam(K)]$ and $\bm{m}{(i)},\bm{m}{(j)}\in L_{h}(K)$ are such that $K_{\bm{m}{(i)}}\cap K_{\bm{m}{(j)}}\neq \emptyset$, then there exists 
$\nu\in \{ 1,\ldots,n_s \}$, a permutation $\Pi:\{i,j\}\to\{i,j\}$, and a similarity $T:\R^n\to\R^n$ such that $T(K_{\bm{a}_\nu}) = K_{\bm{m}(\Pi(i))}$, $T(K_{\bm{b}_\nu}) = K_{\bm{m}(\Pi(j))}$. 
\end{assumption}
To explain the significance of Assumption \ref{a:nondisj}, we note that if $\bm{m}{(i)},\bm{m}{(j)}\in L_{h}(K)$ are such that $K_{\bm{m}{(i)}}\cap K_{\bm{m}{(j)}}\neq \emptyset$, and $\nu$, $\Pi$ and $T$ are as in the final part of Assumption \ref{a:nondisj}, then 
$G^0_{ij}$ can be represented in terms of the canonical singular integral $I_{\Omega_{\bm{a}_\nu},\Omega_{\bm{b}_\nu}}[\Phi_0]$ via a change of variables. Explicitly, suppose that the similarity $T$ from Assumption \ref{a:nondisj} takes the form $T(x)=\rho_T A_T x + \delta_T$ for some $\rho_T>0$, orthogonal $A_T\in \R^{n\times n}$ and $\delta_T\in \R^n$. Then elementary arguments {(cf.\ \cite[Prop.~3.2]{nondisjoint})} show that
\begin{align}
\label{e:Isom2}
\rho_T =\frac{\rho_{\bm{m}{(\Pi(i))}}}{\rho_{\bm{a}_\nu}}=\frac{\rho_{\bm{m}{(\Pi(j))}}}{\rho_{\bm{b}_\nu}}
\end{align}
and
\begin{align}
\label{}
\label{e:G0ijRep}
G^0_{ij} =
\frac{1}{|\Omega_{\bm{m}{(i)}}|^{1/2}|\Omega_{\bm{m}{(j)}}|^{1/2}}
\begin{cases}
\rho_T^{4}\left(I_{\Omega_{\bm{a}_\nu},\Omega_{\bm{b}_\nu}}[\Phi_0] + 
|\Omega_{\bm{a}_\nu}||\Omega_{\bm{b}_\nu}|
\log(\rho_T)\right), & n=2,\\[4mm]
\rho_T^{5} I_{\Omega_{\bm{a}_\nu},\Omega_{\bm{b}_\nu}}[\Phi_0], & n=3.
\end{cases}
\end{align}

The canonical singular integral $I_{\Omega_{\bm{a}_\nu},\Omega_{\bm{b}_\nu}}[\Phi_0]$ is, by \eqref{E:CanLS}, the {$\nu$th} component of the vector 
\begin{align}
\label{e:sRep}
\bm{s} = \cA^{-1} (\cB\bm{r}+\bm{t}),
\end{align}
and to obtain a numerical approximation of the components of $\bm{s}$ we combine the formula \eqref{e:sRep} with a numerical approximation of the regular integrals in the vector $\bm{r}$. Let 
$\tilde{\bm{r}}\coloneqq \left[Q^{h_s}_{\Omega_{\bm{c}_{1}},\Omega_{\bm{d}_1}}[\Phi_0],\ldots, Q^{h_s}_{\Omega_{\bm{c}_{n_r}},\Omega_{\bm{d}_{n_r}}}[\Phi_0] \right] \in \mathbb{C}^{n_r}$ denote the approximation of $\bm{r}$ obtained by applying the composite barycentre approximation to each element of $\bm{r}$ with a quadrature mesh width $h_s>0$, and set
\begin{align}
\label{e:sQuad}
	\tilde{\bm{s}}\coloneqq \cA^{-1}\left[ \cB
	\tilde{\bm{r}}
	+\bm{t} \right] \in\mathbb{C}^{n_s}.
\end{align}
We then define an approximation $\tilde{G}^0_{ij}$ of $G^0_{ij}$ by replacing $I_{\Omega_{\bm{a}_\nu},\Omega_{\bm{b}_\nu}}[\Phi_0]$ in \eqref{e:G0ijRep} 
by 
$\tilde{\bm{s}}_\nu$, 
the $\nu$th component of 
$\tilde{\bm{s}}$, 
so that
\begin{equation}\label{e:G0ijApp}
	\tilde{G}^0_{ij} \coloneqq 
\frac{1}{|\Omega_{\bm{m}{(i)}}|^{1/2}|\Omega_{\bm{m}{(j)}}|^{1/2}}
\begin{cases}
\rho_T^{4}\left(
\tilde{\bm{s}}_\nu + 
|\Omega_{\bm{a}_\nu}||\Omega_{\bm{b}_\nu}|
\log(\rho_T)\right), & n=2,\\[4mm]
\rho_T^{5} 
\tilde{\bm{s}}_\nu, & n=3.
\end{cases}
\end{equation}
\begin{lemma}\label{l:G0ijErr} 
Suppose that Assumptions \ref{a:Sep} and \ref{a:nondisj} hold. Then there exists a constant $C_0>0$ such that if $h\in(0,\diam(K)]$, $\bm{m}{(i)},\bm{m}{(j)}\in L_{h}(\Omega)$ and $K_{\bm{m}{(i)} }\cap K_{\bm{m}{(j)}} \neq \emptyset$ then 
$G^0_{ij}$ and $\tilde{G}^0_{ij}$ defined by \eqref{e:GijDef} and \eqref{e:G0ijApp} satisfy
\begin{align}
\label{e:G0ijErr}
|G^0_{ij}-\tilde{G}^0_{ij}|\leq C_0 h^2 h_s^2, \qquad \forall h_s>0. 
\end{align}
\end{lemma}
\begin{proof}
By Assumption \ref{a:nondisj} we can represent $G^0_{ij}$ by \eqref{e:G0ijRep} for some $\nu\in\{1,\ldots,n_s\}$. We can then define $\tilde{G}^0_{ij}$ by \eqref{e:G0ijApp}, with the result that
\begin{align}
\label{e:G0ijErr1}
|G^0_{ij}-\tilde{G}^0_{ij}| = 
\frac{\rho_T^{n+2}}{|\Omega_{\bm{m}{(i)}}|^{1/2}|\Omega_{\bm{m}{(j)}}|^{1/2}}
|\bm{s}_\nu - \tilde{\bm{s}}_\nu|.
\end{align}
Taking the maximum over $\nu$, and computing the difference of \eqref{e:sRep} and \eqref{e:sQuad}, gives that 
\begin{align}
\label{e:G0ijErr2}
|\bm{s}_\nu - \tilde{\bm{s}}_\nu|
\leq \|\bm{s} - \tilde{\bm{s}}\|_\infty
 \leq \|\cA^{-1}\cB\|_\infty 
 \|\bm{r} - \tilde{\bm{r}}\|_\infty.
\end{align}
From \cite[Lem.~4.1]{disjoint} we have that 
\begin{align*}
|D^\alpha \Phi_0(x,y)|\leq \frac{n}{|x-y|^n}, \qquad \forall x\neq y\in \R^n, \,\forall \alpha\in \N^{4}, |\alpha|=2,
\end{align*}
and combining this with \eqref{e:ErrDoub} we find that
\begin{align}
\label{e:G0ijErr4}
\|\bm{r} - \tilde{\bm{r}}\|_\infty
\leq 2n^2h_s^2 C_r,
\end{align}
where 
\[ C_r:= \max_{\eta=1,\ldots,n_r}\frac{|\Omega_{\bm{c}_\eta}||\Omega_{\bm{d}_\eta}|}{\dist(\Hull(\Omega_{\bm{c}_\eta}),\Hull(\Omega_{\bm{d}_\eta}))^2},\]
which is {finite} by Assumption \ref{a:Sep}(i). 
Finally, combining \eqref{e:G0ijErr1}, \eqref{e:G0ijErr2} and \eqref{e:G0ijErr4}, 
and applying \eqref{e:Scaling} and \eqref{e:LhProp1}, we obtain \eqref{e:G0ijErr} with 
\[ C_0 = \frac{2n^2C_r\|\cA^{-1}\cB\|_\infty}{(\rho_{\rm min})^n\min_{\nu=1,\ldots,n_s}(|\Omega_{\bm{a}_\nu}|(\diam(\Omega_{\bm{a}_\nu}))^2)} .\]
\end{proof}

Combining the above results, in the singular case where $(i,j)\in \cI_{\rm sing}$ we define our approximation $\tilde{A}_{ij}$ to $A_{ij}$ by
\begin{align}
\label{e:SingQuad}
\tilde{A}_{ij} := \delta_{ij} - 
mk^2(c_n\tilde{G}^0_{ij} + \tilde{G}^*_{ij}),
\end{align}
where $\tilde{G}^0_{ij}$ and $\tilde{G}^*_{ij}$ are defined by \eqref{e:G0ijApp} and \eqref{e:GstijQ}, and we have the following error estimate. 
\begin{lemma}
\label{l:SingComb}
Suppose that Assumptions \ref{a:Sep} and \ref{a:nondisj} hold. If $h\in(0,\diam(K)]$, $\bm{m}{(i)},\bm{m}{(j)}\in L_{h}(\Omega)$ and $K_{\bm{m}{(i)} }\cap K_{\bm{m}{(j)}} \neq \emptyset$ then 
$A_{ij}$ and $\tilde{A}_{ij}$ defined by \eqref{e:LinSys2a} and \eqref{e:SingQuad} satisfy
\begin{align}
\label{e:SingComb}
|A_{ij}-\tilde{A}_{ij}|\leq |m|k^2 (|c_n|C_0 h_s^2 + C^*_{\rm sing} h_*)h^2, \qquad \forall h_s>0, \,\forall h_*>0. 
\end{align}
\end{lemma}
\begin{proof}
The statement follows from \eqref{e:SingDec}, \eqref{e:SingQuad} and Lemmas \ref{l:Gst} and \ref{l:G0ijErr}. 
\end{proof}

\subsubsection{Combined estimates}

Now we can combine the above results to estimate $\|A-\tilde{A}\|_2$.

\begin{lemma}
\label{l:AErr}
Suppose that Assumptions \ref{a:Sep} and \ref{a:nondisj} hold. Then there exist $C_1,C_2,C_3>0$ such that if $h\in(0,\diam(K)]$, $h_r>0$, $h_s>0$ and $h_*>0$, then
\begin{align}
\label{e:AErr}
\|A-\tilde{A}\|_2 \leq 
\frac{1}{h^{n/2}}\left(C_1h_r^2 + C_2 h_s^2 h^2 + C_3 h_*h^2\right).
\end{align}
\end{lemma}

\begin{proof}
We first note that
\begin{align}
\label{e:AQuadEst}
\|A-\tilde{A}\|_2^2\leq \|A-\tilde{A}\|_F^2 &= \sum_{i,j=1}^N |A_{ij}-\tilde{A}_{ij}|^2
= \sum_{(i,j)\in \cI_{\rm sing}} |A_{ij}-\tilde{A}_{ij}|^2
+ \sum_{(i,j)\in \cI_{\rm reg}} |A_{ij}-\tilde{A}_{ij}|^2.
\end{align}
To estimate the sum over $\cI_{\rm sing}$ we combine the result of Lemma \ref{l:SingComb} with an estimate of $\#\cI_{\rm sing}$. 
Define $\tau:=\left(\frac{\rho_{\rm min}}{\diam(\Omega)}\right)^n|\Omega|>0$, and recall that, by \eqref{e:tauDef}, $|\Omega_{\bm{m}}|\geq \tau h^n$ for all $\bm{m}\in L_h(\Omega)$. 
Fix $i\in\{1,\ldots,N\}$ and for $r>0$ let $B_{r,i}$ denote the closed ball of radius $r$ centred at $x^{}_{\Omega_{\bm{m}(i)}}$. 
Since $K_{\bm{m}(i)}\subset B_{h,i}$, 
if $(i,j)\in \cI_{\rm sing}$ then $K_{\bm{m}(j)}\cap B_{h,i}\neq \emptyset$. By Lemma \ref{l:Falc}, 
$\#\{j:K_{\bm{m}(j)}\cap B_{h,i}\neq \emptyset\}\leq 2^n\omega_n\tau^{-1}$, 
where $\omega_n$ denotes the volume of the unit ball in $\R^n$. Then, summing over $i=1,\ldots,N$, we get
\begin{align}
\label{e:ISingEst}
\#\cI_{\rm sing} \leq 2^n\omega_n\tau^{-1}N.
\end{align}
Hence, 
by \eqref{e:SingComb} and \eqref{e:LhProp2}, 
\begin{align}
\label{e:AQuadEst2}
\left(\sum_{(i,j)\in \cI_{\rm sing}} |A_{ij}-\tilde{A}_{ij}|^2\right)^{1/2} 
&\leq
2^{n/2}\omega_n^{1/2}\tau^{-1/2}|m|k^2\left(|c_n|C_0 h_s^2 + C^*_{\rm sing} h_*\right)\left( \frac{\diam(\Omega)}{\rho_{\rm min}}\right)^{n/2}h^{(4-n)/2}.
\end{align}

For the sum over $\cI_{\rm reg}$ we use the result of Lemma \ref{l:DisQuad}. Again let $i\in \{1,\ldots,N\}$ be fixed. Then, for each $j\in\{1,\ldots,N\}$, the set $K_{\bm{m}(j)}$ must intersect at least one of either $B_{3h,i}$ or $B_{(q+1)h,i}\setminus B_{q h,i}$ for some $q \in\{3,\ldots q_*\}$, where 
\[ q_*=\left\lceil \frac{\diam(\Omega)}{h}-1 \right\rceil.\] 
Again by Lemma \ref{l:Falc}, 
$\#\{j:K_{\bm{m}(j)}\cap B_{3h,i}\neq \emptyset\}\leq 4^n\omega_n\tau^{-1}$. And, for such $j$, if $(i,j)\in \cI_{\rm reg}$ then by Assumption \ref{a:Sep}(ii) we have $d_{ij}\geq C_{\rm sep}h$, so that by \eqref{e:DisErr}
\begin{align}
\label{}
|A_{ij}-\tilde{A}_{ij}| \leq 
\frac{C_{\rm dis}h_r^2}
{(C_{\rm sep})^n}. 
\end{align}
By Lemma \ref{l:Falc2}, for $q\geq 3$ we have that $\#\{j:K_{\bm{m}(j)}\cap B_{(q+1)h,i}\setminus B_{q h,i}\neq \emptyset\}\leq 3.2^n q^{n-1}\omega_n\tau^{-1}$. And, for such $j$, $d_{ij}\geq (q-2)h$ since $\Hull(K_{\bm{m}(j)})\cap B_{(q-1)h,i}=\emptyset$ and $\Hull(K_{\bm{m}(i)})\subset B_{h,i}$, so that, by \eqref{e:DisErr},
\begin{align}
\label{}
|A_{ij}-\tilde{A}_{ij}| \leq 
\frac{C_{\rm dis}h_r^2}
{(q-2)^n}. 
\end{align}
Combining these observations and summing over $i=1,\ldots,N$ gives that
\begin{align}
\label{}
\sum_{(i,j)\in \cI_{\rm reg}} |A_{ij}-\tilde{A}_{ij}|^2 \leq N\omega_n\tau^{-1}\left(4^n\left(\frac{C_{\rm dis}h_r^2}
{(C_{\rm sep})^n}\right)^2 + \sum_{q=3}^{q_*} 3.2^n q^{n-1} \left(\frac{C_{\rm dis}h_r^2}
{(q-2)^n}\right)^2 \right),
\end{align}
and, noting that $q\geq 3$ implies that $q\leq 3(q-2)$, so that (with $\zeta$ the Riemann zeta function)
\[ 
\sum_{q=3}^{q_*} \frac{q^{n-1}}{(q-2)^{2n}} 
\leq 3^{n-1}\sum_{q=1}^\infty \frac{1}{q^{n+1}} = 3^{n-1}\zeta(n+1),
\]
we obtain, using \eqref{e:LhProp2}, that
\begin{align}
\label{e:AQuadEst3}
\left(\sum_{(i,j)\in \cI_{\rm reg}} |A_{ij}-\tilde{A}_{ij}|^2\right)^{1/2} \leq 
\left( \frac{\diam(\Omega)}{\rho_{\rm min}}\right)^{n/2}
C_{\rm dis}\omega_n^{1/2}\tau^{-1/2}
\left( \frac{2^n}{(C_{\rm sep})^n} + 6^{n/2}\zeta(n+1)^{1/2}\right)\frac{h_r^2}{h^{n/2}}.
\end{align}
Finally, combining \eqref{e:AQuadEst2} and \eqref{e:AQuadEst3} with \eqref{e:AQuadEst}, and applying \eqref{e:tauDef}, gives \eqref{e:AErr} with 
\[ C_1 = C_{\rm dis}\omega_n^{1/2}|\Omega|^{-1/2}
\left( \frac{2^n}{(C_{\rm sep})^n} + 6^{n/2}\zeta(n+1)^{1/2}\right)\left( \frac{\diam(\Omega)}{\rho_{\rm min}}\right)^{n},\]
\[C_2 = 
2^{n/2}\omega_n^{1/2}|\Omega|^{-1/2}|m|k^2|c_n|C_0\left( \frac{\diam(\Omega)}{\rho_{\rm min}}\right)^{n}
,\]
\[C_3 = 2^{n/2}\omega_n^{1/2}|\Omega|^{-1/2}|m|k^2 C^*_{\rm sing} \left( \frac{\diam(\Omega)}{\rho_{\rm min}}\right)^{n}.\]
\end{proof}

\subsubsection{Right-hand side vector}
\label{s:RHSQuad}

Recalling that $g=u^{\rm i}|_\Omega$ for the scattering problem, given $h_g>0$ we approximate the entries $g_i$ of $\bm{g}$ defined in \eqref{e:LinSys2a} by
\begin{align}
\label{e:tgDef}
\tilde{g}_i:=\frac{Q^{h_g}_{\Omega_{\bm{m}(i)}}[u^{\rm i}]}{|\Omega_{\bm{m}(i)}|^{1/2}}.
\end{align}

\begin{lemma}
\label{l:gQuad}
Suppose that $u^{\rm i}$ is $C^\infty$ on an open set $D^{\rm i}$ and $K\subset D^{\rm i}$. Then there exist $\tilde{h}_g>0$ and $C_g>0$ such that
\begin{align}
\label{e:tgErr2}
\|\bg-\tilde{\bg}\|_2 \leq C_g h_g^2,
\qquad \forall h_g\in(0, \tilde{h}_g].
\end{align}

\end{lemma}

\begin{proof}
By assumption, there exists $\eps>0$ such that $K_{\eps}\subset D^{\rm i}$. 
Furthermore, we claim there exists $\tilde{h}_q>0$ such that 
$\Hull(\Omega,h_q)\subset K_\eps$ for $0<h_q\leq \tilde{h}_q$. 
To see this, define 
$D:=d_{\rm H}(\Hull(K),K)=\sup_{x\in \Hull(K)}\inf_{y\in K} |x-y|$, and note that for $\bm{m}\in L_{h_q}(\Omega)$ it holds that $d_{\rm H}(\Hull(\Omega_{\bm{m}}),\Omega_{\bm{m}})\leq \frac{h_q D}{\diam(K)}$. Hence the claim holds for any $0<\tilde{h}_q< \frac{\eps\diam(K)}{D}$.

Combining \eqref{e:LinSys2a}, \eqref{e:tgDef} and \eqref{e:ErrSing} provides an estimate of $|g_i-\tilde{g}_i|$ for each $i=1,\ldots,N$, which, after applying \eqref{e:Scaling} and \eqref{e:LhProp2}, gives  \eqref{e:tgErr2} with
\begin{align}
 \label{e:CgDef}
C_g = \frac{n|\Omega|^{1/2}}{2(\rho_{\rm min})^{n/2}}\sup_{x\in\Hull(\Omega,\tilde{h}_g)}\max_{\substack{\alpha\in \N_0^2\\|\alpha|=2}}|D^\alpha u^{\rm i}(x)|.
\end{align}
\end{proof}

\subsubsection{Linear functionals}
\label{s:LFQuad}

Similarly, for a linear functional $J$ defined by \eqref{e:JDef}, given $h_J>0$ we approximate the entries $j_i$ of $\bm{j}$ defined in \eqref{e:JSD} by
\begin{align}
\label{e:tjDef}
\tilde{j}_i:=\frac{Q^{h_J}_{\Omega_{\bm{m}(i)}}[j]}{|\Omega_{\bm{m}(i)}|^{1/2}}.
\end{align}

\begin{lemma}
\label{l:jQuad}
Suppose that $j$ is $C^2$ on an open set $D^{J}$ and $K\subset D^{\rm J}$. Then there exist $\tilde{h}_J>0$ and $C_J>0$ such that
\begin{align}
\label{e:tjErr2}
\|\bj-\tilde{\bj}\|_2 \leq C_J h_J^2,
\qquad \forall h_J\in(0, \tilde{h}_J].
\end{align}

\end{lemma}

\begin{proof}
Follows the proof of Lemma \ref{l:gQuad}, with $\tilde{h}_J$ chosen in a similar way to $\tilde{h}_g$, and with $C_J$ given by \eqref{e:CgDef} but with $u^{\rm i}$ replaced by $j$. 
\end{proof}

\subsection{Fully discrete error analysis}
\label{s:IFSFD}
We can now combine the above results with those of \S\ref{s:AbstFD} to provide a fully discrete error analysis for our piecewise constant $h$-version VIM on an $n$-attractor.

\begin{theorem}
[Fully discrete analysis]
\label{t:FD}
Let $K$ be an $n$-attractor, and let $\Omega:=K^\circ$.
Let $\m=m\chi|_\Omega$ for some constant $m\in \C$ with $\Im[m]\geq 0$. 
Let $u^{\rm i}$ be an incident field that is $C^\infty$ on an open set $D^{\rm i}$ containing $K$. 
Let $u$ denote the unique solution of the scattering problem \eqref{e:IntEqn3}. 

Suppose that Assumptions \ref{a:Sep} and \ref{a:nondisj} hold.
As $h\to 0$ suppose that the quadrature parameters introduced in \S\ref{s:IFSquad} satisfy
\begin{align}
\label{e:FDQuadAss}
h_r =O(h^{\alpha/2+n/4}),
\quad
h_s=O(h^{\alpha/2-1+n/4}),
\quad
h_*=O(h^{\alpha-2+n/2}),
\quad
h_g=O(h^{\alpha/2}),
\end{align}
with $\alpha=1$. 
Then for sufficiently small $h\in (0,\diam(\Omega)]$ the fully discrete solution $\tilde{u}_N$ of the {piecewise-constant} VIM, defined using the quadrature rules described in \S\ref{s:IFSquad}, exists and satisfies 
\[ \|u-\tilde{u}_N\|_{L^2(\Omega)}
=O(h),
\quad h\to 0, \] 
{and the fully discrete system matrix $\tilde{A}$ is well-conditioned, with 
\begin{align}
\label{e:CondIFS}
\kappa_2(\tilde{A}) = \|\tilde{A}\|_2\|\tilde{A}^{-1}\|_2 =O(1), \quad \text{as }\, h\to 0 .
\end{align}
}

Furthermore, if \eqref{e:FDQuadAss} holds with $\alpha=2$, and if 
\begin{align}
 \label{e:FDQuadAss2}
 h_J=O(h),
 \end{align} 
then for the far-field pattern evaluation functional we have $|\tilde{J}^{\rm FFP}_{\hat{x}}(\tilde{u}_N)-J^{\rm FFP}_{\hat{x}}(u)|=O(h^2)$, uniformly for $\hat{x}\in \mathbb{S}^{n-1}$; and, for the scattered field evaluation functional, for every $\eps>0$ we have that $|J^{\rm SF}_{x_0}(u_N)-J^{\rm SF}_{x_0}(u)|=O(h^2)$, uniformly for $x_0\in \R^n\setminus K_\eps$. 
\end{theorem}

\begin{proof}
Follows by combining Corollary \ref{c:Scatt}, Theorem \ref{t:FDErr} and Corollary \ref{c:App3} with Lemmas \ref{l:AErr}, \ref{l:gQuad} and \ref{l:jQuad}. %
\end{proof}

\begin{remark}
\label{r:Inside}
In Theorem \ref{t:FD} we presented fully discrete superconvergence estimates for the scattered field evaluation functional only in the case where $x_0\in \R^n\setminus K_\eps$ for some $\eps>0$. When $x_0\in K$ at least one of the integrals \eqref{e:JSD} will be singular, so that Lemma \ref{l:jQuad} does not apply, and the approximation \eqref{e:tjDef} may be inaccurate (or even undefined, if $x_0$ coincides with one of the quadrature points in the composite barycentre rule). 
In principle, special quadrature rules could be developed and analysed to deal with this case. 
However, since this is a well-known issue with volume integral equation methods in general, and is unrelated to the fractality of the inhomogeneity, we do not pursue this matter further here. Nonethelesss, we note that by 
Remark \ref{r:SupConF} we expect a degree of superconvergence at a semi-discrete level, even for $x_0\in K$, 
and that, even with naive quadrature 
the approximation \eqref{e:tjDef} often performs well in practice for $x_0\in K$, as the results in Figure \ref{f:ConvScat}(d)-(f) below show.
\end{remark}

\begin{remark}
As an example of what the assumptions \eqref{e:FDQuadAss} and \eqref{e:FDQuadAss2} mean in practical terms, we note that, in the case $n=2$, to achieve \eqref{e:FDQuadAss} with $\alpha=2$ and \eqref{e:FDQuadAss2}, giving optimal $O(h^2)$ superconvergence of the functional evaluations,
we need $h_r/h, h_s=O(h^{1/2})$ %
and $h_*/h,h_g/h,h_J/h = O(1)$, as $h\to 0$. This means that a fixed ($h$-independent) number of quadrature points suffices for the evaluation of each of the integrals involving $\Phi_*$, $g$ and $J$, while the number of quadrature points for the evaluation of each of the regular integrals appearing in $A_{ij}$ when $K_{\bm{m}(i)}\cap K_{\bm{m}(j)}=\emptyset$, and the evaluation of each of the regular canonical integrals in $\bm{r}$, needs to increase modestly (in proportion to $(h^{1/2})^{-2}=h^{-1}$) as $h\to 0$.

\end{remark}

\section{Examples}
\label{s:Examp}
In this section we consider some examples of $n$-attractors $K$ to which the results of \S\ref{s:IFS} apply. 
While our theory applies for both $n=2$ and $n=3$, for ease of visualisation, and to reduce the cost of numerical computations, we restrict our attention henceforth to the case $n=2$.\footnote{Examples of $3$-attractors can be found in \cite{thuswaldner2020self,zaitseva2022self}. We note that the term ``$2$-attractor'' 
has a different meaning in \cite{zaitseva2022self} compared to its meaning in this paper.}
We focus here on the three examples illustrated in Figure \ref{f:Scatt}, namely the Fudgeflake, the Gosper Island, and the Koch Snowflake. 
Each is a $2$-attractor defined by an IFS $\{s_1,\ldots,s_M\}$ in $\R^2$ for some $M\in \N$, with each $s_i$ a contracting similarity satisfying \eqref{e:IFS}, as detailed below.

In order to justify the applicability of the quadrature rules described in \S\ref{s:IFSquad} and the fully discrete error estimates of Theorem \ref{t:FD} 
we need to verify in each case that the conditions of Assumptions \ref{a:Sep} and \ref{a:nondisj} are satisfied. 
Careful inspection of the first few $L_h$ meshes for each example (see Figures \ref{f:FudgeMesh}-\ref{f:KochMesh}) suggests that this should be true. But rigorously proving this %
involves somewhat tedious (albeit elementary) geometrical arguments that we do not reproduce here for reasons of space. 
What we do provide for each example is a collection of basic facts (stated without proof)  
that provide the main building blocks of the arguments. 
In particular, for each example we report the canonical singular integrals required for Assumption \ref{a:nondisj} (determined by the algorithm in \cite{nondisjoint}), we detail the symmetry properties of the attractor, and we state formulas showing that the $L_h$ meshes possess a certain lattice structure.
We also report results (obtained in part using \cite[Theorem 2.1]{vass2013explicit}, which is restated as Lemma \ref{l:BoundBall} below) guaranteeing that the attractors are contained in certain bounding balls. In conjunction with the lattice formulas, this permits rigorous verification of the separation conditions in Assumption \ref{a:Sep}, and guarantees the existence of a suitable similarity $T$ in Assumption \ref{a:nondisj} that allows a given singular interaction 
$G^0_{ij}$ (as defined in \eqref{e:GijDef}) to be expressed in terms of one of the canonical singular integrals (as in \eqref{e:G0ijRep}). 
For the interested reader, full details of the relevant calculations for the Koch Snowflake case can be found in \cite[\S4.2.2]{bannisterthesis}; the analogous calculations for the Fudgeflake and Gosper Island, while not presented in \cite{bannisterthesis}, are simpler because their IFS are homogeneous. 

To define our examples  
we introduce the notation 
\[ R_\theta := 
\left(
\begin{array}{cc}
\cos\theta & -\sin\theta\\
\sin\theta & \cos\theta
\end{array}
\right)
\]
to denote the matrix corresponding to the operation of anticlockwise rotation around the origin in $\R^2$ by an angle $\theta\in[0,2\pi)$.  
We also define the lattice vectors
\begin{align}
\label{e:eps12}
\eps_1 := \left(\begin{array}{c}
1 \\
0
\end{array}\right), \qquad
\eps_2 := \left(\begin{array}{c}
1/2 \\
\sqrt{3}/2
\end{array}\right)=R_{\pi/3}\left(\begin{array}{c}
1 \\
0
\end{array}\right).
\end{align}


\subsection{Fudgeflake}
\label{s:Fudge}

The Fudgeflake 
is defined by a homogeneous IFS, with 
$M=3$ and 
\begin{align}
&\rho_1=\rho_2=\rho_3 = \frac{1}{\sqrt{3}},
\qquad
A_1=A_2=A_3=R_{\pi/6},\notag\\
&
\delta_j = 
\frac{1}{3}R_{(2j-1)\pi/3}
\,\eps_1
, \,\,
j=1,\ldots 3.%
\label{e:FudgeIFS}
\end{align}
An illustration of the Fudgeflake and its first few $L_h$ meshes can be found in Figure~\ref{f:Meshes}(a). %
The barycentre $x_\Omega$ of the Fudgeflake is at the origin, its area $|\Omega|$ equals $\sqrt{3}/2$, and
the dimension of its boundary is 
$d=\dimH(\partial\Omega)= \log 4/\log 3\approx 1.26$ \cite{riddleIFS}.
The Fudgeflake has $C_3$ symmetry, i.e.\ $\Omega$ is invariant under rotation around the origin by multiples of $2\pi/3$. 
An application of Lemma \ref{l:BoundBall} (noting that $\varrho(0)=\rho_*=\mu_*=1/\sqrt{3}$ in this case) proves that the Fudgeflake lies within the ball centred at the origin of radius $(1/2)(1+1/\sqrt{3})\approx 0.789$.

The sequence of $L_h$ mesh widths for which \eqref{e:hkProp} holds is given by $h_l=3^{-l/2}h_0$, $l\in \N_0$, 
{where $h_0:=\diam(\Omega)$, } 
and the corresponding number of mesh elements satisfies {$\#L_{h_l}(\Omega)=3^l$}.
The meshes have the following lattice structure: given $l\in \N$, for $\bm{m}\in L_{h_l}(\Omega)$ there exist $a_1,a_2\in\Z$ such that
\begin{align}
\label{e:FudgeLat}
\Omega_{\bm{m}} = {
3^{-l/2}R_{l\pi/6}\Omega 
+ 3^{-(l+1)/2}R_{(l+1)\pi/6}\eps_1
+ 3^{-l/2}R_{l\pi/6}(a_1\eps_1+a_2\eps_2)}.
\end{align}
Applying the algorithm of \cite{nondisjoint} to the Fudgeflake IFS produces two canonical singular integrals, namely $I_{\Omega,\Omega}[\Phi_0]$, which we refer to as a ``self-interation'', and $I_{\Omega_{1},\Omega_{2}}[\Phi_0]$, which we refer to as an ``edge interaction'', since it involves integration over the pair of neighbouring elements $\Omega_{1}$ and $\Omega_{2}$, whose closures intersect in a fractal curve (see Figure \ref{f:FudgeMesh}). 

By combining the above facts one can show that Assumption \ref{a:Sep} is satisfied with $C_{\rm sep} =(2/\sqrt{3}-1)/(h_0\sqrt{3})$, and that Assumption \ref{a:nondisj} holds, since for any pair of $L_h$ mesh elements $\Omega_{\bm{m}(i)}$ and $\Omega_{\bm{m}(j)}$ whose closures intersect nontrivially either $i=j$, in which case $G^0_{ij}$ can be expressed in terms of the self-interaction $I_{\Omega,\Omega}[\Phi_0]$, or $i\neq j$, in which case the pair $\Omega_{\bm{m}(i)}$ and $\Omega_{\bm{m}(j)}$ must be immediate neighbours on the lattice \eqref{e:FudgeLat}, and then $G^0_{ij}$ can be expressed in terms of the edge interaction $I_{\Omega_{1},\Omega_{2}}[\Phi_0]$.

\subsection{Gosper Island}
\label{s:Gosp}
The Gosper Island IFS is also homogeneous, with 
$M=7$ and 
\begin{align}
&\rho_1=\ldots=\rho_7 = \frac{1}{\sqrt{7}},
\qquad
A_1=\ldots=A_7=R_{\vartheta}, \quad \vartheta = \arcsin\left( \frac{\sqrt{3}}{2\sqrt{7}}\right),
\quad 
R_\vartheta=\left(\begin{array}{cc}
\frac{5}{2 \sqrt{7}} & -\frac{\sqrt{3}}{2\sqrt{7}} \\
\frac{\sqrt{3}}{2\sqrt{7}} & \frac{5}{2 \sqrt{7}}
\end{array}\right). 
\notag\\
&\delta_j = \frac{\sqrt{3}}{\sqrt{7}}R_{(j-1)\pi/3 + \vartheta} \,\eps_1
, \,\,
j=1,\ldots 6, \,\,
\delta_7=\left(\begin{array}{c}
0 \\
0
\end{array}\right). 
\label{e:GospIFS}
\end{align}
An illustration of the Gosper Island and its first few $L_h$ meshes can be found in 
Figure~\ref{f:Meshes}(b). %
Its barycentre $x_\Omega$ is at the origin, its area $|\Omega|$ equals $3\sqrt{3}/2$, 
and the dimension of its boundary is $d=\dimH(\partial\Omega)= \log 3/\log {\sqrt{7}}\approx 1.13$ \cite{riddleIFS}.
The Gosper Island has $C_6$ symmetry, i.e.\ $\Omega$ is invariant under rotation around the origin by multiples of $\pi/3$. 
An application of Lemma \ref{l:BoundBall} (noting that $\varrho(0)=1$, $\rho_*=1/\sqrt{7}$, $\mu_*=\sqrt{3}/\sqrt{7}$ in this case) implies that the Gosper Island lies within the ball centred at the origin of radius $(1+\sqrt{7})/(2\sqrt{3})\approx 1.052$. 

The sequence of $L_h$ mesh widths for which \eqref{e:hkProp} holds is given by $h_l=7^{-l/2}h_0$, $l\in \N_0$, 
{where $h_0:=\diam(\Omega)$, } 
and the corresponding number of mesh elements satisfies {$\#L_{h_l}(\Omega)=7^l$}.
The meshes have the following lattice structure: given $l\in \N$, for $\bm{m}\in L_{h_l}(\Omega)$ there exist $a_1,a_2\in\Z$ such that
\begin{align}
\label{e:GosperLat}
\Omega_{\bm{m}} = 7^{-l/2}R_{l\vartheta}\Omega 
+3^{1/2}7^{-l/2}R_{l\vartheta}(a_1\eps_1+a_2\eps_2).
\end{align}
As for the Fudgeflake, applying the algorithm of \cite{nondisjoint} to the Gosper Island produces two canonical singular integrals, namely the self interaction $I_{\Omega,\Omega}[\Phi_0]$ and the edge interaction $I_{\Omega_{1},\Omega_{2}}[\Phi_0]$ (see Figure \ref{f:GosperMesh}). By combining the above facts one can show that Assumption \ref{a:Sep} is satisfied with $C_{\rm sep} =(3-(1+\sqrt{7})/\sqrt{3})/(h_0\sqrt{7})$ and that Assumption \ref{a:nondisj} holds, since all the singular interactions $G^0_{ij}$ can be expressed either in terms of $I_{\Omega,\Omega}[\Phi_0]$ or $I_{\Omega_{1},\Omega_{2}}[\Phi_0]$.

\subsection{Koch Snowflake}
\label{s:Koch}
The Koch Snowflake IFS is non-homogeneous, with 
$M=7$ and 
\begin{align}
&\rho_1=\ldots=\rho_6 = \frac{1}{3}, \, \rho_7=\frac{1}{\sqrt{3}},
\qquad
A_1=\ldots=A_6 = I, \,\, A_7=R_{\pi/6},\notag\\
&\delta_j = \frac{{h_0}}{3}R_{{j\pi/3 - \pi/6}}
\,\eps_1
, \,\,
j=1,\ldots 6, \,\,
\delta_7=\left(\begin{array}{c}
0 \\
0
\end{array}\right), 
\label{e:KochIFS}
\end{align}
{where $h_0>0$ is the diameter of the Snowflake and $\tilde{h}_0:=\sqrt{3} h_0/2$ is the side length of the equilateral triangle forming the first prefractal approximation to it, as illustrated in Figure \ref{f:KochPreFrac}. 
In our numerical results in \S\ref{s:Numer} below we use $h_0=2/\sqrt{3}$, so that $\tilde{h}_0=1$, but we note that \cite{nondisjoint} uses $h_0=2$, so that $\tilde{h}_0=\sqrt{3}$.} 
An illustration of the Koch Snowflake and its first few $L_h$ meshes can be found in Figure~\ref{f:Meshes}(c). %
Its barycentre $x_\Omega$ is at the origin, its area $|\Omega|$ equals 
{$3h_0^2\sqrt{3}/10$}, 
 and the dimension of its boundary is $d=\dimH(\partial\Omega)= \log 4/\log 3\approx 1.26$ \cite{riddleIFS}. 
The Koch Snowflake has $D_6$ symmetry, i.e.\ $\Omega$ is invariant under rotation around the origin by multiples of $\pi/3$ and under reflection in the lines through the origin with angles $j\pi/6$, $j=1,\ldots,6$. 
An application of Lemma \ref{l:BoundBall} (noting that $\varrho(0)={h_0/2}$, $\rho_*=1/\sqrt{3}$, $\mu_*=2/3$ in this case) implies that the Koch Snowflake lies within the ball centred at the origin of radius ${(1+1/\sqrt{3})h_0/2}\approx 0.789h_0$. However, by different arguments (e.g.\ prefractal approximation, as in Figure \ref{f:KochPreFrac}) one can prove the sharper result that  
{$\diam(\Omega)=h_0$} 
(as stated already above)  
and that $\Hull(\Omega)$ is a hexagon of circum-radius ${h_0/2}$. 

The sequence of $L_h$ mesh widths for which \eqref{e:hkProp} holds is given by $h_l=3^{-l/2}h_0$, $l\in \N_0$, as for the Fudgeflake.  
However, the fact that the IFS is non-homogeneous means that the lattice structure of the $L_h$ meshes is more complicated than for the Fudgeflake or Gosper Island, as we now explain. For the Koch Snowflake each $L_h$ mesh comprises elements of two different sizes. Explicitly,
\begin{align}
\label{}
L_{h_l}(\Omega) = B_l \cup S_l,
\end{align}
where (with ``$B$'' for ``big'' and ``$S$'' for ``small'')
\begin{align}
\label{}
B_l &= \{\bm{m}\in L_{h_l}(\Omega):\diam(\Gamma_{\bm{m}})=h_03^{-l/2}\},\\
S_l &= \{\bm{m}\in L_{h_l}(\Omega):\diam(\Gamma_{\bm{m}})=h_03^{-(l+1)/2}\}.
\end{align}
In obtaining $L_{h_{l+1}}(\Omega)$ from $L_{h_l}(\Omega)$, each element of $B_l$ is subdivided into 7 pieces, one of which becomes an element of $B_{l+1}$, with the remaining six becoming elements of $S_{l+1}$; and each element of $S_l$ becomes an element of $B_{l+1}$. Hence
\begin{align}
\label{}
\#B_{l+1} = \#B_{l} + \#S_{l}, 
\qquad
\#S_{l+1} = 6(\#B_{l}).
\end{align}
Solving these difference equations with $\#B_1=1$ and $\#S_1=6$ gives 
\begin{align}
\label{}
\#B_l = \frac{3^{l+1}-(-2)^{l+1}}{5},
\qquad 
\#S_l = \frac{2.3^{l+1}+3.(-2)^{l+1}}{5},
\qquad l\in\N,
\end{align}
so that the total number of mesh elements satisfies
\begin{align}
\label{e:NlSim}
\#L_{h_l}(\Omega) = \#B_l + \#S_l = \frac{3^{l+2}-(-2)^{l+2}}{5}\sim \frac{3^{l+2}}{5}{=\frac{9}{5}\left(\frac{h_l}{h_0}\right)^{-2}}, \quad l\to\infty.
\end{align}
Regarding the lattice structure, given $l\in \N$, we can further decompose $S_l=S_l^1\cup S_l^2$, with $\#S_l^1=\#S_l^2=(\#S_l)/2$, and %
\begin{align}
B_l&=\{\bm{m}\in L_{h_l}(\Omega): \Omega_{\bm{m}} = 
{\frac{h_0}{2}}3^{-l/2}R_{l\pi/6}\Omega 
+ {h_0}3^{-l/2}R_{(l+1)\pi/6}(a_1\eps_1 + a_2\eps_2), \notag\\
& \qquad \qquad \qquad \qquad \qquad \qquad \qquad \qquad \qquad \qquad \qquad \qquad \qquad
\text{ for some }a_1,a_2\in\Z\}, \notag\\
S_l^j &=
\{ \bm{m}\in L_{h_l}(\Omega): \Omega_{\bm{m}} = {\frac{h_0}{2}}3^{-(l+1)/2}R_{(l+1)\pi/6}\Omega + {h_0}3^{-(l+1)/2}R_{(l+2)\pi/6}\,\eps_j \notag\\
&\qquad\qquad \qquad \qquad \qquad
+ {h_0}3^{-l/2}R_{(l+1)\pi/6}(a_1\eps_1 + a_2\eps_2), 
\quad 
\text{for some }a_1,a_2\in\Z\}, 
\quad j=1,2.
\label{e:KochLat}
\end{align}
The three lattices described by \eqref{e:KochLat} are illustrated in Figure \ref{f:KochLat}. 

\begin{figure}[t!]
\centering
\includegraphics[width=.25\textwidth]{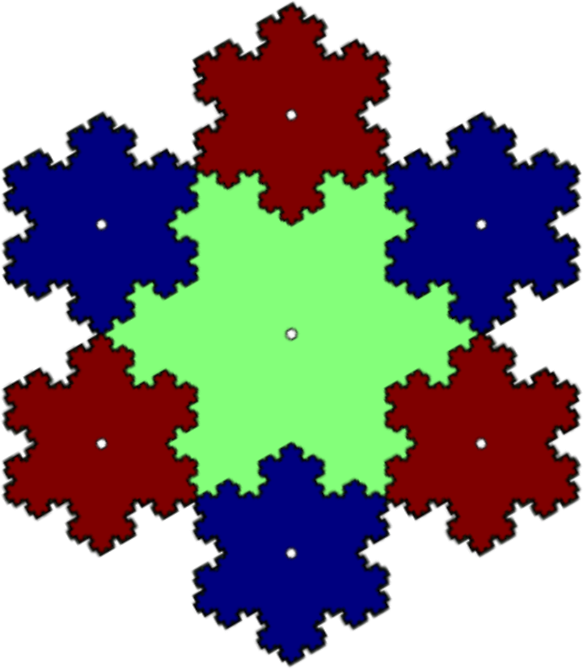}
\hspace{5mm}
\includegraphics[width=.25\textwidth]{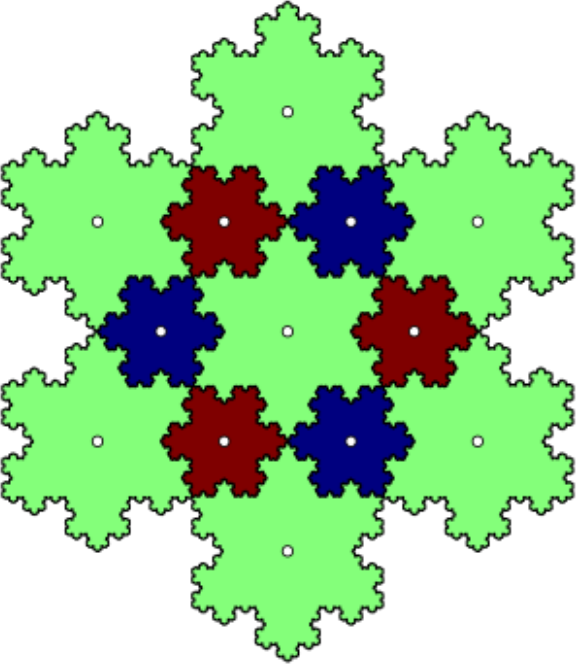}
\hspace{5mm}
\includegraphics[width=.25\textwidth]{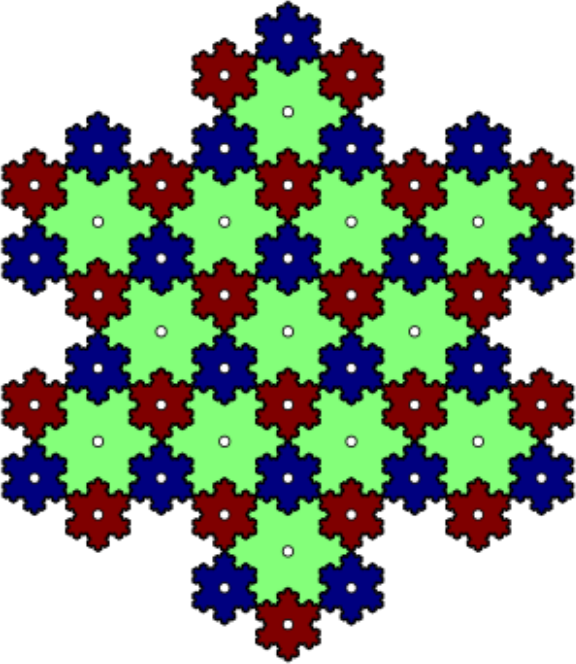}
\caption{Lattice structure of the $L_{h_l}$ meshes for the Koch Snowflake, $l=1,2,3$. The elements corresponding to the index sets $B_l$, $S_l^1$ and $S_l^2$ are coloured green, red and blue respectively. 
}
\label{f:KochLat}
\end{figure}

As was detailed in \cite[\S5.4]{nondisjoint}, for the Koch Snowflake the algorithm of \cite{nondisjoint} produces three canonical singular integrals, namely the self interaction $I_{\Omega,\Omega}[\Phi_0]$, the edge interaction $I_{\Omega_{1},\Omega_{7}}[\Phi_0]$, and the ``point interaction'' $I_{\Omega_{1},\Omega_{2}}[\Phi_0]$, the latter so named because the closures of the elements $\Omega_{1}$ and $\Omega_{2}$ intersect in a single point (see Figure \ref{f:GosperMesh}). 
By combining the above facts one can show that Assumption \ref{a:Sep} is satisfied 
with $C_{\rm sep} =1/(2h_0)$ %
and that Assumption \ref{a:nondisj} holds, since all the singular interactions $G^0_{ij}$ can be expressed either in terms of $I_{\Omega,\Omega}[\Phi_0]$, $I_{\Omega_{1},\Omega_{7}}[\Phi_0]$ or $I_{\Omega_{1},\Omega_{2}}[\Phi_0]$.
Further details of the calculations underpinning these statements can be found in \cite{bannisterthesis}.

\section{Numerical results}
\label{s:Numer}

\begin{figure}[t!]
\begin{subfigure}[t]{0.3\textwidth}
\includegraphics[height=\textwidth]
{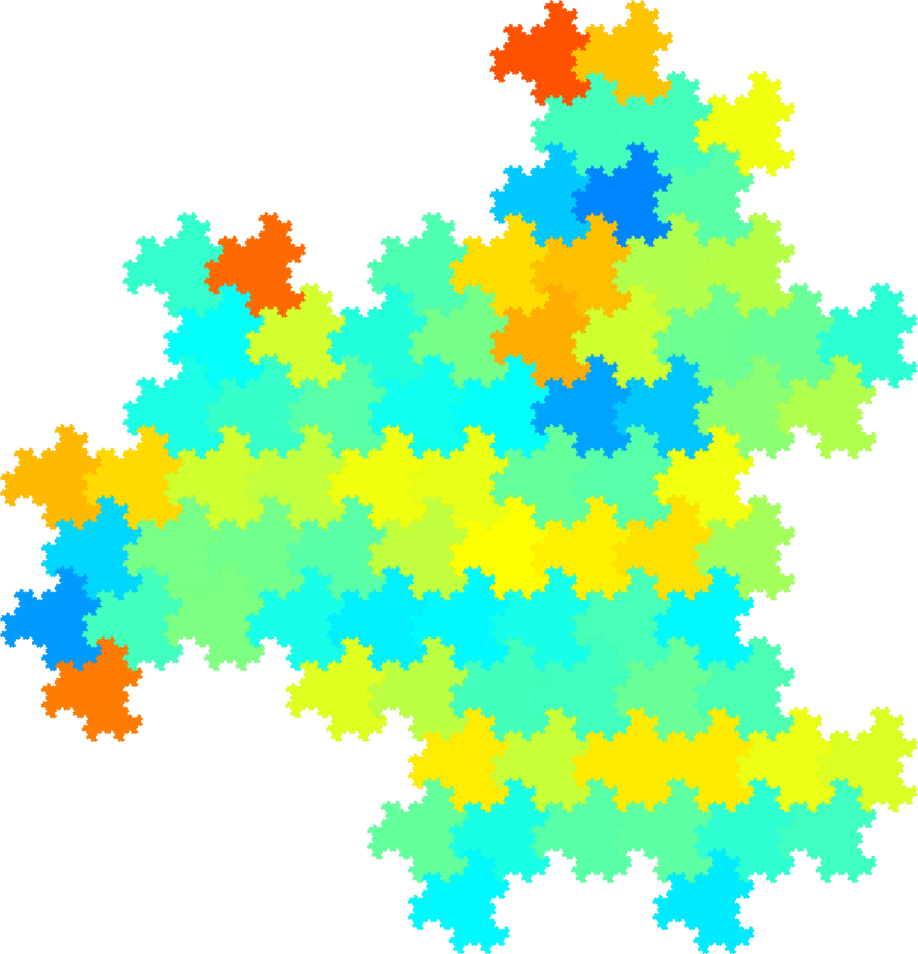}
\caption{$l=4$}
\end{subfigure}
\hspace{1mm}
\begin{subfigure}[t]{0.3\textwidth}
\includegraphics[height=\textwidth]
{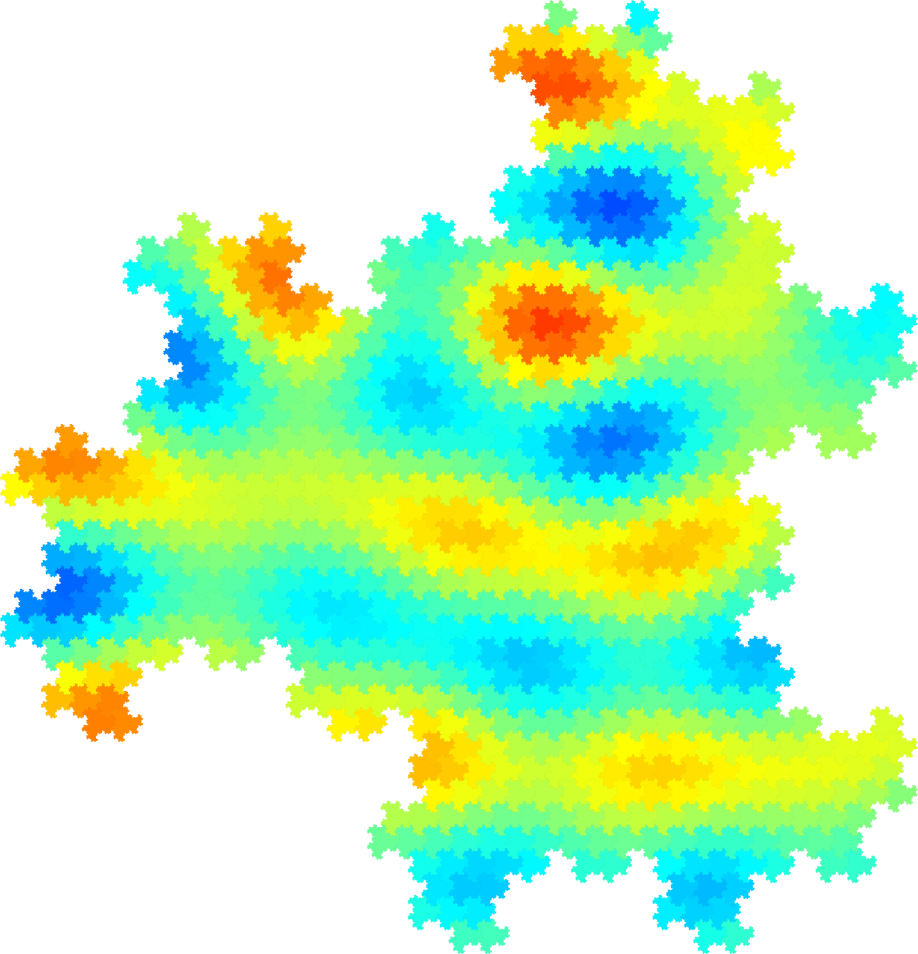}
\caption{$l=6$}
\end{subfigure}
\hspace{1mm}
\begin{subfigure}[t]{0.3\textwidth}
\includegraphics[height=\textwidth]
{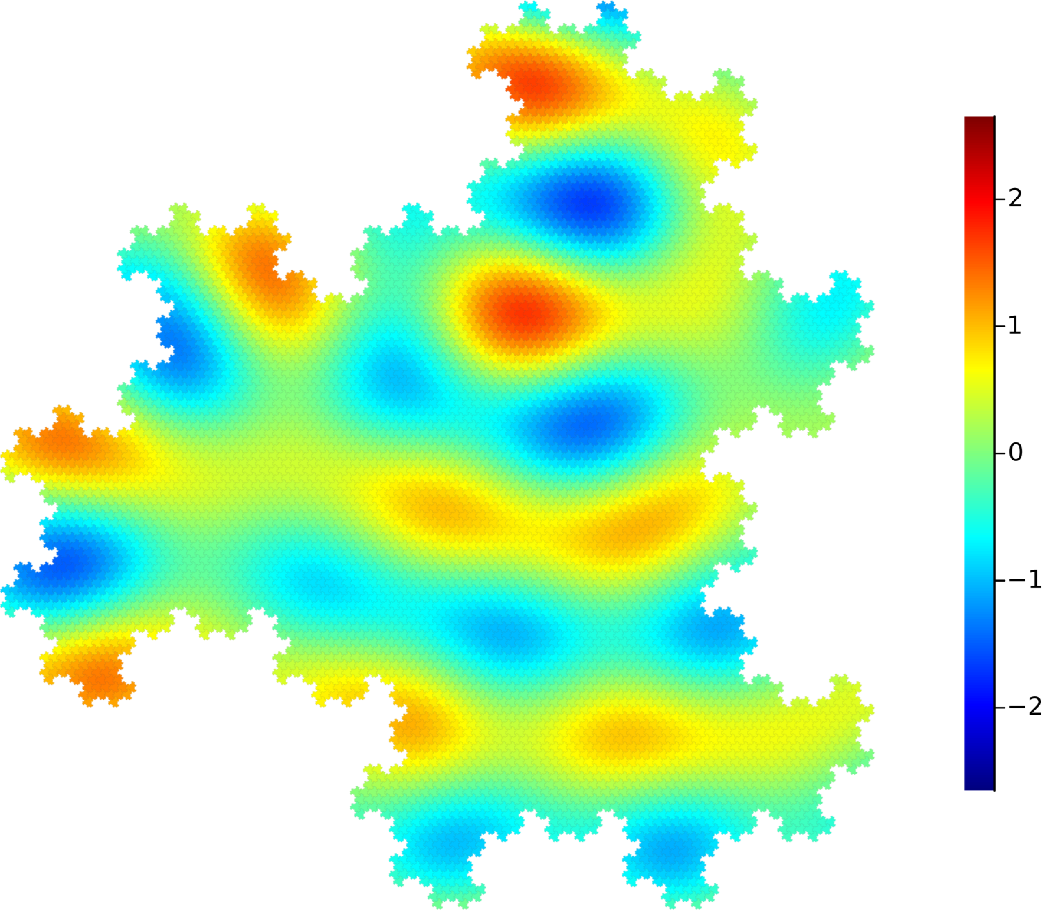}
\caption{$l=8$}
\end{subfigure}

\bigskip

\begin{subfigure}[t]{0.3\textwidth}
\includegraphics[height=\textwidth]
{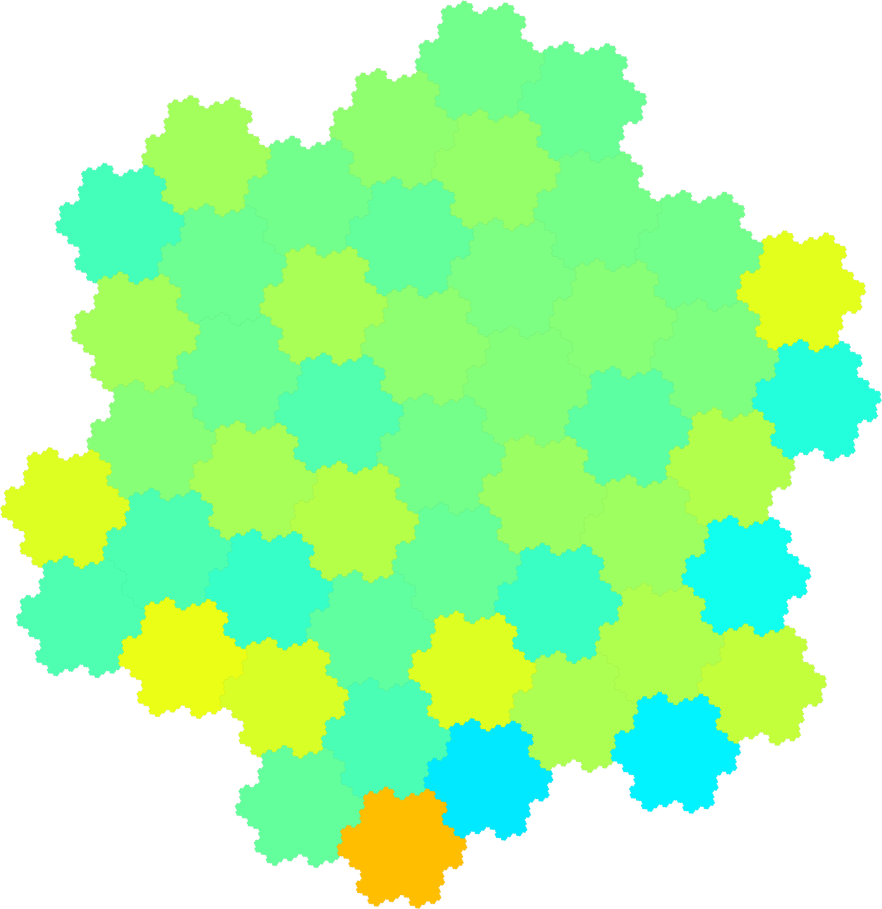}
\caption{$l=2$}
\end{subfigure}
\hspace{1mm}
\begin{subfigure}[t]{0.3\textwidth}
\includegraphics[height=\textwidth]
{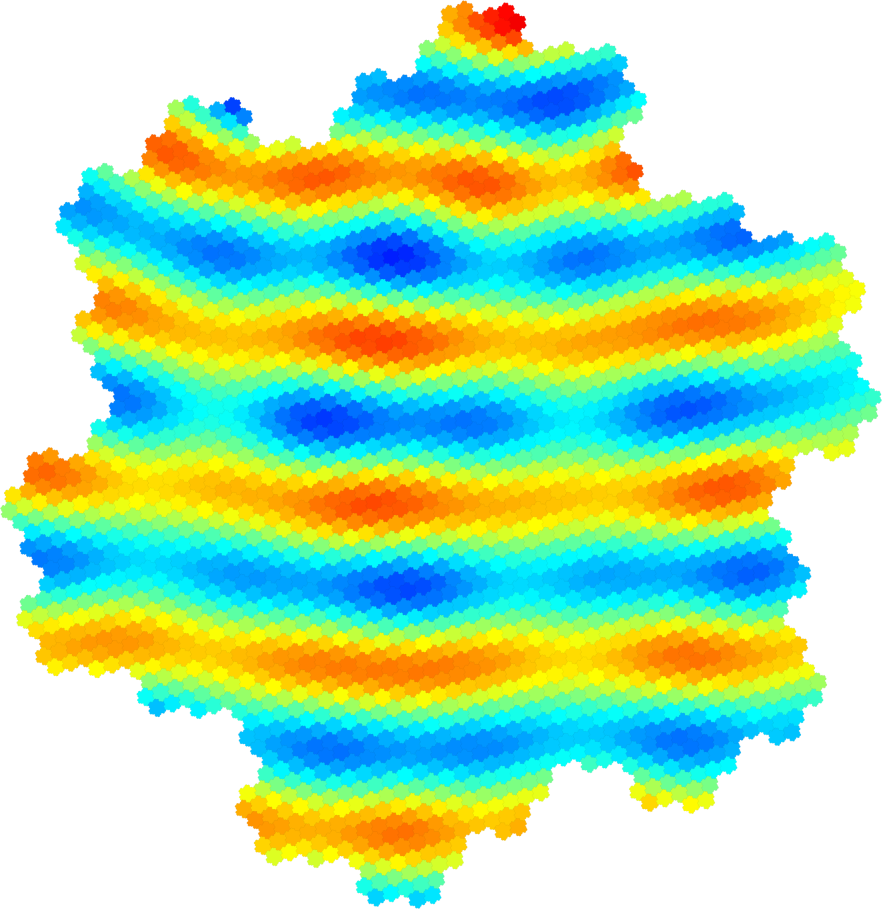}
\caption{$l=4$}
\end{subfigure}
\hspace{1mm}
\begin{subfigure}[t]{0.3\textwidth}
\includegraphics[height=\textwidth]
{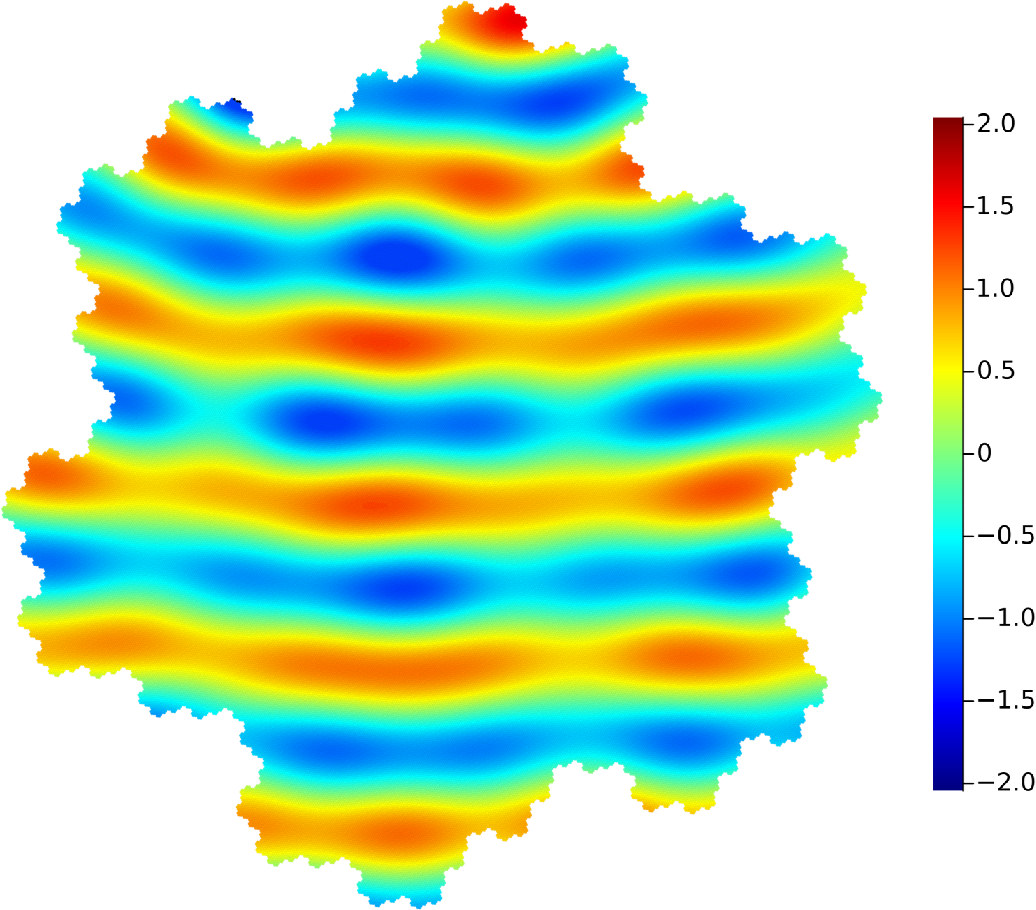}
\caption{$l=6$}
\end{subfigure}

\bigskip

\begin{subfigure}[t]{0.3\textwidth}
\includegraphics[height=\textwidth]
{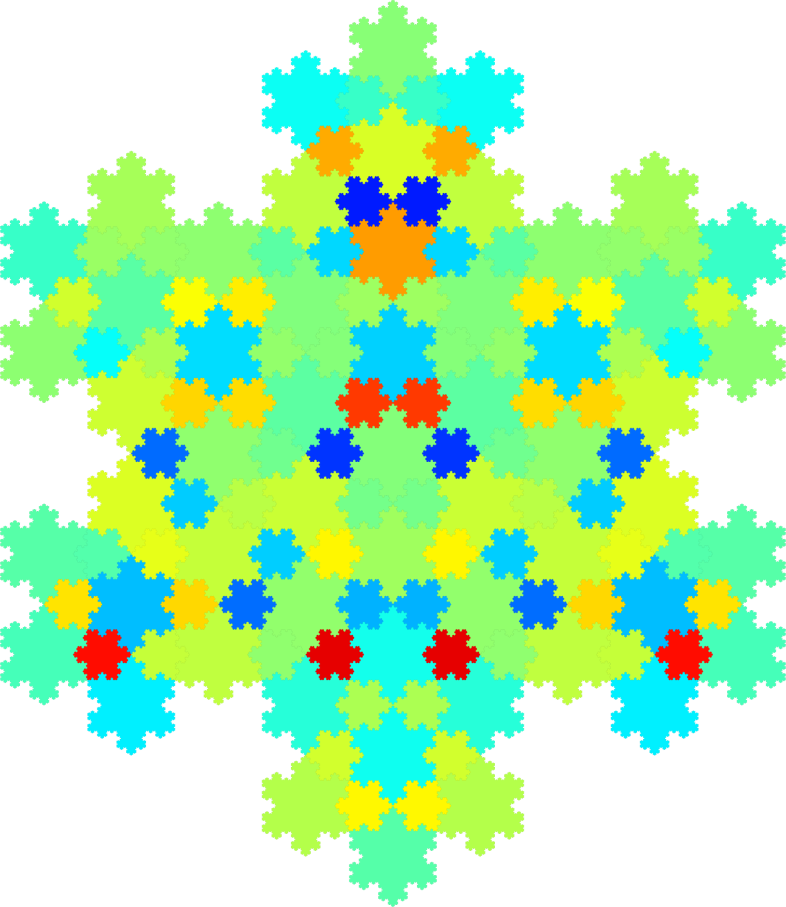}
\caption{$l=4$}
\end{subfigure}
\hspace{1mm}
\begin{subfigure}[t]{0.3\textwidth}
\includegraphics[height=\textwidth]
{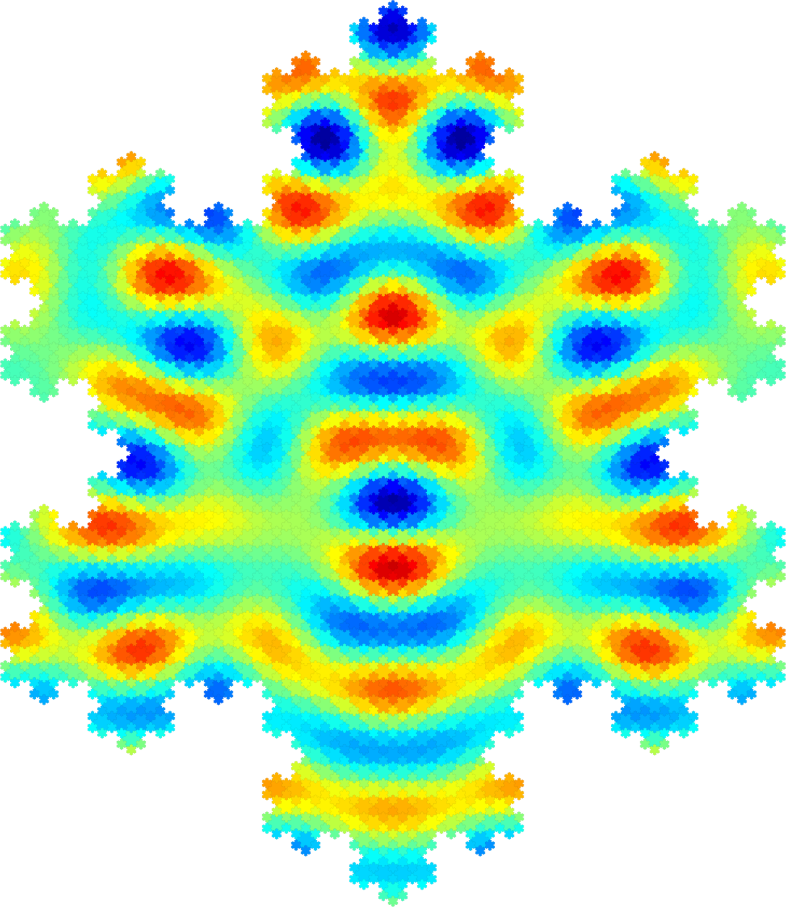}
\caption{$l=7$}
\end{subfigure}
\hspace{1mm}
\begin{subfigure}[t]{0.3\textwidth}
\includegraphics[height=\textwidth]
{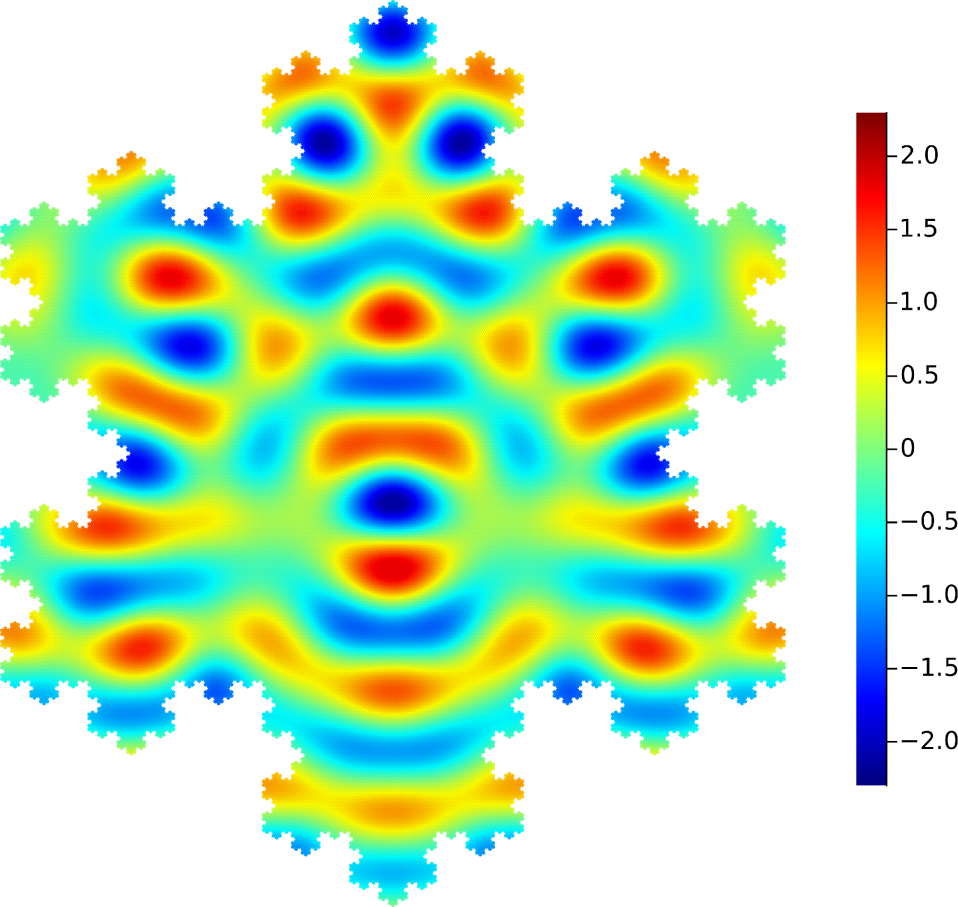}
\caption{$l=10$}
\end{subfigure}
\caption{Piecewise-constant VIM solutions (real parts) on each of the three examples, for coarse (left), medium (middle) and fine meshwidths (right). 
Here $k=30$, $\vartheta=(0,1)$, and in (a)-(c) and (g)-(i) $1+m=2.0$, while in (d)-(f) $1+m=1.311+2.289\times 10^{-9}\ri$. 
}
\label{f:PWC}
\end{figure}

In this section we validate our theoretical analysis via numerical experiments. We present results for scattering by the three IFS attractors discussed in \S\ref{s:Examp}, namely the Fudgeflake, the Gosper Island and the Koch Snowflake. 
We use the quasi-uniform $L_{h}$ meshing strategy described in \S\ref{s:IFSMesh}, for the meshwidth sequences $(h_l)_{l=0}^\infty$ detailed in \S\ref{s:Examp}, and we restrict attention to piecewise constant approximations, using the quadrature rules described in \S\ref{s:IFSquad}. Our first objective is to demonstrate that our geometry-conforming IFS-based method achieves the accuracy predicted by the fully discrete error analysis in Theorem \ref{t:FD} (namely, $O(h)$ convergence for the VIE solution in $L^2(\Omega)$ and $O(h^2)$ superconvergence for the scattered field and far-field pattern evaluations). 
Our second objective is to show that our  method provides significantly more accurate results than a comparable method based on prefractal approximation. 

Our numerical experiments were carried out using the open-source Julia code available at \href{https://github.com/AndrewGibbs/IFSIntegrals}{github.com/AndrewGibbs/IFSIntegrals}, which was used to generate the $L_{h}$ meshes and carry out numerical quadrature. %
For inversion of the linear system \eqref{e:LinSysFD}
we used matrix-free GMRES, exploiting the lattice structure of the $L_{h}$ meshes (as described in \S\ref{s:Examp}) to reduce memory complexity and to accelerate matrix-vector products via FFT-based techniques. Details of this aspect of our implementation can be found in \cite[\S5-\S6]{bannisterthesis} and will be presented in a separate publication. 

We present results for the scattering of an incident plane wave $u^{\rm i}(x) := e^{ik\vartheta\cdot x}$ by the three examples from \S\ref{s:Examp} (with $h_0=2/\sqrt{3}$ in the case of the Koch Snowflake, so that $\tilde{h}_0=1$, see \S\ref{s:Koch}), for different combinations of wavenumbers $k\in\{5,7.5,10, 15, 20, 30\}$ and refractive index perturbation $\m = m\chi_\Omega$, with $1+m\in\{2.0, 1.311+2.289\times 10^{-9}\ri\}$ (the second value is a representative value for light scattering by ice crystals at wavelength $0.55\mu m$, cf.\ \cite{warren2008optical}). 

Plots of the numerically computed total field for the three examples under consideration were given in Figure \ref{f:Scatt}. Here $k=30$ and $\vartheta=(0,1)$ in all three cases, while $1+m=2$ in Figure \ref{f:Scatt}(a)\&(c) and  $1+m=1.311+2.289\times 10^{-9}\ri$ in Figure \ref{f:Scatt}(b). 
These plots were generated by numerical evaluation of the formula \eqref{e:uthDef} via the composite barycentre quadrature rule with a large number of elements, using a  piecewise constant VIM solution $u_N$ with meshwidth $h=h_l$ for $l=9,6,10$ respectively, where $h_l$ is defined for each example in \S\ref{s:Examp}.

In Figure \ref{f:PWC} we show plots of the piecewise constant VIM solution on the three examples, for the same physical parameters as in Figure \ref{f:Scatt}, for three different mesh widths. The purpose of this figure is to emphasize that in our VIM the geometry of $\Omega$ is completely captured by the mesh, so there is no geometry approximation. In the coarsest examples one can clearly see the fractal nature of the mesh elements. 

\begin{figure}[t!]
\begin{subfigure}[t]{0.31\textwidth}
\centering
\includegraphics[width=\textwidth]
{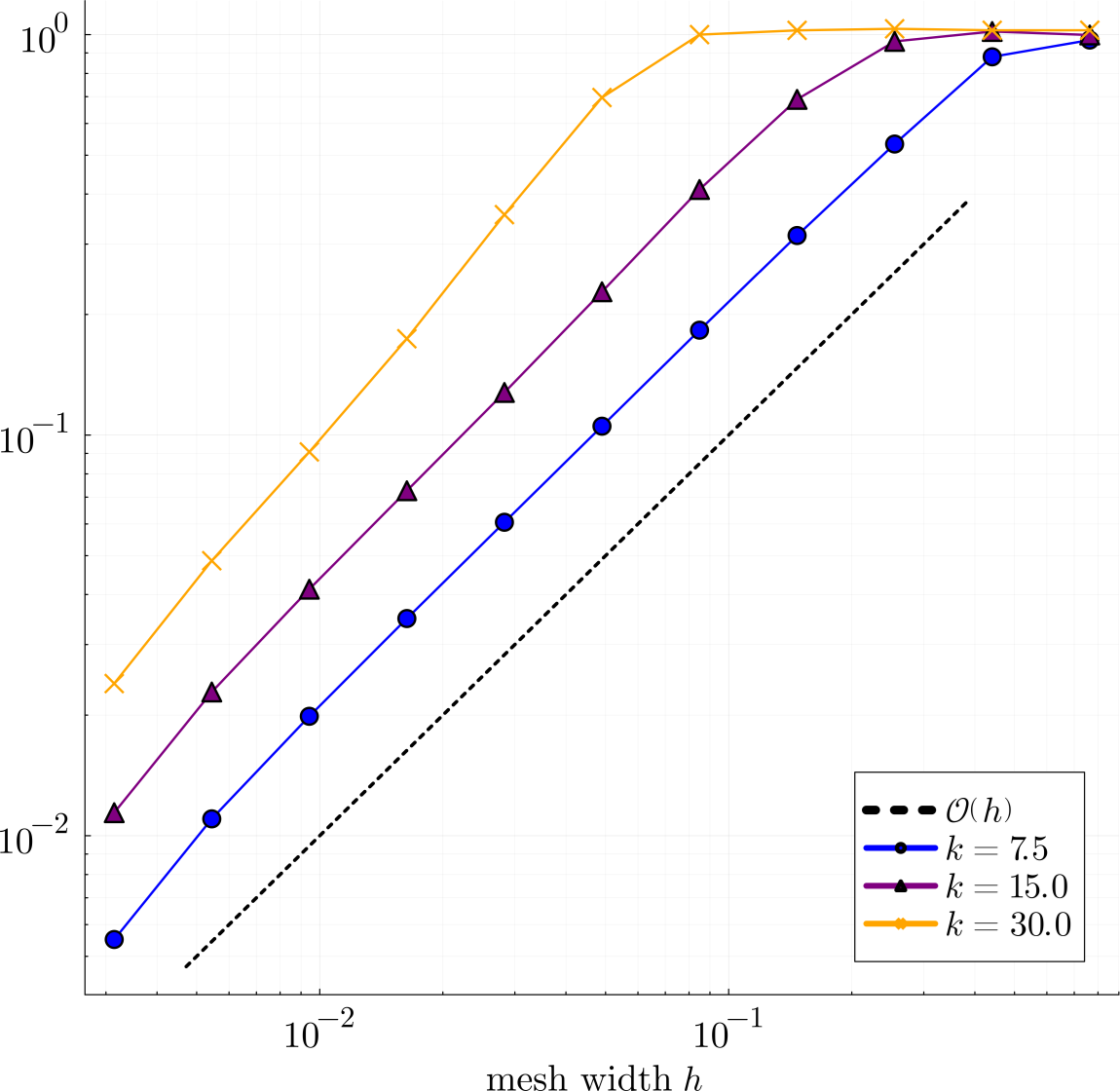}
\caption{Fudgeflake}
\end{subfigure}
\hspace{1mm} 
\begin{subfigure}[t]{0.31\textwidth}
\centering
\includegraphics[width=\textwidth]
{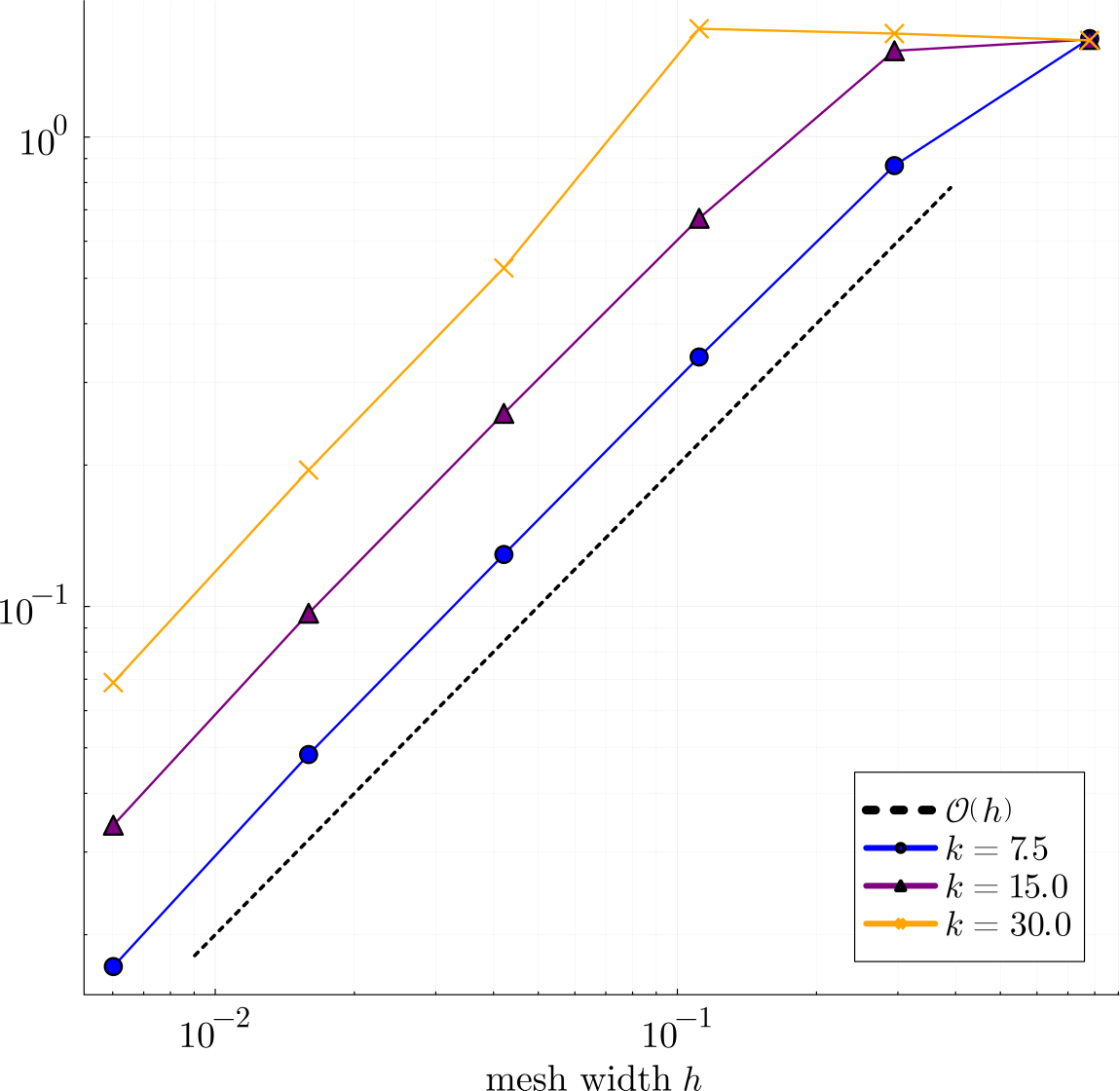}
\caption{Gosper Island}
\end{subfigure}
\hspace{1mm}
\begin{subfigure}[t]{0.31\textwidth}
\centering
\includegraphics[width=\textwidth]
{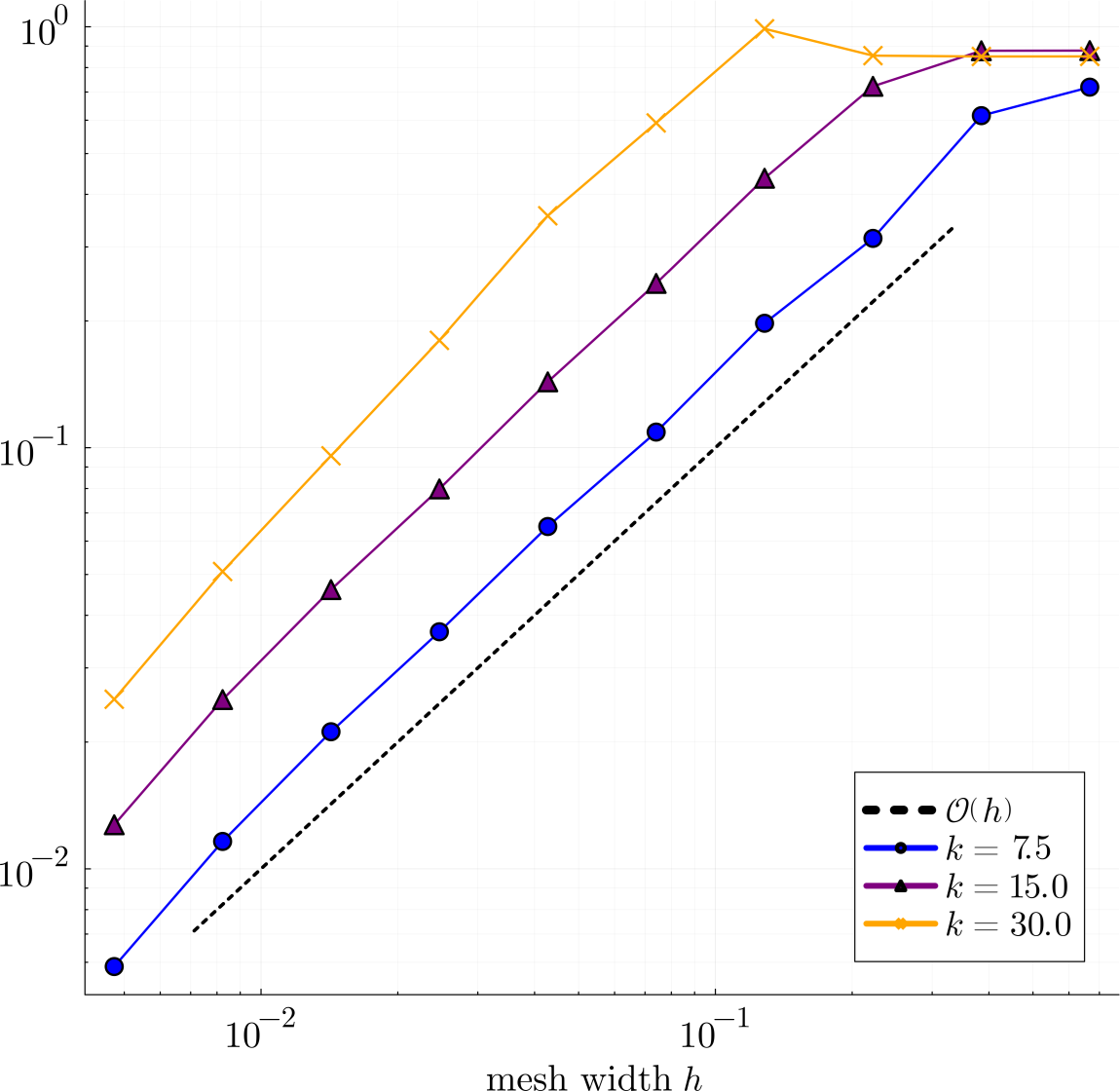}
\caption{Koch Snowflake}
\end{subfigure}
\caption{Convergence plots showing the error in $L^2(\Omega)$ of piecewise-constant solutions as a function of the meshwidth $h$. 
Physical parameters are as in Figure \ref{f:PWC}.
}
\label{f:ConvPWC}
\end{figure}

A corresponding convergence plot showing the error in $L^2(\Omega)$ of the VIM solutions as a function of $h$ is presented in Figure \ref{f:ConvPWC}. Here the reference solution is that obtained by the VIM with $h=h_l$ for and $l=12$, $7$ and $11$ respectively (for the three examples). The quadrature parameters were chosen in line with \eqref{e:FDQuadAss} with $\alpha=1$, specifically with 
\begin{align}
\label{e:QuadPar}
h_r =h, %
\quad
h_s=h_0, %
\quad
h_*=h,    %
\quad
h_g=h. %
\end{align}
This means that $1$-point quadrature is used for all the integrals computed, emphasizing the effectiveness of the singularity subtraction and self-similarity techniques used to handle the singular integrals. 
In all three cases we observe the $O(h)$ convergence rate predicted by Theorem \ref{t:FD}.

\begin{figure}[t!]
\begin{subfigure}[t]{0.3\textwidth}
\centering
\includegraphics[width=\textwidth]{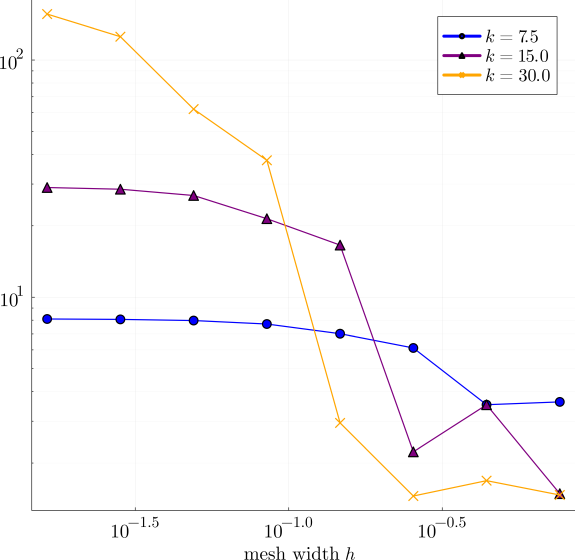}
\caption{Fudgeflake}
\end{subfigure}
\begin{subfigure}[t]{0.3\textwidth}
\centering
\includegraphics[width=\textwidth]{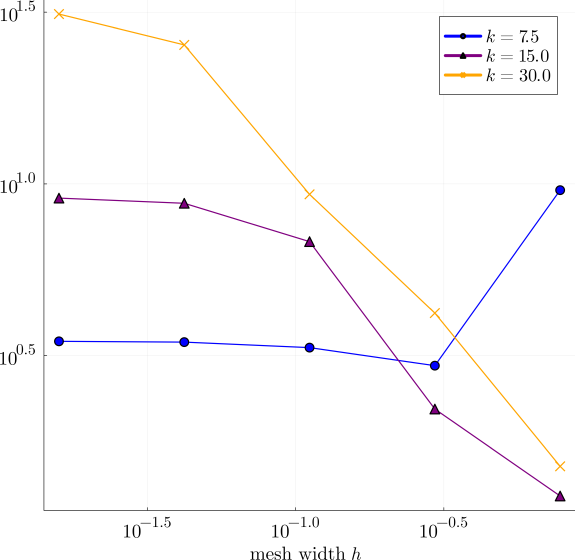}
\caption{Gosper Island}
\end{subfigure}
\begin{subfigure}[t]{0.3\textwidth}
\centering
\includegraphics[width=\textwidth]{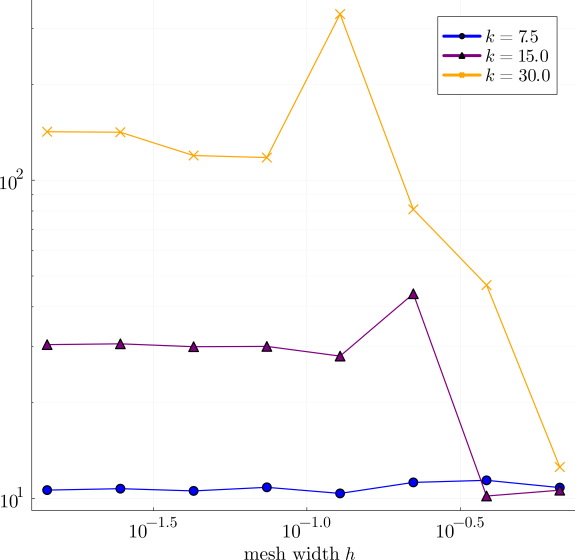}
\caption{Koch Snowflake}
\end{subfigure}
\caption{Condition numbers of the system matrices corresponding to the results in Figure \ref{f:ConvPWC}. 
}
\label{f:CondPWC}
\end{figure}

Plots of the associated condition numbers for the system matrices are presented in Figure \ref{f:CondPWC}.
In each case the discretisations remain well conditioned as $h\to 0$, as predicted by Theorem \ref{t:FD}, and as is expected for a second-kind integral equation formulation.

\begin{figure}[t!]
\centering
\begin{subfigure}[t]{0.3\textwidth}
\includegraphics[width=\textwidth]
{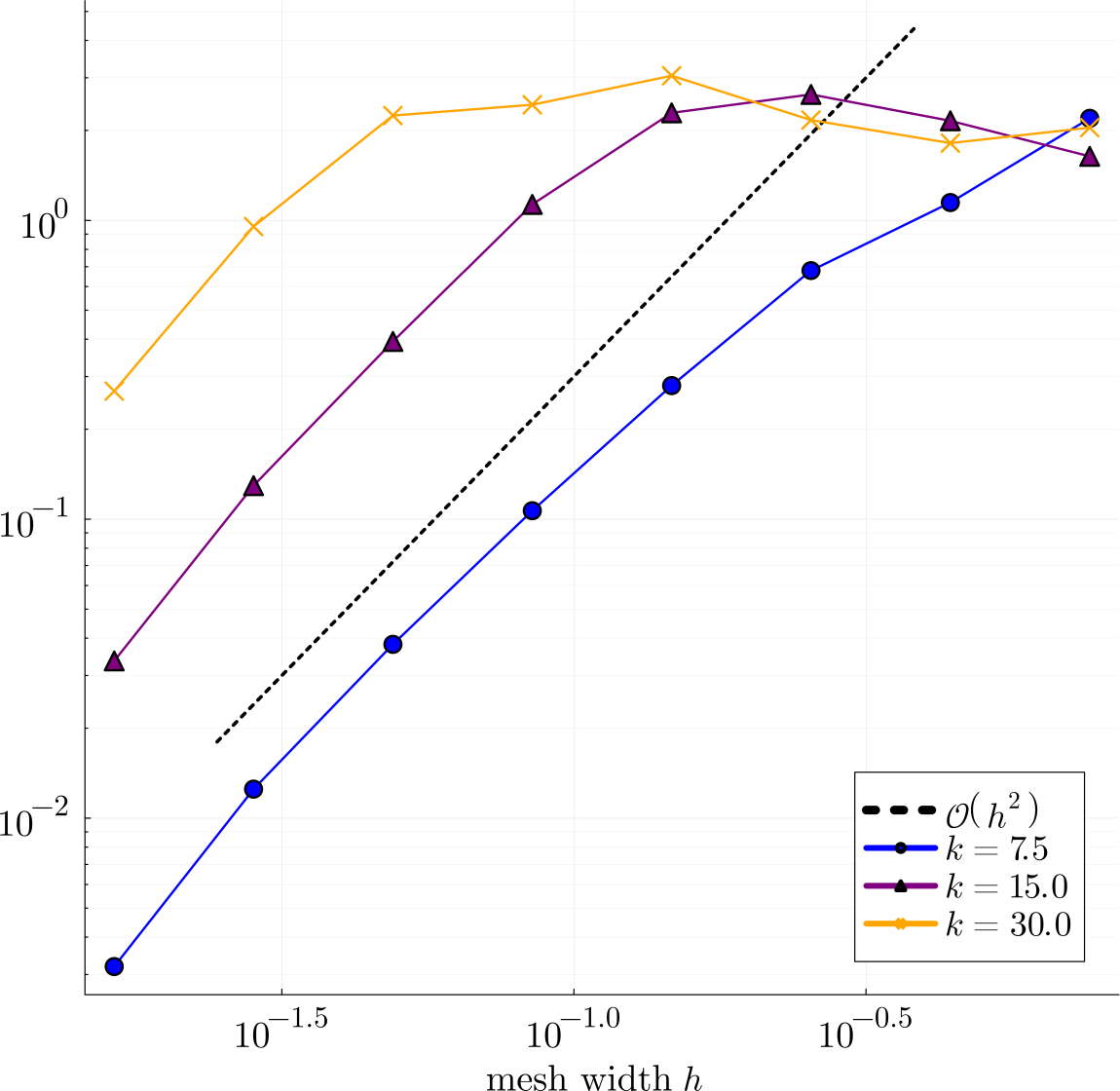}
\caption{Fudgeflake $R={h_0}$}
\end{subfigure}
\hspace{3mm}
\begin{subfigure}[t]{0.3\textwidth}
\includegraphics[width=\textwidth]
{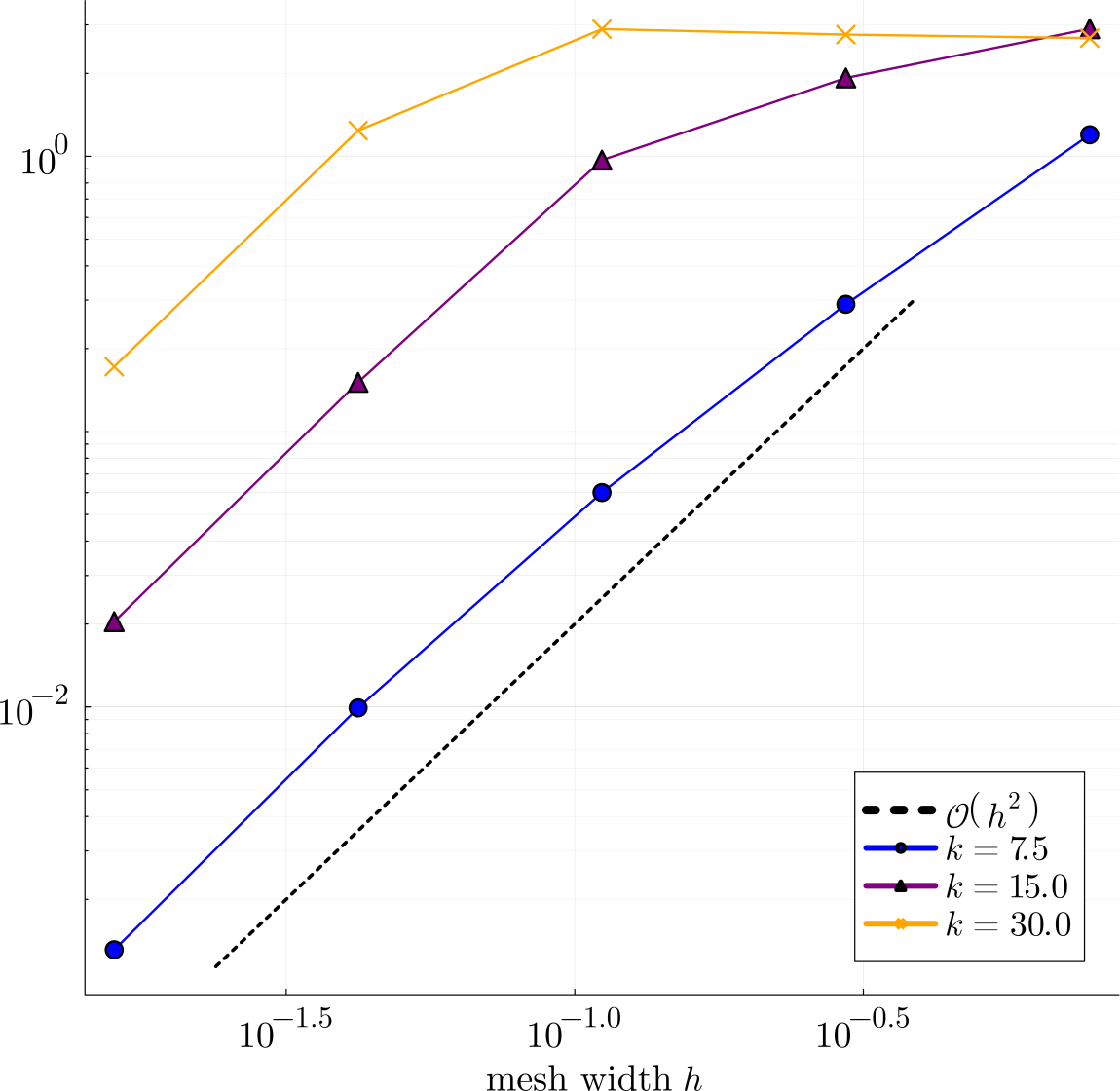}
\caption{Gosper Island $R={h_0}$}
\end{subfigure}
\hspace{3mm}
\begin{subfigure}[t]{0.3\textwidth}
\includegraphics[width=\textwidth]
{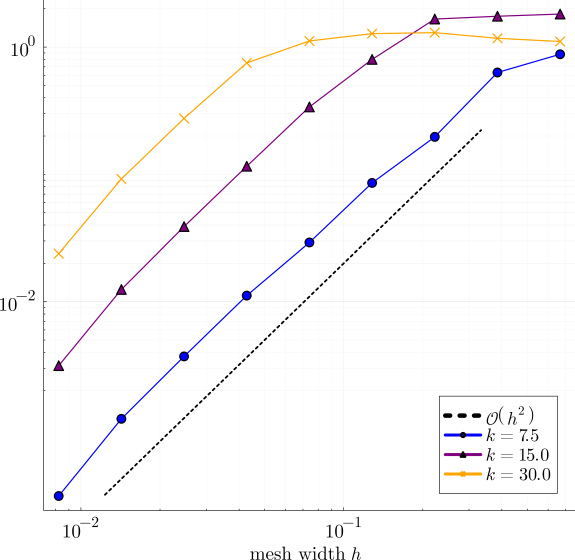}
\caption{Koch Snowflake $R={h_0}$}
\end{subfigure}

\begin{subfigure}[t]{0.3\textwidth}
\includegraphics[width=\textwidth]
{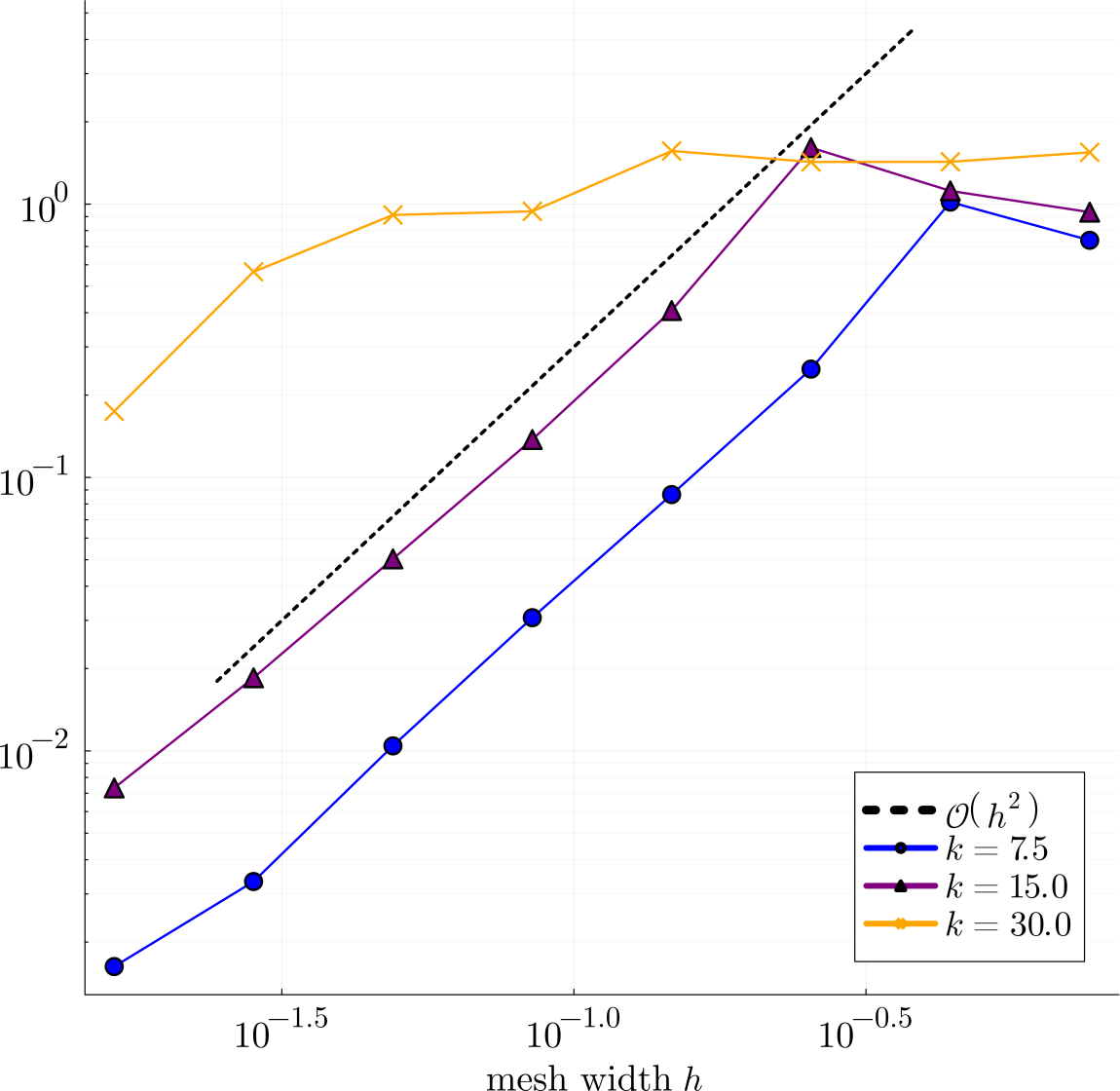}
\caption{Fudgeflake $R={h_0/10}$}
\end{subfigure}
\hspace{3mm}
\begin{subfigure}[t]{0.3\textwidth}
\includegraphics[width=\textwidth]
{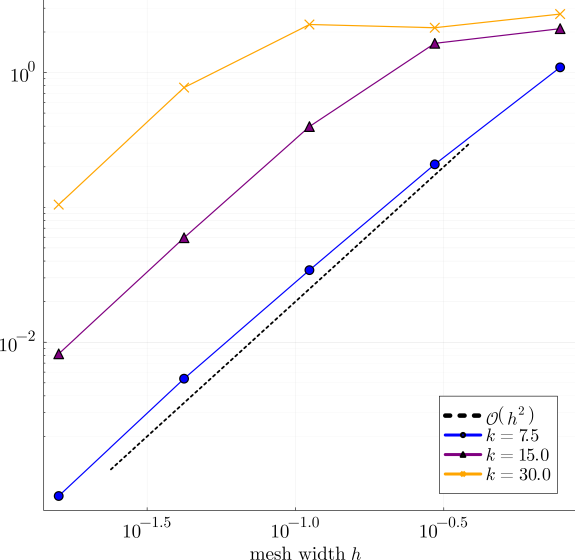}
\caption{Gosper Island $R={h_0/10}$}
\end{subfigure}
\hspace{3mm}
\begin{subfigure}[t]{0.3\textwidth}
\includegraphics[width=\textwidth]
{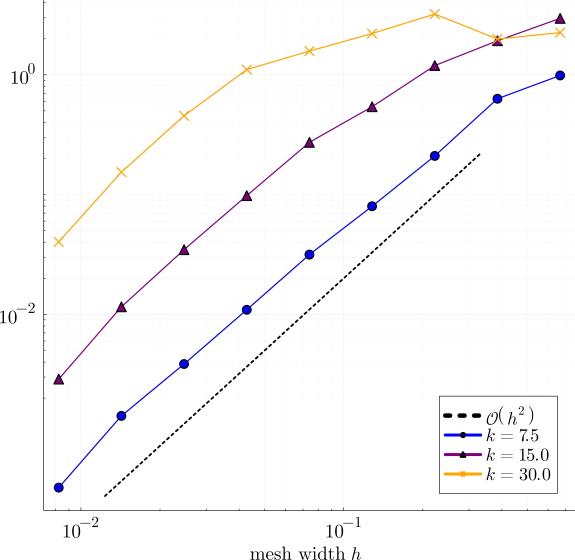}
\caption{Koch Snowflake $R={h_0/10}$}
\end{subfigure}
\caption{Convergence plots 
for the scattered field.
Top row - maximum error on a circle of radius {$R=h_0=\diam(\Omega)$} (outside the inhomogeneity). Bottom row - maximum error on a circle of radius {$R=h_0/10$} (inside the scatterer). The parameters $\vartheta$ and $m$ are as in Figure \ref{f:PWC}.
}
\label{f:ConvScat}
\end{figure}

In Figure \ref{f:ConvScat}(a)-(c) we present results showing the convergence of the scattered field on a circle of radius 
{$h_0=\diam(\Omega)$} 
surrounding the inhomogeneities. Here the quadrature parameters are chosen in accordance with \eqref{e:FDQuadAss} with $\alpha=2$, specifically with 
\begin{align}
\label{e:QuadPar2}
h_r =h^{3/2}, %
\quad
h_s=h^{1/2}, %
\quad
h_*=h,   %
\quad
h_g=h,   %
\quad 
h_J = h. 
\end{align}
Again this means $1$-point quadrature is used for the integrals involving $h_*$, $h_g$ and $h_J$, but not for the integrals involving $h_r$ and $h_s$.  
The plots show the maximum error over 2000 equally-spaced points on the circle, using as reference solution our VIM solution with $l=9$, $6$ and $10$ respectively (for the three examples). 
In all cases we see the $O(h^2)$ convergence rate predicted by Theorem \ref{t:FD}. 
In Figure \ref{f:ConvScat}(d)-(f) we show analogous results for convergence of the scattered field \textit{inside} the inhomogeneity, on a circle of radius 
{$h_0/10$. }
As discussed in Remark \ref{r:Inside}, in \S\ref{s:IFSquad} we have not presented quadrature rules tailored to the evaluation of the scattered field inside the inhomogeneity. However, as discussed in Remark \ref{r:SupConF}, our theoretical analysis proves semi-discrete convergence like $O(h^{2-\epsilon})$ for every $\epsilon>0$. So with sufficiently accurate quadrature for the evaluation of the scattered field we would expect in practice to see convergence approximately like $O(h^2)$. This is indeed what we observe in 
Figure \ref{f:ConvScat}(d)-(f), where for the computation of the scattered field functional we used $h_J=\xi h$, where $\xi = 1/\sqrt{3}$ for the Fudgeflake and the Koch Snowflake, and $\xi=1/\sqrt{7}$ for the Gosper Island (i.e.\ $h_J=h_{l+1}$ for each $l$,  rather than $h_J=h_l$).

\begin{figure}[t!]
\begin{subfigure}[t]{0.3\textwidth}
\centering
\includegraphics[width=\textwidth]
{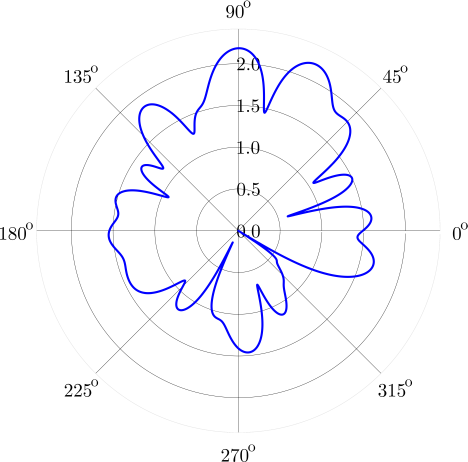}
\caption{Fudgeflake}
\end{subfigure}
\begin{subfigure}[t]{0.3\textwidth}
\centering
\includegraphics[width=\textwidth]
{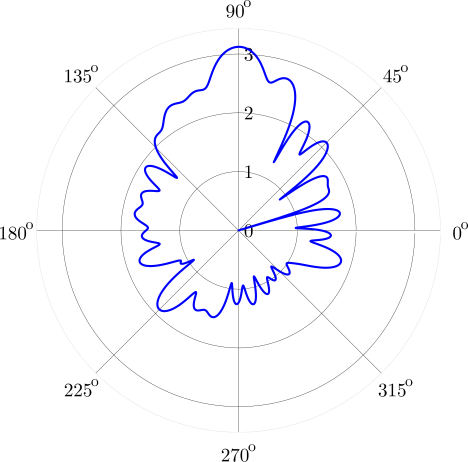}
\caption{Gosper Island}
\end{subfigure}
\begin{subfigure}[t]{0.3\textwidth}
\centering
\includegraphics[width=\textwidth]
{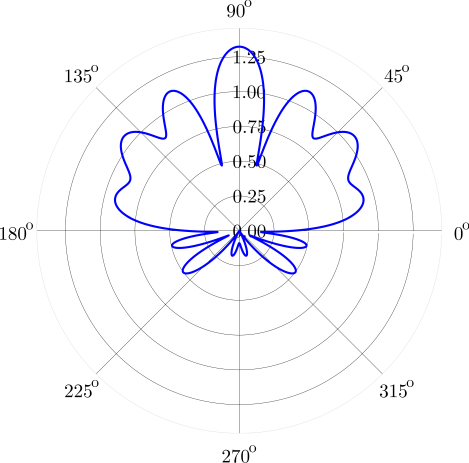}
\caption{Koch Snowflake}
\end{subfigure}
\caption{Polar plots of the far-field patterns for $k=15$ for the three scatterers, with other physical parameters as in Figure \ref{f:PWC}. In each case the radial coordinate shows $\log_{10}(|J^{\rm FFP}_{\theta}|/|J^{\rm FFP}_{\rm min}|)$, where $|J^{\rm FFP}_{\rm min}|$ is the minimum modulus of the far-field pattern computed on the 2000 equally-spaced $\theta$ values considered. 
}
\label{f:Polar}
\end{figure}

\begin{figure}[t!]
\begin{subfigure}[t]{0.3\textwidth}
\centering
\includegraphics[width=\textwidth]
{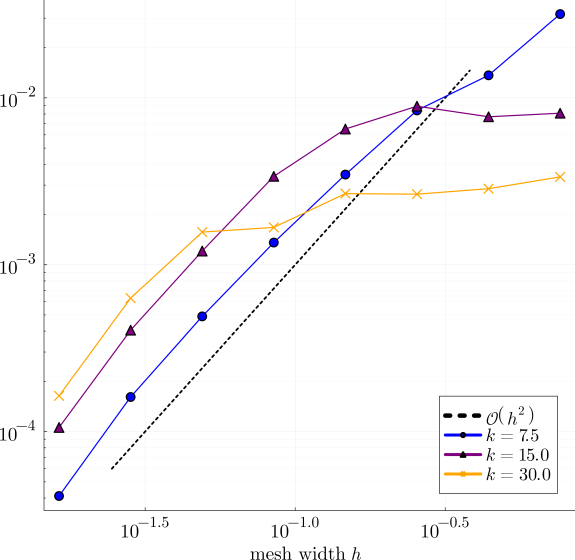}
\caption{Fudgeflake}
\end{subfigure}
\begin{subfigure}[t]{0.3\textwidth}
\centering
\includegraphics[width=\textwidth]
{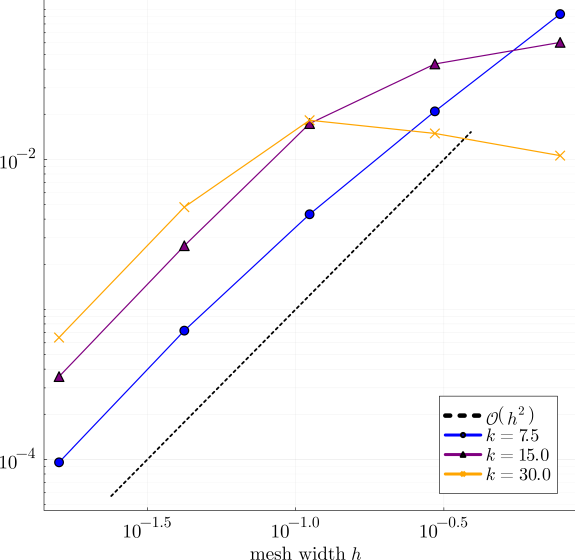}
\caption{Gosper Island}
\end{subfigure}
\begin{subfigure}[t]{0.3\textwidth}
\centering
\includegraphics[width=\textwidth]
{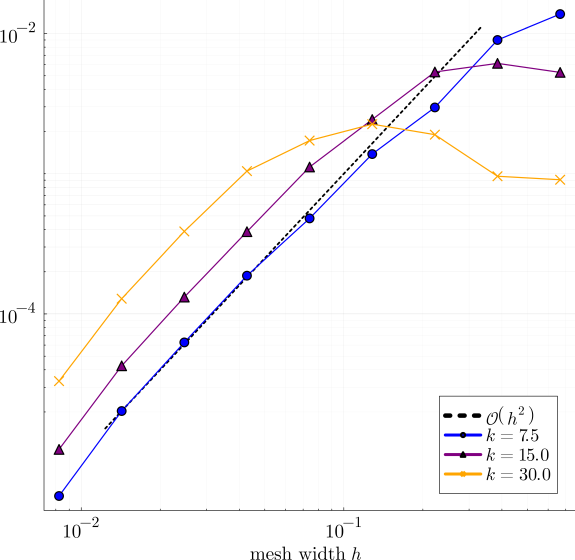}
\caption{Koch Snowflake}
\end{subfigure}
\caption{Convergence of the far-field pattern. Plots show the maximum error over $\theta\in[0,2\pi]$, as a function of the meshwidth $h$. 
Physical parameters $\vartheta$ and $m$ are as in Figure \ref{f:PWC}.
}
\label{f:FFPConv}
\end{figure}

In Figure \ref{f:Polar} we show polar plots of the far-field pattern for the three examples, for $k=15$, computed using the quadrature parameters \eqref{e:QuadPar2}. In Figure \ref{f:FFPConv} we show corresponding convergence plots, in which we again see the $O(h^2)$ convergence rate predicted by Theorem \ref{t:FD}. 
Here we are plotting the maximum error over 2000 equally-spaced directions, using the same reference solutions as for Figure \ref{f:ConvScat}.

\begin{figure}[t!]
\centering
\begin{subfigure}[T]{0.24\textwidth}
\includegraphics[width=.9\textwidth]{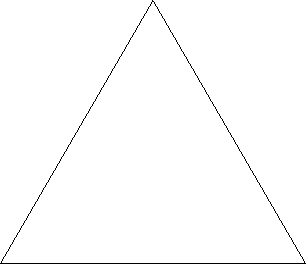}
\vskip .3\textwidth
\caption{$j=0$}
\end{subfigure}
\begin{subfigure}[T]{0.24\textwidth}
\includegraphics[width=.9\textwidth]
{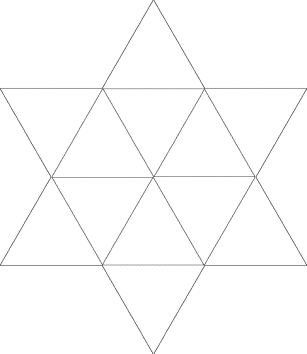}
\caption{$j=1$}
\end{subfigure}
\begin{subfigure}[T]{0.24\textwidth}
\includegraphics[width=.9\textwidth]
{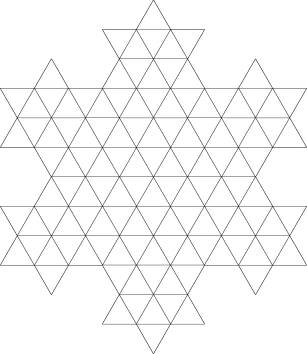}
\caption{$j=2$}
\end{subfigure}
\begin{subfigure}[T]{0.24\textwidth}
\includegraphics[width=.9\textwidth]
{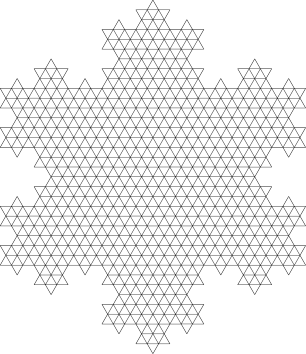}
\caption{$j=3$}
\end{subfigure}
\caption{Uniform triangular meshes for the prefractal-based method. 
}
\label{f:PFMesh}
\end{figure}

\begin{figure}[t!]
\centering
\begin{subfigure}[t]{0.3\textwidth}
\includegraphics[height=\textwidth]
{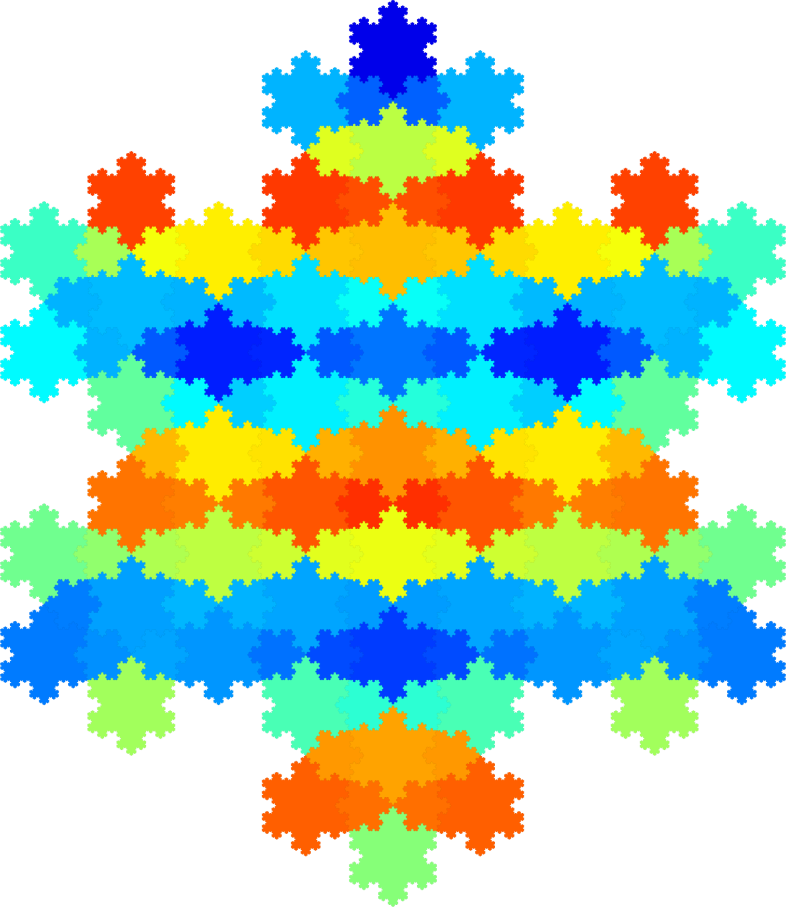}
\caption{$l=4$, $N=133$}
\end{subfigure}
\hspace{2mm}
\begin{subfigure}[t]{0.3\textwidth}
\includegraphics[height=\textwidth]
{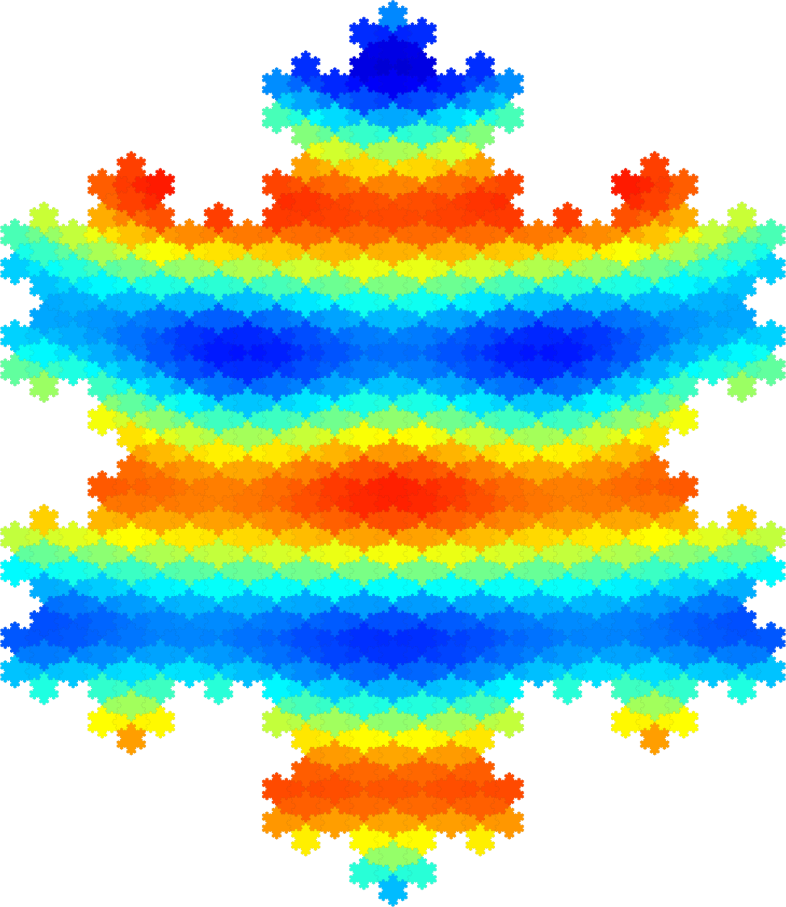}
\caption{$l=6$, $N=1261$}
\end{subfigure}
\hspace{2mm}
\begin{subfigure}[t]{0.3\textwidth}
\includegraphics[height=\textwidth]
{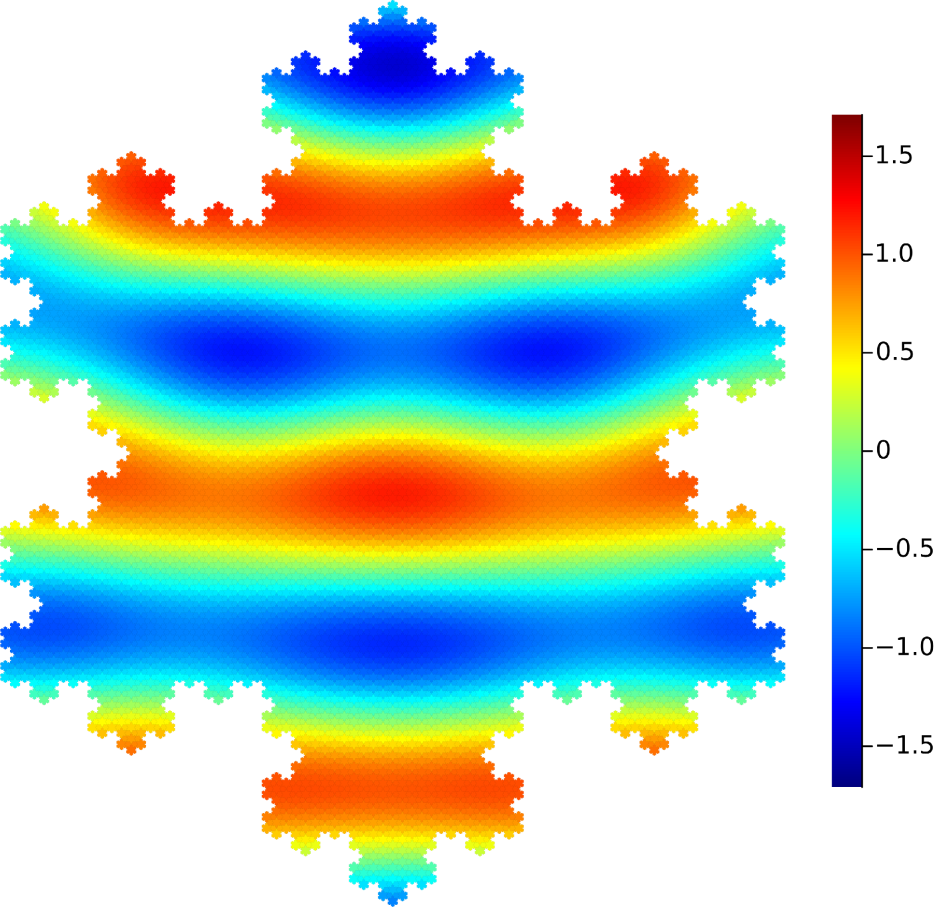}
\caption{$l=8$, $N=11605$}
\end{subfigure}

\begin{subfigure}[t]{0.3\textwidth}
\includegraphics[height=\textwidth]
{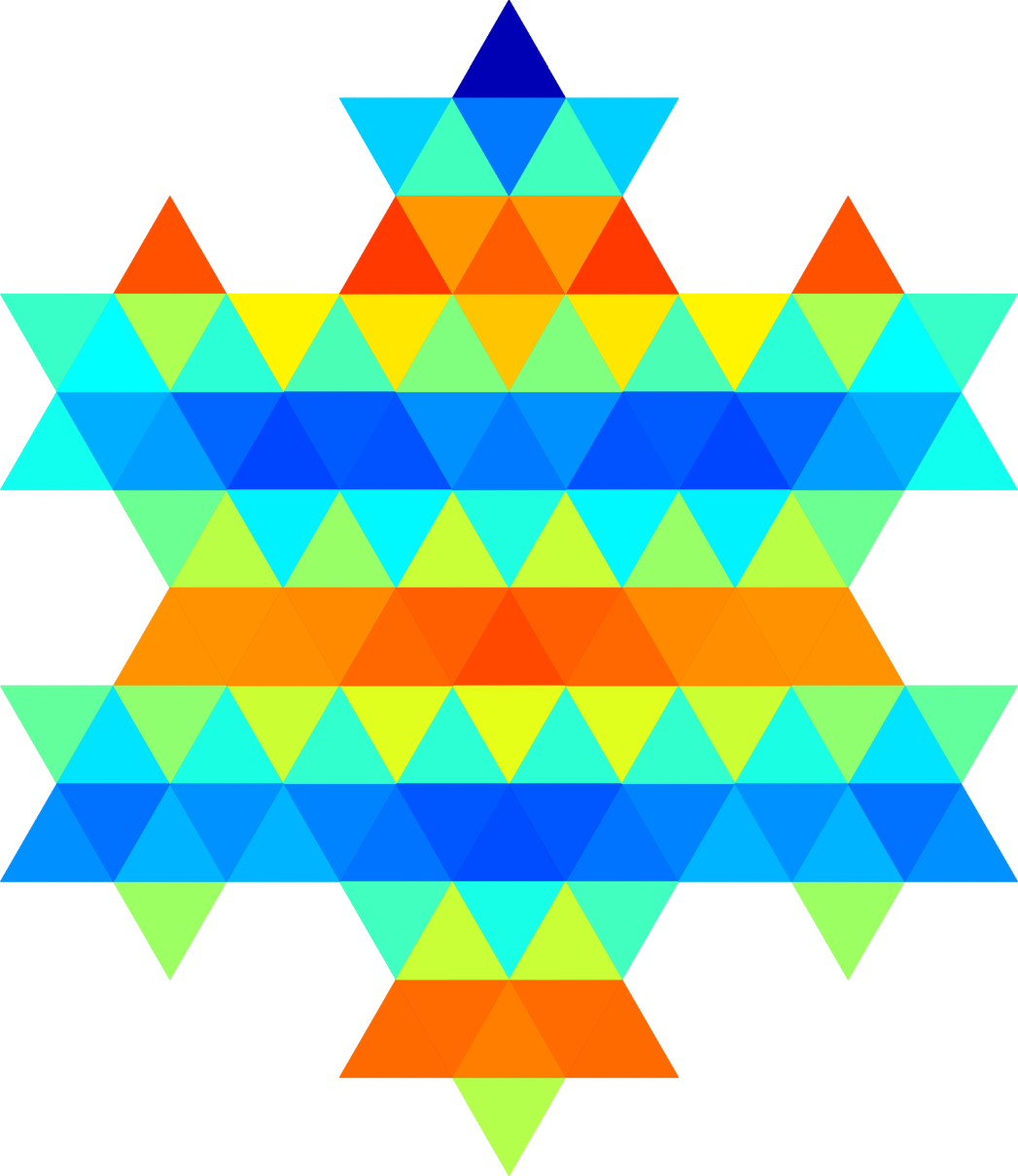}
\caption{$j=2$, $N=120$}
\end{subfigure}
\hspace{2mm}
\begin{subfigure}[t]{0.3\textwidth}
\includegraphics[height=\textwidth]
{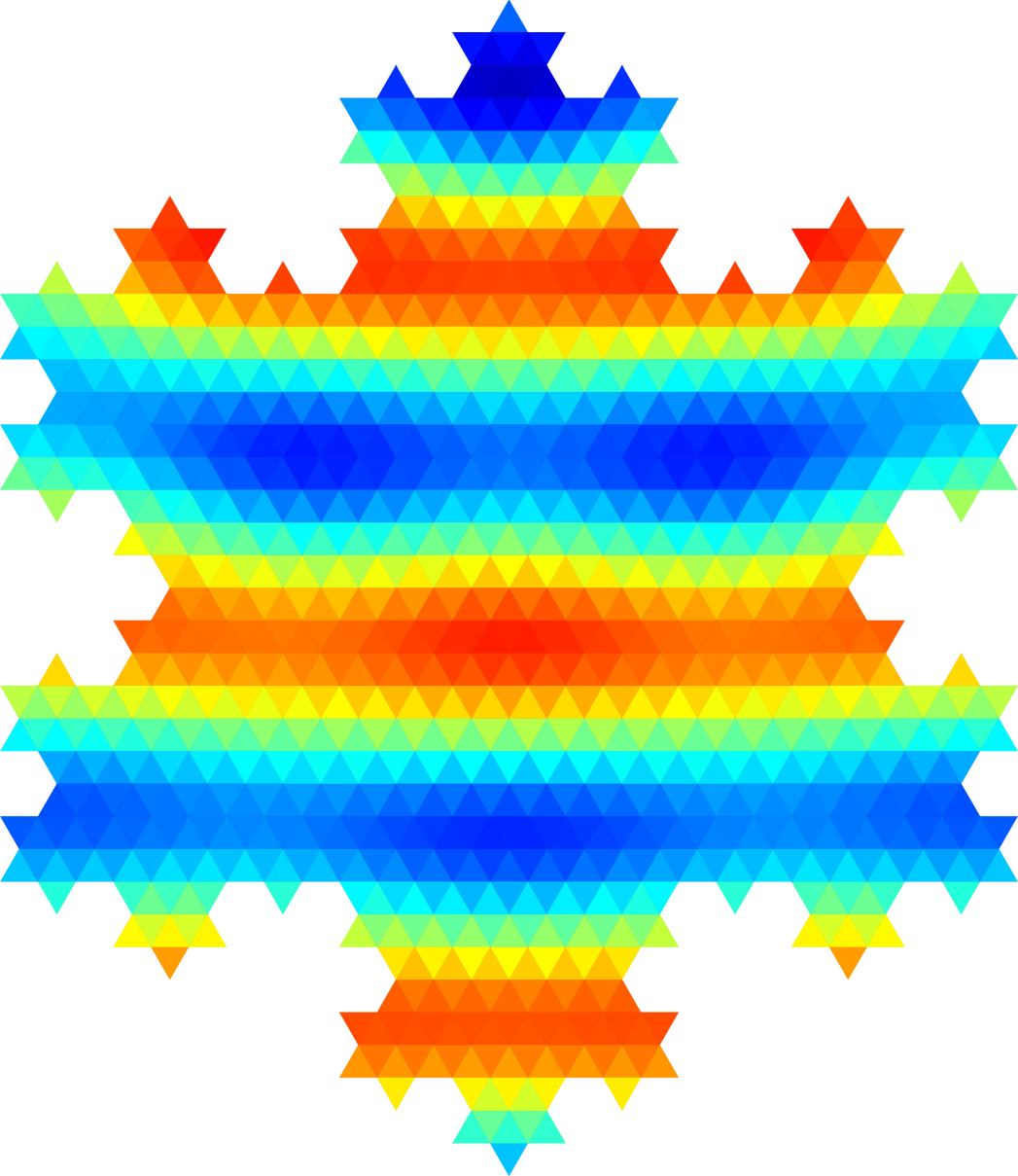}
\caption{$j=3$, $N=1128$}
\end{subfigure}
\hspace{2mm}
\begin{subfigure}[t]{0.3\textwidth}
\includegraphics[height=\textwidth]
{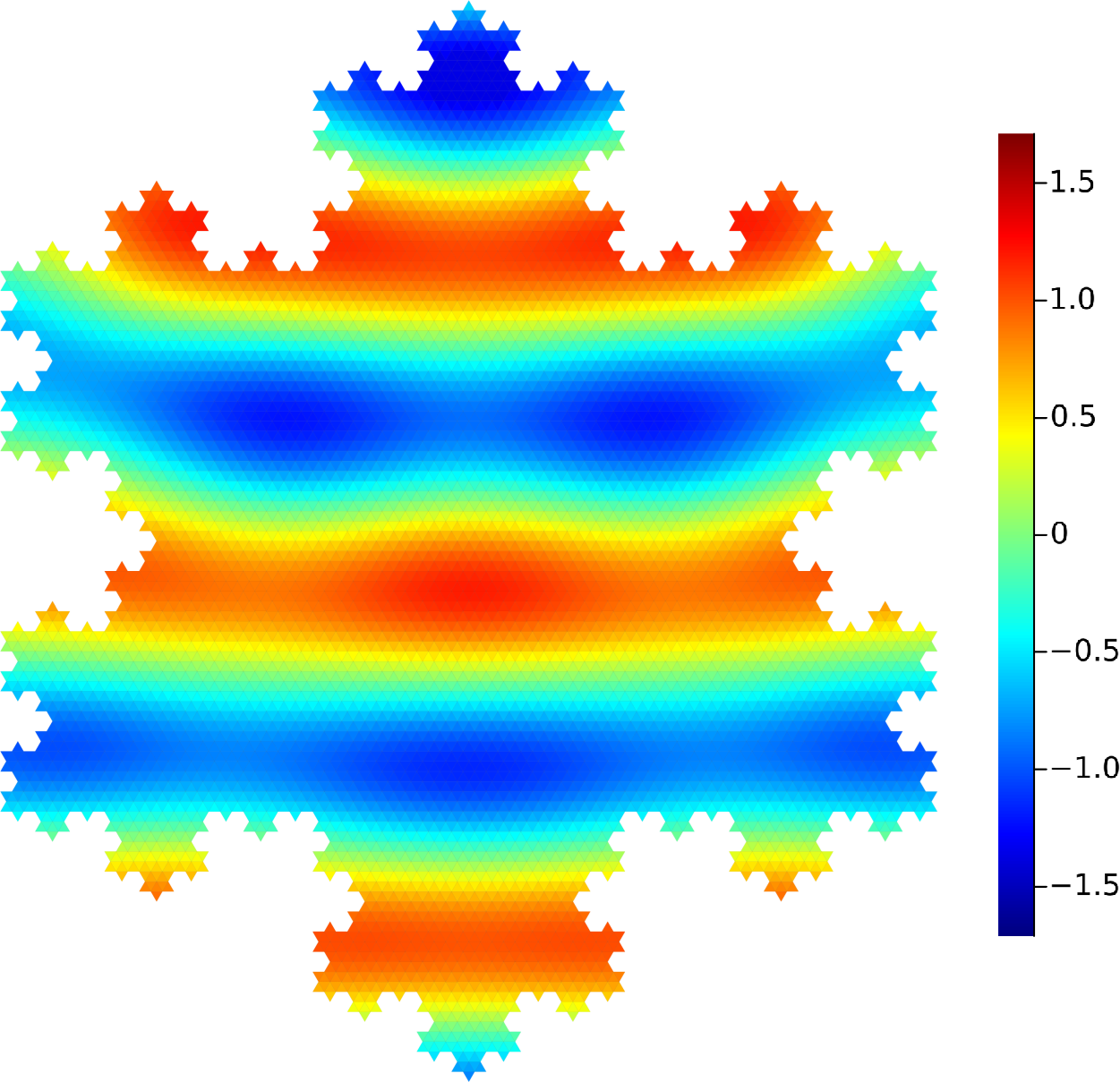}
\caption{$j=4$, $N=10344$}
\end{subfigure}
\caption{Piecewise-constant integral equation solution on the Koch Snowflake (top row) and on prefractal approximations (bottom row) for different values of $h$. 
{The number of mesh elements $N$ is reported in each case.}
}
\label{f:KochPWC}
\end{figure}

Finally, in the case of the Koch Snowflake we compare the performance of our IFS-based method to that of a prefractal-based method. To define the latter, we consider the standard sequence of prefractals for the Koch Snowflake, as illustrated in Figure \ref{f:KochPreFrac}. Each prefractal is meshed by a uniform mesh composed of elements that are equilateral triangles with a side length equal to the length of the line segments making up the boundary of the prefractal. The first 4 such meshes are shown in Figure \ref{f:PFMesh} (which should be compared to Figure \ref{f:Meshes}(c), which shows the first four IFS-based meshes), and general properties of these meshes are discussed in Appendix \ref{a:AppB}. On each such mesh we compute the piecewise constant Galerkin approximation to the solution $u$ of \eqref{e:IntEqn2a}, as the solution of \eqref{e:VIM} with $V_N$ replaced by the space of piecewise constant functions on the prefractal mesh. For fairness of comparison, quadrature is carried out using a similar strategy to that used for our IFS-based method, exploiting the fact that the equilateral triangle is itself a $2$-attractor {(for any $a\in\N$ it can be decomposed as a union of $a^2$ equilateral triangles of side length $1/a$ times the original side length)}. 
In Figure \ref{f:KochPWC} we show example integral equation solutions for the VIM and the prefractal method, for $k=15$, $\vartheta=(0,1)$ and $1+m= 1.311+2.289\times 10^{-9}\ri$. 

\begin{figure}[t!]
\centering
\includegraphics[width=.6\textwidth]%
{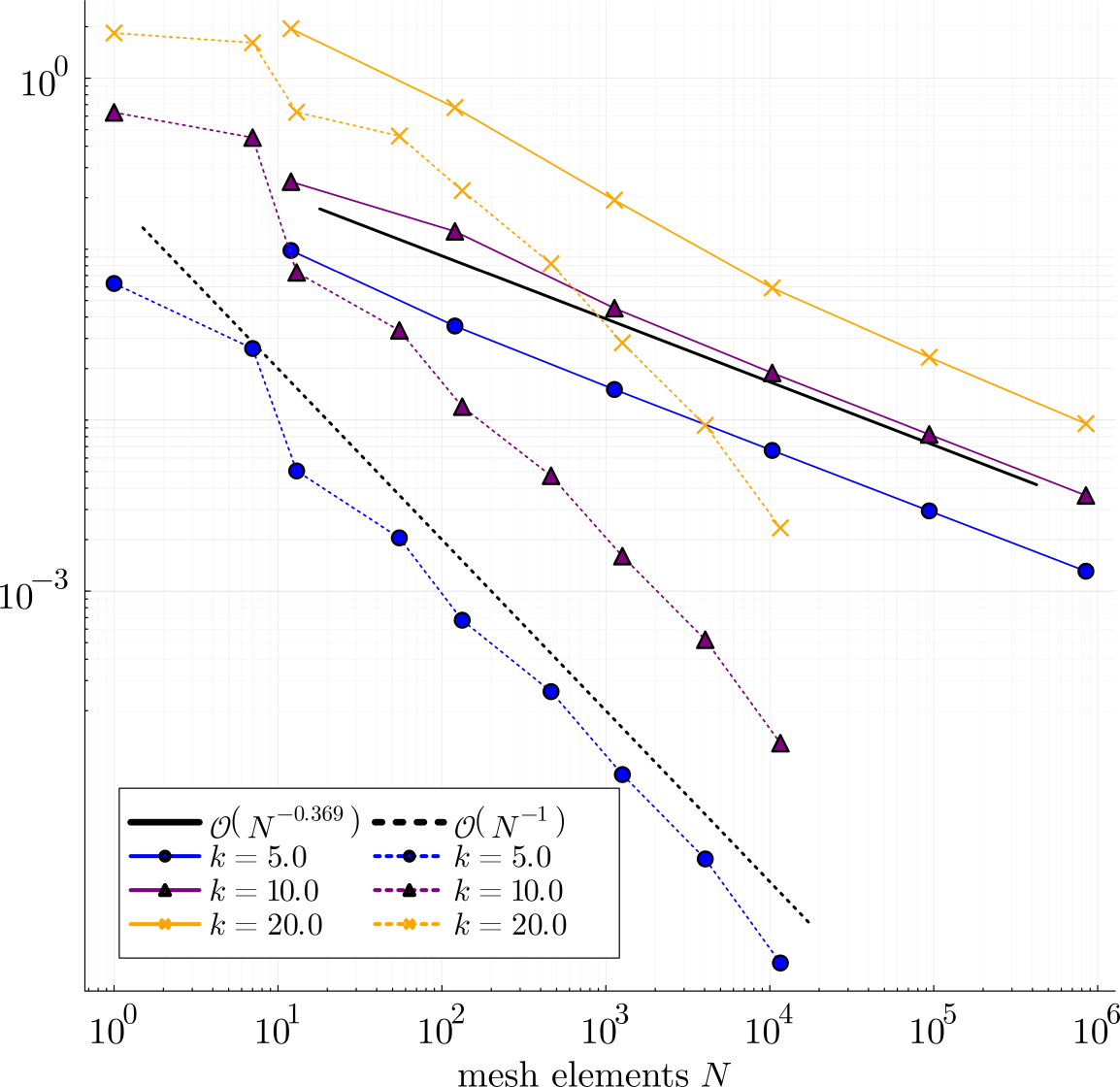}
\caption{
Convergence plot for the scattered field {showing the maximum error on a circle of radius $h_0=\diam(\Omega)$ (outside the inhomogeneity)} for the IFS-based approach (dotted lines) and the prefractal approach (solid lines). 
}
\label{f:KochComp}
\end{figure}

In Figure \ref{f:KochComp} we show corresponding error plots, comparing the accuracy of the IFS-based and prefractal-based methods, for the calculation of the scattered field on a circle of radius $4$ surrounding $\Omega$, as a function of the total number of degrees of freedom (mesh elements) used. As the reference solution we use our IFS-based method with meshwidth $h_{10}=3^{-5}h_0$. For our IFS-based method, by combining Theorem \ref{t:FD} with \eqref{e:NlSim}, we know that the error in the scattered field decays like $O(N_l^{-1})$ as $l\to\infty$, where $N_l=\# L_h(\Omega)$ is the total number of elements in the $L_{h_l}$ mesh of $\Omega$. For the prefractal-based method, a semi-discrete error analysis is provided in Appendix \ref{a:AppB}, and by combining Proposition \ref{p:Prefract} with \eqref{e:NjSim} we expect that the error in the scattered field should decay much more slowly, namely like $O(N_j^{-(1-d/2)})=O(N_j^{-0.369\ldots})$ as $j\to\infty$, where $N_j$ is the total number of elements in the $j$th prefractal mesh. These theoretical predictions are validated by our numerical results in Figure \ref{f:KochComp}, demonstrating that, for this problem, our geometry-conforming IFS-based approach provides significantly more accurate solutions than the prefractal-based approach.

\section*{Acknowledgements}
The authors thank Simon Chandler-Wilde, Ivan Graham, Andrea Moiola, Valery Smyshlyaev, Euan Spence and Timo Betcke for helpful discussions. DH acknowledges support from EPSRC grants EP/S01375X/1 and EP/V053868/1. Part of this work was carried out while JB and DH were in residence at Institut Mittag-Leffler in Djursholm, Sweden in autumn 2025, supported by the Swedish Research Council under grant no.\ 2021-06594.



\appendix

\section{Analysis of the prefractal method}
\label{a:AppB}

Let $\Omega$ denote the interior of the Koch Snowflake, and for $j\in \N_0$ let $\Omega_j$ denote the interior of the $j$th standard interior prefractal for the Koch Snowflake, as defined in e.g.\ \cite[\S5.1]{caetano2021density} (with $\beta=\pi/6$ in the notation of \cite{caetano2021density}) and as illustrated in Figure \ref{f:KochPreFrac}. As in \S\ref{s:Numer} we scale the Snowflake so it has diameter $h_0=2/\sqrt{3}$. The length of the sides of $\Omega_j$ is then $\tilde{h}_j:=3^{-j}$. Let $\tilde{V}_j$ denote the space of piecewise-constant functions on a uniform mesh of $\Omega_j$ composed of equilateral triangular elements of side length $\tilde{h}_j$, as illustrated for the case $j=2$ in Figure \ref{f:Disc}(b). The number $N_j$ of elements in the resulting mesh of $\Omega_j$ satisfies the recurrence $N_{j+1}=9N_j+3.4^j$ with $N_0=1$, so that $N_j=(8.9^j-3.4^j)/5$ for $j\in \N_0$, and hence
\begin{align}
 \label{e:NjSim}
 N_j\sim \frac{8}{5}9^j = \frac{8}{5}\tilde{h}_j^{-2}, \quad j\to\infty.
 \end{align} 
Let $\tilde{u}_j$ denote the semi-discrete prefractal Galerkin approximation to the solution $u$ of \eqref{e:IntEqn2a}, i.e.\ the solution of \eqref{e:VIM} with $V_N$ replaced by $\tilde{V}_j$. We can then prove the following convergence result. 

\begin{proposition}
\label{p:Prefract}
Let $\Omega$ denote the interior of the Koch Snowflake and let $\m=m\chi_\Omega$ for some $m\in \C\setminus\{0\}$ with $\Im{m}\geq 0$.  Suppose that the functions $g$ and $\jj$ in \eqref{e:IntEqn2a} and \eqref{e:JDef} satisfy $g,\jj \in H^s(\Omega)$ for some $s>1$. Then, with $\tilde{V}_j$ as above, for sufficiently large $j$ the {piecewise-constant} prefractal Galerkin approximation $\tilde{u}_j\in \tilde{V}_j$ exists and satisfies
\begin{align}
 \label{e:PFu}
 \|u-\tilde{u}_j\|_{L^2(\Omega)}&=O(h_j^{1-d/2}), \qquad j\to\infty,\\
 |J(u)-J(\tilde{u}_j)|&=O(h_j^{2-d}), \qquad j\to\infty,
 \label{e:PFJ}
 \end{align}  
where $d=\log{4}/\log{3}\approx 1.26$ is the Hausdorff dimension of $\partial\Omega$. 
\begin{proof}
We first show that the spaces $\tilde{V}_j$ form a dense sequence in $L^2(\Omega)$. This holds by a standard triangle inequality argument. 
We begin by noting that,
by applying the standard Poincar\'e inequality on convex domains \cite{payne1960optimal,bebendorf2003note} (of which the equilateral triangle is an example) on each of the elements of the mesh of $\Omega_j$, and summing over all these elements, 
we have for each $j\in \N_0$ that
\begin{align}
\label{e:BestAppP2}
\inf_{v\in \tilde{V}_j}\|\phi - v\|_{L^2(\Omega_j)}
\leq \frac{{\tilde{h}}_j}{\pi} \|\phi\|_{W^1(\Omega_j)} 
\leq \frac{{\tilde{h}}_j}{\pi} \|\phi\|_{W^1(\Omega)},
\qquad \phi\in W^1(\Omega).
\end{align}
Now, given $\psi\in L^2(\Omega)$ and $\eps>0$ we can find $\phi\in C^\infty_{\rm comp}(\Omega)$ such that $\|\psi-\phi\|_{L^2(\Omega)}<\eps/2$. Since $\supp\phi$ is a compact subset of $\Omega$ and the prefractals $\{\Omega_j\}_{j\in\N_0}$ form a nested increasing sequence, 
we have that, for
sufficiently large $j$, 
$\phi\in C^\infty_{\rm comp}(\Omega_j)$ and hence 
$\inf_{v\in \tilde{V}_j}\|\phi - v\|_{L^2(\Omega)} 
= \inf_{v\in \tilde{V}_j}\|\phi - v\|_{L^2(\Omega_j)}$, which by \eqref{e:BestAppP2} 
can be made smaller than $\eps/2$ by taking $j$ sufficiently large. The claimed density result then follows by the triangle inequality. 

Asymptotic well-posedness and quasi-optimality of the prefractal Galerkin approximation then hold by the same standard argument used in the proof of Theorem \ref{t:GalWP}. 
Without loss of generality we can assume that $1<s<2+(2-d)/2=3-d/2$. (If not, just redefine $s$). This constraint, combined with the assumption that $g\in H^s(\Omega)$,  ensures by Corollary \ref{c:App3} 
(specifically, by \eqref{e:MP5b}) 
that the exact solution satisfies $u\in H^s(\Omega)$. Hence to prove \eqref{e:PFu} it suffices to show that 
\begin{align}
 \label{e:PFBA}
 \inf_{v\in \tilde{V}_j}\|\phi-v\|_{L^2(\Omega)}&=O({\tilde{h}}_j^{1-d/2}), \qquad j\to\infty, \,\,\phi\in H^s(\Omega).
 \end{align}
Furthermore, if \eqref{e:PFBA} holds then \eqref{e:PFJ} also follows, by essentially the same superconvergence argument that was used to prove Theorem \ref{t:SupCon}, because if $\jj\in H^s(\Omega)$ then (again assuming without loss of generality that $1<s<3-d/2$) by Corollary \ref{c:App3} 
(specifically, by \eqref{e:MP7b}) the dual solution satisfies $\zeta\in H^s(\Omega)$. To see \eqref{e:PFBA}, first note that 
\begin{align}
 \label{e:PFBA2}
 \inf_{v\in \tilde{V}_j}\|\phi-v\|_{L^2(\Omega)}&\leq\|\phi\|_{L^2(\Omega\setminus \Omega_j)} + \inf_{v\in \tilde{V}_j}\|\phi-v\|_{L^2(\Omega_j)}.
 \end{align}
If $\phi\in H^s(\Omega)$ then since $s>1$ we have $\phi\in W^1(\Omega)$, so that by \eqref{e:BestAppP2} the second term in \eqref{e:PFBA2} is $O(\tilde{h}_j)$ as $j\to\infty$. Furthermore, again since $s>1$, $\phi$ is bounded on $\Omega$ by the Sobolev embedding theorem and hence the first term in \eqref{e:PFBA2} can be bounded by 
\begin{align}
 \label{e:PFBA3}\|\phi\|_{L^2(\Omega\setminus \Omega_j)} \leq \|\phi\|_{L^\infty(\Omega)}|\Omega\setminus \Omega_j|^{1/2}.
  \end{align} 
By direct geometrical calculations (cf.\ \cite[p.80]{bannisterthesis}) one can show that 
\begin{align}
\label{e:GapArea}
|\Omega\setminus \Omega_j| = O((4/9)^j) = O((\tilde{h}_j)^{2-d}),  \quad j\to\infty,
\end{align}
where we used the fact that $h_j=3^{-j}$ and $d=\log{4}/\log{3}$ in the second equality. 
Combined with \eqref{e:PFBA3}, this implies that the first term in \eqref{e:PFBA2} is $O(\tilde{h}_j^{1-d/2}$), from which \eqref{e:PFBA} follows. 
 
\end{proof}

\end{proposition}

\section{Geometrical lemmas}
\label{a:AppA}
The following elementary lemma is a 
sharpening of 
\cite[Lem.\ 9.2]{falconer2014fractal}). 
Specifically, \cite[Lem.\ 9.2]{falconer2014fractal} assumes that the sets $\tilde{E}_i$ in the statement of Lemma \ref{l:Falc} are balls, which constitutes a ``shape-regularity'' assumption on the collection $\{E_i\}$. Our more general result shows this is not necessary.
\begin{lemma}
\label{l:Falc}
Let $r,a,\tau>0$. Let $\{E_i\}$ be a collection of subsets of $\R^n$ such that, for each $i$, $\diam(E_i)\leq ar$ and $E_i$ contains a measurable subset $\tilde{E}_i$ such that $|\tilde{E}_i|\geq \tau  r^n$, with $\{\tilde{E}_i\}$ being pairwise disjoint. Then 
any closed ball of radius $r$ intersects at most $(1+a)^n\omega_n\tau^{-1}$ of the closures $\{\overline{E_i}\}$, where $\omega_n$ denotes the volume of the unit ball in $\R^n$.
\end{lemma}

\begin{proof}
We follow the proof of \cite[Lem 9.2]{falconer2014fractal}. If $\overline{E_i}$ intersects $B$ then, since $\diam(E_i)\leq ar$, $\overline{E_i}$ is contained in the closed ball of radius $(1+a)r$ concentric with $B$. Suppose that $q$ of the sets $\{\overline{E_i}\}$ intersect $B$. Summing the volumes of the corresponding subsets $\tilde{E}_i$, it follows that $q\tau r^n \leq (1+a)^nr^n\omega_n$,
	giving the stated bound for $q$.
\end{proof}

The next lemma is a variant of the above, for intersections with annuli.
\begin{lemma}
\label{l:Falc2}
If $\{E_i\}$ satisfy the hypothesis of Lemma \ref{l:Falc} with $a=1$, then for any $x_0\in\R^n$ and any $j\in\N$, $j\geq 2$, the annulus $A_j:=\{x\in\R^n: jr\leq |x-x_0|\leq (j+1)r\}$ intersects at most $3.2^n j^{n-1}\omega_n\tau^{-1}$ of the closures $\{\overline{E_i}\}$.
\end{lemma}
\begin{proof}
We adopt a similar argument to that used to prove 
Lemma \ref{l:Falc}. %
If $\overline{E_i}$ intersects $A_j$ then by the assumption that $\diam(E_j)\leq r$ we have that $\overline{E_i}$ is contained in 
$\{x\in\R^n: (j-1)r\leq |x-x_0|\leq (j+2)r\}$. 
Hence if $q$ of the $\{\overline{E_i}\}$ intersect $A_j$ we have by comparing volumes that $q\tau \leq ((j+2)^n-(j-1)^n)\omega_n$. We then note that 
\[ (j+2)^n-(j-1)^n = 3((j+2)^{n-1} + (j+2)^{n-2}(j-1) + \ldots + (j+2)(j-1)^{n-2} + (j-1)^{n-1}),\]
which for $j\geq 2$ can be bounded above by
\[ 3(2^{n-1} +2^{n-2} + \ldots + 1)j^{n-1} \leq 3.2^nj^{n-1}.\]
\end{proof}

The next lemma, taken from \cite{vass2013explicit}, provides a guaranteed bounding ball for IFS attractors.
\begin{lemma}[{\cite[Theorem 2.1]{vass2013explicit}}]
\label{l:BoundBall}
Let $\{s_1,\ldots,s_M\}$ be an IFS as in Definition \ref{d:IFS} with attractor $K$.  
Given $x\in \R^n$, $K$ is contained in the closed ball centred at $x$ with radius 
\begin{align}
\label{}
r(x) \coloneqq \frac{\mu_{*}\varrho(x)}{1-\rho_{*}},
\end{align}
where 
\begin{align}
\label{}
\rho_{*} &\coloneqq \max_{i=1,\ldots,M} \rho_{i},\\
\mu_{*} &\coloneqq  \max_{i=1,\ldots,M} \|I-\rho_{i}A_{i}\|_{2},\\
\varrho(x) &\coloneqq \max_{i=1,\ldots,M} \| p_i - x \|_{2}, \quad  x\in \R^n, 
\end{align}
and, for $i=1,\ldots,M$, $p_i \in \R^n$ denotes the unique fixed point of the map $s_{i}$. 
\end{lemma}

\bibliographystyle{siam}
\bibliography{asbssfi}

\end{document}